\documentclass[12pt,letterpaper]{article}

\usepackage[T1]{fontenc}
\usepackage[utf8]{inputenc}
\usepackage[margin=1in]{geometry}
\usepackage{amsmath,amssymb,amsthm}
\usepackage{newtxtext,newtxmath}
\usepackage{bm,endnotes}
\usepackage{microtype}
\usepackage{graphicx,array}
\graphicspath{{experiment2/figures/monthly_returns_L36_net10bps_kl_labels_v2/}}

\usepackage{mathtools}       
\usepackage{mathrsfs}        
\usepackage{dsfont}          
\usepackage{xcolor}
\usepackage{enumitem}
\usepackage{booktabs}
\usepackage{subcaption}      
\usepackage{algorithm}
\usepackage{algpseudocode}   
\usepackage{wasysym}
\usepackage{natbib}
\bibpunct[, ]{(}{)}{,}{a}{}{,}

\usepackage{hyperref}
\hypersetup{
  colorlinks=true,
  linkcolor=teal,
  citecolor=teal,
  urlcolor=teal
}

\makeatletter
\g@addto@macro{\UrlBreaks}{\UrlOrds}
\makeatother

\theoremstyle{plain}
\newtheorem{theorem}{Theorem}
\newtheorem{lemma}{Lemma}
\newtheorem{proposition}{Proposition}

\newtheorem{assumption}{Assumption}
\theoremstyle{definition}
\newtheorem{definition}{Definition}
\newtheorem{example}{Example}
\theoremstyle{remark}
\newtheorem{remark}{Remark}

\def\P{\mathbb{P}}

\def\E{\mathbb{E}}
\def\R{\mathbb{R}}

\def\S{\mathbb{S}}

\def\F{\mathrm{F}}
\def\PSD{{\S^p_{+}}}
\def\PD{{\S^p_{++}}}

\def\prob{{\mathbb{P}}}

\def\diff{\mathrm{d}}

\def\transpose{\top}
\def\la{\langle}
\def\ra{\rangle}

\def\diag{\mathrm{diag}}
\def\tr{\operatorname{tr}}
\def\asto{\xrightarrow{\text{a.s.}}}
\def\diseq{\overset{d}{=}}

\let\argmin\relax

\DeclareMathOperator*{\argmin}{arg\,min}

\newcommand{\wh}{\widehat}

\newcommand{\Let}{\triangleq}

\newcommand{\bgeqn}{\begin{eqnarray}}
\newcommand{\edeqn}{\end{eqnarray}}
\newcommand{\bgeq}{\begin{eqnarray*}}
\newcommand{\edeq}{\end{eqnarray*}}

\newcommand{\setd}{{d\kern-.15em l}}
\newcommand{\hatsetd}{d\hat{\kern-.15em l}}
\newcommand{\dd}{\mathsf{d\kern-0.07em l}}

\let\amssymbboxplus\boxplus
\let\amssymbboxminus\boxminus
\renewcommand{\boxplus}{\mathbin{\mathop\amssymbboxplus}}
\renewcommand{\boxminus}{\mathbin{\mathop\amssymbboxminus}}

\definecolor{rose}{HTML}{D81B60}

\newcommand{\good}[1]{\ifmmode\mathbf{#1}\else\textbf{#1}\fi}

\usepackage{setspace}
\begin{document}

\title{A Composite Bayesian Risk--Distributionally Robust Optimization Model for Stochastic Programming with Elliptical Distribution Ambiguity}
\author{Renjie Chen\thanks{The authors are with the Chinese University of Hong Kong (\url{rchen@se.cuhk.edu.hk, hfxu@se.cuhk.edu.hk}).} \and Huifu Xu\footnotemark[1]} 
\date{}

\maketitle

\begin{abstract}

We study a one-stage stochastic optimization problem in which the random vector follows a semiparametric elliptical distribution with unknown location and shape parameters and an unspecified radial distribution. Assuming that the decision maker only has access to samples of the random vector, we develop a framework to jointly address parameter uncertainty and distributional ambiguity arising from this composite structure by integrating Bayesian posterior learning with distributionally robust optimization. To update the posterior distribution of the parameters, a key challenge is that the density of an elliptical distribution depends on the unknown radial distribution, and hence conventional likelihood-based Bayesian inference is infeasible. To overcome this challenge, we propose a novel estimating-equation-based Gibbs inference framework that updates posterior distributions of model parameters solely through estimating equations. We establish strong consistency of the resulting Gibbs posteriors under mild regularity conditions. For elliptical distributions, we construct a pair of estimating equations for the location and shape parameters that are invariant to the radial component, which enables posterior learning without specifying the density. For the distributional uncertainty, we construct an ambiguity set for the radial distribution by transforming observed samples into pseudo-samples of the radial variable under fixed location and shape parameters. Combining the above components, we formulate a composite Bayesian risk--distributionally robust optimization (CBR--DRO) model, in which the inner problem captures worst-case radial distributions within the ambiguity set, and the outer problem evaluates the decision risk induced by the parameter uncertainty through a user-specified risk measure. We further develop a cutting-plane algorithm to solve the resulting optimization problem. Numerical experiments demonstrate that the proposed approach provides substantial performance improvements over benchmark methods, including sample average approximation and distributionally robust optimization approaches, which implies the value of jointly modeling parameter uncertainty and distributional ambiguity.
\end{abstract}


\section{Introduction}

Stochastic optimization provides a principal framework for decision-making under uncertainty, but the quality of the resulting decisions depends on how the distribution of the underlying random vector is modeled and estimated. In many applications, the multivariate normal distribution is attractive because of its analytical convenience, yet its light-tail assumption can be restrictive when the data exhibit heavy tails or atypical observations. This issue is particularly relevant in financial applications, where the normality assumption for asset returns can be inadequate and hence, Student's $t$ distributions, or more generally elliptical distributions, have been widely used as alternatives~\citep[see, e.g.,][]{ref:bodnar2007distribution,ref:mcneil2015quantitative,ref:capponi2022systemic}. Elliptical distributions have also found broad applications in signal processing and machine learning; see \citet{ref:ollila2012complex} for a survey and \citet{ref:delmas2024elliptically} for a recent monograph. The flexibility of elliptical distributions makes them particularly useful for modeling multivariate data with a wide range of tail behaviors~\citep{ref:KOZUBOWSKI201359}.

A key feature of an elliptical distribution is its stochastic representation, i.e., $\xi \overset{d}{=} \mu_0+r\Sigma_0^{1/2}u$,  where $\mu_0\in\mathbb{R}^p$ is a location vector, $\Sigma_0\in\mathbb{S}_{++}^p$ is a shape matrix, and $u$ is uniformly distributed on the unit sphere, and $r\in\R_+$ is a nonnegative radial random variable independent of $u$. This representation separates the locational and directional structure of the distribution, represented by $(\mu_0,\Sigma_0)$, from its radial behavior, characterized by the distribution $P_0$ of $r$. 
Different choices of $P_0$ can therefore generate elliptical distributions with substantially different tail behaviors. The multivariate normal and Student's $t$ distributions are prominent examples of this family, see Examples \ref{exam:n-dis} and \ref{exam:t-dis}.

In this paper, we consider a stochastic optimization problem in which the underlying random vector $\xi$ is known to be elliptically distributed, but its location parameter $\mu_0$, shape matrix $\Sigma_0$, and radial distribution $P_0$ are all unknown. The only available information are $n$ i.i.d.\ observations of $\xi$. Under this setting, we seek a decision model that explicitly accounts for the epistemic uncertainty induced by finite data while preserving the known elliptical structure. This goal leads to two intertwined challenges.
The first is a \emph{decision modeling} challenge: 
how should one account for parametric uncertainty (of $\mu_0$ and $\Sigma_0$) and nonparametric uncertainty (of $P_0$) of an elliptical distribution in an integrated model, without sacrificing the tractability and convexity of the resulting optimization problem?
The second is a \emph{statistical inference} challenge: how can one quantify the epistemic uncertainty about $(\mu_0,\Sigma_0)$ when the radial distribution $P_0$ is unknown?   These two questions motivate the two main methodological developments of this paper.

We first discuss the decision-modeling challenge. A standard approach to stochastic optimization under distributional uncertainty is distributionally robust optimization (DRO). Instead of committing to a single estimated probability distribution, DRO constructs an ambiguity set of plausible distributions and optimizes against the worst-case distribution within this set, and hence hedges against the uncertainty. 
A large literature has developed ambiguity sets based on moment information
\citep{ref:delage2010distributionally,ref:hanasusanto2015distributionally,xu2018distributionally}, probability divergences
\citep{ref:ben2013robust}, and Wasserstein distances
\citep{ref:mohajerin2018data,ref:xie2021distributionally,ref:gao2023distributionally}, among others 
\citep[e.g.][]{ref:wiesemann2014distributionally,ref:hanasusanto2015distributionally}.
We refer to \citet{rahimian2019distributionally} and
\citet{ref:kuhn2025distributionally} for comprehensive reviews of the DRO literature.
The nonparametric nature of DRO makes it especially attractive when little structural information about the true distribution is available. However, when meaningful distributional structure is known
in advance, a generic ambiguity set that ignores distributional structure information may contain
unrealistic distributions and consequently lead to unnecessarily conservative
decisions \citep{ref:shang2022distributionally,ref:iyengar2023hedging,ref:jiang2024distributionally}.


At the other end of the spectrum, Bayesian approaches provide a natural framework for quantifying uncertainty in unknown parameters of the distributions. In Bayesian risk optimization (BRO)~\citep{ref:wu2018bayesian,ref:qian2019composite,ref:shapiro2023bayesian}, the underlying distribution is assumed to belong to a parametric family $\{P_\theta:\theta\in\Theta\}$. A prior distribution for the unknown parameter $\theta$ is specified based on available information or subjective beliefs and is subsequently updated through Bayes' rule using observed samples of $\xi$. To account for parameter-estimation uncertainty, a risk measure is applied to the resulting random objective value. This approach is particularly well suited to settings in which the distribution of the random vector is fully characterized by a finite-dimensional parameter.

The elliptical setting considered in this paper lies between these two regimes. Even if $(\mu_0,\Sigma_0)$ were known exactly, the distribution of $\xi$ would remain undetermined because the radial distribution $P_0$ is unspecified. Hence, a purely parametric Bayesian model does not capture all distributional uncertainty. Conversely, treating the entire distribution of $\xi$ nonparametrically through a conventional DRO ambiguity set ignores the structural distinction between uncertainty about $(\mu_0,\Sigma_0)$ and uncertainty about $P_0$. This distinction is particularly important because the two sources of uncertainty have different forms: the former concerns finite-dimensional structural parameters, whereas the latter concerns an unknown probability distribution.

Motivated by this observation, we propose a \emph{composite Bayesian risk--distributionally robust optimization} (CBR--DRO) framework that treats these two sources of uncertainty separately and differently. Based on $n$ i.i.d. observations $\xi_1,\ldots,\xi_n$, we use a posterior distribution $\Pi_n$ to represent the epistemic uncertainty about the finite-dimensional parameters $(\mu_0,\Sigma_0)$, and an ambiguity set is used to represent uncertainty about the nonparametric radial distribution $P_0$. Given a candidate pair $(\mu,\Sigma)$, by the stochastic representation of elliptical distributions, the observations $\xi_1,\ldots,\xi_n$ can be transformed into pseudo-observations of the radial variable,
$
    r_i(\mu,\Sigma)
    \coloneqq
    \left\|
        \Sigma^{-1/2}(\xi_i-\mu)
    \right\|,
$ for $i=1,\ldots,n$, which can then be used to construct an ambiguity set $\mathcal{M}(\mu,\Sigma)$ for $P_0$. This leads to the model
\begin{equation}
    \min_{x\in\mathcal{X}}
    \rho_{(\mu,\Sigma)\sim\Pi_n}
    \left(
        \max_{Q\in\mathcal{M}(\mu,\Sigma)}
        \mathbb{E}_{r\sim Q,\,
        u\sim\mathrm{Unif}(\mathbb{U}^{p-1})}
        \left[
            f\bigl(x,\mu+r\Sigma^{1/2}u\bigr)
        \right]
    \right),
\end{equation}
where the inner DRO problem hedges against ambiguity in the radial distribution, and the outer risk measure $\rho$ evaluates the decision risk induced by posterior uncertainty in the location and shape parameters. The inner ambiguity set $\mathcal{M}(\mu,\Sigma)$ is not introduced independently, since the pseudo-observations of the radial variable are constructed from predetermined parameters $(\mu,\Sigma)$. 
We note that the framework does not require any particular choice of ambiguity set, and can accommodate, for example, Wasserstein-, moment-, or divergence-based constructions.

The CBR--DRO construction, however, requires a posterior distribution for $(\mu_0,\Sigma_0)$, and obtaining such a posterior under a general elliptical model is itself nontrivial. In a classical Bayesian model with density $p(\xi\mid\theta)$, the likelihood directly specifies how observations update the prior distribution of $\theta$. When the density of an elliptically distributed random vector exists, it takes the form
$
    f_\xi(x)
    \propto
    g\left(
        (x-\mu_0)^\top
        \Sigma_0^{-1}
        (x-\mu_0)
    \right),
$
where the density generator $g$ is determined by the radial distribution $P_0$. However, since $P_0$, and hence $g$, is unknown in general elliptical setting, there is no fully specified likelihood for $(\mu_0,\Sigma_0)$ on which to base a conventional Bayesian update.


This difficulty naturally motivates the use of generalized Bayesian inference, a subcategory of which is known as Gibbs posterior inference. Gibbs posterior methods replace the negative log-likelihood with a problem-specific loss function and update the prior through exponential weighting of the empirical loss. 
Early developments include \citet{ref:zhang2006information}, \citet{ref:bissiri2010bayesian}, and \citet{ref:bissiri2012bayesian}, and \citet{ref:bissiri2016general} develop a general framework for updating beliefs through loss functions. The concentration and calibration properties of Gibbs posteriors for risk minimizers have subsequently been studied in \citet{ref:martin2022direct} and \citet{ref:syring2023gibbs}.
In PAC-Bayesian theory, Gibbs posteriors also arise as variational optimizers that balance empirical risk against a Kullback--Leibler divergence from the prior~\citep{ref:mcallester1998some,ref:catoni2007pac}.
Existing Gibbs posterior methods are particularly useful when a fully specified likelihood is unavailable. However, they are typically loss-function-based and are therefore less directly applicable when the parameter of interest can not be characterized as a minimizer of an expected loss.


Our development starts from a different primitive. Rather than beginning with a loss-function-based parameter, we consider the case when the parameter of interest $\theta_0$ is characterized directly as the unique root of a general estimating equation $\Phi(\theta)
    \coloneqq
    \mathbb{E}[\phi(\theta,\xi)]
    =
    0.$
Given observations $\xi_1,\ldots,\xi_n$, we construct a Gibbs posterior by directly penalizing the residual of the empirical estimating equation, with posterior density proportional to
$
    \exp\left(
        -\eta
        \left\|
            \sum_{i=1}^n
            \phi(\theta,\xi_i)
        \right\|
    \right)\pi(\theta),
$
where $\pi$ is a prior density and $\eta>0$ is the learning rate. The magnitude of the empirical estimating equation therefore plays the role of a data-dependent discrepancy: parameter values for which the estimating equation is closer to being satisfied receive larger posterior weight, and the learning rate $\eta$ controls how strongly the posterior density responds to this empirical discrepancy relative to the prior.
We note that the loss-function-based formulation can be recovered when the estimating function corresponds to the first-order condition of the loss, whereas the estimating-equation formulation remains applicable when a convenient integrable loss function is unavailable.


We exploit the Gibbs framework to address posterior inference for elliptical distributions. Although the unknown radial distribution prevents likelihood-based inference, the stochastic representation of an elliptical random vector reveals that the radial magnitude can be removed through suitable normalization. For the location parameter, we use the normalized directional estimating function $\varphi_1(\mu,\xi) =    \frac{\xi-\mu}{\|\xi-\mu\|}$.
For the shape parameter, we construct an estimating equation closely related to Tyler's distribution-free scatter estimator and the angular central Gaussian model~\citep{ref:tyler1987distribution,ref:tyler1987statistical}. The resulting estimating equations depend only on locational and directional information and normalize out the radial component. Consequently, their population roots identify the location and shape parameters without specifying $P_0$. This feature is particularly useful for heavy-tailed elliptical distributions: the location and shape parameters remain well defined through the stochastic representation even when the first or second moments of the distribution do not exist. After imposing a scale normalization on the shape matrix, we show that the joint estimating equation uniquely identifies $(\mu_0,\Sigma_0)$ and establish strong consistency of the corresponding Gibbs posterior.

The proposed CBR--DRO model also raises computational challenges. Its objective is a risk functional with respect to the posterior distribution of $(\mu,\Sigma)$, while every posterior realization requires evaluating an inner DRO value function. 
To deal with such structure, we adopt sample average approximation (SAA) of the risk measure, and develop a cutting-plane algorithm based on a first-order oracle for the inner DRO value function. The method is related to the classical cutting-plane method of \citet{ref:kelley1960cutting}, but exploits the composite structure of our model by constructing scenario-wise piecewise-linear lower approximations of the inner value functions before applying the outer risk measure. This produces a lower approximation that retains information from individual posterior scenarios rather than linearizing the aggregated risk objective directly. We establish convergence of the upper and lower bounds, optimality of accumulation points, and finite termination for every prescribed positive optimality tolerance.

The main contributions of this paper can be summarized as follows.
\begin{itemize}

    \item \textbf{Composite modeling of elliptical distribution ambiguity.}
    We propose a CBR--DRO framework for stochastic optimization when the underlying random vector is elliptically distributed but its location, shape, and radial distribution are unknown. By exploiting the stochastic representation of elliptical distributions, the model separates parametric uncertainty in the finite-dimensional location and shape parameters from distributional ambiguity in the radial component. The former is handled through a Bayesian risk measure, and the latter is handled through an inner DRO problem constructed from pseudo-observations of radial variable. This provides an integrated uncertainty-aware framework that preserves known elliptical structure information.

    \item \textbf{Gibbs posterior inference based on estimating equations.}
    We develop a novel Gibbs posterior framework in which the parameter of interest is characterized directly as the root of an estimating equation and establish general sufficient conditions for strong posterior consistency. We then construct radial-invariant estimating equations that jointly identify the location and normalized shape parameters of an elliptical distribution. Combining the methodology, we obtain a consistent posterior inference procedure for $(\mu_0,\Sigma_0)$, which does not depend on the specification of the radial distribution or the existence of the moments of the underlying elliptical distribution.

    \item \textbf{Computational developments.}
    We establish convexity and a uniform sample approximation result for the CBR--DRO objective, and develop a tailored cutting-plane algorithm with convergence and finite-termination guarantees. We further consider the application of CBR--DRO to portfolio optimization, and conduct numerical tests on both synthetic data and real data.

\end{itemize}

The remainder of the paper is organized as follows. Section~\ref{sec:prel} reviews elliptical distributions and risk measures. Section~\ref{sec:model} formulates stochastic optimization under elliptical distribution ambiguity and introduces the CBR--DRO model. Section~\ref{sec:gibbs} develops the estimating-equation-based Gibbs posterior framework and specializes it to posterior inference for the location and shape parameters of elliptical distributions. Section~\ref{sec:comp} studies the computational properties of the model, establishes the posterior sample approximation result, and develops the cutting-plane algorithm. Section~\ref{sec:num-exp} introduces the application of CBR--DRO to portfolio optimization, and reports the numerical results. 
Some technical lemmas and proofs can be found in Section~\ref{sec:lemmas}.

\noindent\textbf{Notation.}
We use $\R$, $\R_+$, and $\R_{++}$ to denote the sets of real,
nonnegative real, and strictly positive real numbers.
The space of real symmetric $p\times p$ matrices and its cones of
positive semidefinite and positive definite matrices are denoted by
$\S^p$, $\S^p_+$, and $\S^p_{++}$, respectively.
Unless otherwise specified, $\|\cdot\|$ denotes the Euclidean norm,
and $\langle x,y\rangle=x^\top y$ is the inner product of vectors of
the same dimension.
For matrices $A$ and $B$ of the same dimensions, the Frobenius inner
product and norm are defined by $\langle A,B\rangle=\tr(A^\top B)$
and $\|A\|_{\F}=\sqrt{\langle A,A\rangle}$.
We write $I_p$, or simply $I$ when the dimension is clear, for the $p\times p$
identity matrix, and $\diag(a_1,\ldots,a_p)$ for the diagonal matrix
with diagonal entries $a_1,\ldots,a_p$.
The smallest and largest eigenvalues of a symmetric matrix $A$ are
written as $\lambda_{\min}(A)$ and $\lambda_{\max}(A)$.
For $A\in\S^p_+$, $A^{1/2}$ denotes its unique symmetric positive
semidefinite square root; when $A\in\S^p_{++}$, we write
$A^{-1/2}=(A^{1/2})^{-1}$.
The unit sphere in $\R^p$ is denoted by
$\mathbb{U}^{p-1}=\{u\in\R^p:\|u\|=1\}$, and
$\mathrm{Unif}(\mathbb{U}^{p-1})$ denotes the uniform distribution
on this sphere.
Finally, $\mathds{1}_A$ denotes the indicator of an event $A$.

\section{Preliminaries}\label{sec:prel}

In this section, we recall some basic concepts about elliptical distributions and risk measures to be used throughout the paper.

\subsection{Spherical Distribution and Elliptical Distribution}

We begin by recalling the concepts of spherical distribution and elliptical distribution and some related properties from~\cite{ref:fang2018symmetric}. 
Let $\xi$ be a random vector supported on $\R^p$. 
$\xi$ is said to be \textit{spherically distributed}, or equivalently, has \textit{spherical distribution}, if for every $p\times p$ orthogonal matrix $Q$, 
$\xi \diseq Q\xi$.
Let $\varphi(t) \Let \E[e^{it^\transpose\xi}]$ be the characteristic function of $\xi$. Then $\xi$ is spherically distributed if and only if there exists a function $\phi:[0,+\infty)\to \mathbb{C}$, called characteristic generator of $\xi$, such that $\varphi(t) = \phi(t^\transpose t)$ for any $t \in \R^p$. 
Moreover, $\xi$ is spherically distributed 
if and only if it has a stochastic representation
 $   \xi \diseq ru$,
where $r$ is a random variable supported on $\R_+$, $u\sim \mathrm{Unif}(\mathbb{U}^{p-1})$ is a random vector distributed uniformly on the unit sphere $\mathbb{U}^{p-1}$ in $\R^p$, and $r$ and $u$ are independent. The concept of spherical distribution is an extension of the standard multinormal distribution $N(0,I_p)$. 

\begin{example}[Standard Multivariate Normal Distribution]
    Let $\xi\sim N(0,I_p)$. Then the characteristic function of $\xi$ is 
    $
        \varphi(t) = \exp(-t^\transpose t/2).
    $
    Moreover, it admits the stochastic representation $\xi\diseq ru$, where $r^2\sim \chi^2_p$, $u$ is distributed uniformly on the unit sphere, and $r$ and $u$ are independent.
\end{example}

The concept of elliptical distribution is an extension of $N(\mu_0,\Sigma_0)$.

\begin{definition}[Elliptical distribution]
    A random vector $\xi$ is said to be \textit{elliptically distributed} (or has \textit{elliptical distribution}), denoted by $\xi \sim \operatorname{EC}(\mu_0, \Sigma_0, P_0)$, if it admits the stochastic representation
\begin{align}
\label{eq:EC-stochastic-rep}
    \xi \diseq \mu_0 + r\Sigma_0^{1/2} u,
\end{align}
where $\mu_0 \in \R^p$ is the location vector, $\Sigma_0\in \PD$ is the shape matrix, $r$ is a random variable supported on $\R_+$, $P_0$ is the distribution of $r$, $u\sim \mathrm{Unif}(\mathbb{U}^{p-1})$ is a random vector distributed uniformly on the unit sphere in $\R^p$, and $r$ and $u$ are independent. 
A random vector is elliptically distributed with parameters $\mu_0$, $\Sigma_0$ and $P_0$ 
if and only if its characteristic function can be written as
\[
    \varphi(t) = \E[e^{it^\transpose\xi}] = \exp(it^\transpose \mu_0) \psi(t^\transpose \Sigma_0 t),
\]
where $\psi:[0,+\infty)\to \mathbb{C}$ is the characteristic generator that is uniquely determined by $P_0$ (up to scale). 
Moreover, the density function of an elliptically distributed random vector $f_{\xi}$, if exists, 
has the form
\[
    f_{\xi}(x) = c\cdot g\left((x-\mu_0)^\transpose \Sigma_0^{-1} (x-\mu_0)\right),
\]
where $c$ is a normalization scaler, and $g: \R_+\to \R_+$ is called a density generator, which is also uniquely determined by $P_0$. 
\end{definition}

Throughout the paper, we assume that the elliptically distributed random vector $\xi$ admits a bounded density function, i.e., $\|f_\xi\|_{\infty} \Let \sup_{x \in \R^p} f_\xi(x) < \infty$.
To see why the above is called elliptical distribution, we note that an ellipsoid in Euclidean space can be represented by an affine image of a ball: $
    \mathcal{E} = \{x=\mu_0+\Sigma_0^{1/2}u: \|u\|=r\}$,
where $\mu_0$ is the center of the ball, $\Sigma_0$ is the affine mapping, and $r$ is the radius of the ball. Then~\eqref{eq:EC-stochastic-rep} can be viewed as a stochastic counterpart of the ellipsoid.

\begin{example}[Multivariate Normal Distribution]\label{exam:n-dis}
    Let $\xi\sim N(\mu_0,\Sigma_0)$. Then it admits the stochastic representation $\xi\diseq \mu_0 + r\Sigma_0^{1/2} u$,
    where $r^2$ follows the $\chi^2_p$-distribution, $u$ is distributed uniformly on the unit sphere in $\R^p$, and $r$ and $u$ are independent.
    Moreover, the characteristic function of $\xi$ is $
        \varphi(t) = \E[e^{it^\transpose\xi}] = \exp(it^\transpose \mu_0) \exp(-t^\transpose \Sigma_0 t/2)$.
    The density function of $\xi$ is 
    \[
        f(x) = \frac{1}{\sqrt{(2\pi)^p\det(\Sigma_0)}} \exp\left(-\frac{1}{2}(x-\mu_0)^\transpose \Sigma_0^{-1} (x-\mu_0)\right).
    \]
\end{example}

\begin{example}[Multivariate Student's $t$ Distribution]\label{exam:t-dis}
    Let $\xi \sim t_\nu(\mu_0,\Sigma_0)$ be a Student's $t$ distribution with degrees of freedom $\nu>0$, location vector $\mu_0$, and shape matrix $\Sigma_0$. Then it admits the stochastic representation $    \xi \diseq \mu_0 + r\,\Sigma_0^{1/2}u$,
    where $u$ is distributed uniformly on the unit sphere in $\R^p$, $r$ and $u$ are independent, and $r^2/p$ follows the $F$-distribution with degrees of freedom $p$ and $\nu$.
    Moreover, the characteristic function of $\xi$ is $
    \varphi(t) = \E\left[e^{it^\transpose \xi}\right]
    = \exp\left(it^\transpose \mu_0\right)\psi\left(t^\transpose \Sigma_0 t\right)$,
    where, for $s\ge 0$,
    $
    \psi(s) = \frac{\big(\sqrt{\nu s}\big)^{\nu/2}\,K_{\nu/2}\big(\sqrt{\nu s}\big)}
         {2^{\nu/2-1}\Gamma(\nu/2)},
    $
    and $\Gamma(\cdot)$ is the gamma function and $K_{\alpha}(\cdot)$ is the modified Bessel function of the second kind \citep{ref:kibria2006short}.
    The density function of $\xi$ is 
    \[
        f(x) = \frac{\Gamma\big((\nu+p)/2\big)}{\Gamma(\nu/2)\nu^{p/2}\pi^{p/2}\det(\Sigma_0)^{1/2}}\left(1+\frac{1}{\nu}(x-\mu_0)^\transpose \Sigma_0^{-1} (x-\mu_0)\right)^{-(\nu+p)/2}.
    \]
\end{example}

\begin{remark}[Moments of elliptical distributions]\label{rmk:ellip}
    The tail behavior and moment properties of an elliptically distributed random vector $\xi$ can be very flexible and are determined by the distribution of $r$. For example, if $\E[r]$ exists, then the mean of $\xi$ is given by $
        \E[\xi] = \mu_0 + \E[r]\Sigma_0^{1/2}\E[u]=\mu_0,$
    since $r$ and $u$ are independent and $\E[u]=0$. More generally, the $k$-th moment of $\xi$ exists if and only if 
    the $k$-th moment of $r$ exists, since 
    \[
        \E[\|\xi-\mu_0\|^k] = \E[r^k\|\Sigma_0^{1/2}u\|^k] = \E[r^k]\cdot\E[\|\Sigma_0^{1/2}u\|^k],
    \]
    and $u$ is a bounded random vector. Nevertheless, even when the first or second moment of $\xi$ does not exist, the location vector $\mu_0$ and the shape matrix $\Sigma_0$ remain well defined through the stochastic representation~\eqref{eq:EC-stochastic-rep}.
\end{remark}

\begin{remark}[Scale non-identifiability]\label{rmk:scale}
In the stochastic representation~\eqref{eq:EC-stochastic-rep}, the scale of the shape matrix $\Sigma_0$ is not identifiable without a normalization on the radial distribution $P_0$. Indeed, for any constant $c>0$,
\[
    \xi \diseq \mu_0 + r\Sigma_0^{1/2}u
    \;\diseq\;
    \mu_0 + (c^{-1/2}r)\,(c\Sigma_0)^{1/2}u,
\]
so that the same elliptical distribution can be generated by simultaneously rescaling $r$ and $\Sigma_0$. Consequently, only the shape of $\Sigma_0$ is identifiable, while its scale depends on the choice of the distribution of $r$. To ensure identifiability, one usually imposes a scale normalization. In this paper, we fix the scale of $\Sigma_0$ by requiring $
    \tr(\Sigma_0)=p,$
so that $\Sigma_0$ is uniquely defined.
\end{remark}

Before we 
end this
subsection, we recall the following property of $u$, a random vector distributed uniformly on the unit sphere $\mathbb{U}^{p-1}$, which will be used throughout the paper.
\begin{proposition}\label{prop:E-uu}
    Let $u$ be a random vector that has a uniform distribution over the unit sphere $\mathbb{U}^{p-1}$. Then 
     $   \E[uu^\top] = I/p.
    $
\end{proposition}

\subsection{Risk Measure}


In our framework, we use risk measure to capture the epistemic uncertainty in the parameters of the probability distributions. Let $(\Omega,\mathcal{F},\mathbb{P})$ be a probability space, and for $q\geq 1$, let
\[
    L^q\Let\left\{X:\Omega\to\R~\text{measurable}\mid \int_{\Omega} |X(\omega)|^q \mathbb{P}(\diff \omega) < \infty\right\}
\]
be the set of all measurable random variables defined over $(\Omega,\mathcal{F},\mathbb{P})$ with finite $q$-th order moments.
The norm defined over $L^q$ is 
 $   \|X\|_{L^q} \Let \left(\int_{\Omega} |X(\omega)|^q \mathbb{P}(\diff \omega)\right)^{1/q}.
$
    A functional $\rho: L^q \to \R$ is called a {\em monetary risk measure} if it satisfies 
        (a) (monotonicity) 
        for any $X, Y\in L^q$, if $X\ge Y$, then $\rho(X) \ge \rho(Y)$, where $X\ge Y$ means that $X(\omega)\ge Y(\omega)$ for any $\omega\in\Omega$; 
        (b) (translation invariance) for any $X\in L^q$ and $c\in\R$, $\rho(X+c) = \rho(X) + c$.
In particular, it is called a {\em convex risk measure}
   if it satisfies (c) (convexity) for any $X, Y\in L^q$ and $\lambda \in [0,1]$, $\rho(\lambda X + (1-\lambda)Y) \le \lambda\rho(X) + (1-\lambda)\rho(Y)$,
and {\em coherent risk measure}
if it further satisfies
(d)
(positive homogeneity) for any $X\in L^q$ and $\lambda\ge 0$, $\rho(\lambda X) = \lambda \rho(X)$, see \cite{ref:artzner1999coherent} 
and \cite{ref:follmer2002convex}.

A risk measure $\rho$ is 
said to be {\em law invariant} if $\rho(X) = \rho(Y)$ for any random variables $X,Y \in L^q$ with the same probability distribution. We note that any law invariant risk measure defined over a non-atomic $L^q$ space can be represented as a functional over the space of the probability distributions
 $   \mathscr{P}_q(\R)\Let \left\{\mu\in\mathscr{P}(\R):\int_{\R} |z|^q\mu(\diff z) < \infty\right\},
$
where $\mathscr{P}(\R)$ stands for the space of all possible probability distributions over $\R$. That is, there exists a functional $\varrho:\mathscr{P}_q(\R)\to \R$ such that 
$\rho(X) = \varrho(\mu)$ for any $X\in L^q$, where $\mu\in\mathscr{P}_q(\R)$ is the distribution of $X$. 
Throughout this paper, we assume that the risk measure $\rho$ under consideration is law invariant, convex, proper, and lower semicontinuous.
The examples below list several commonly used risk measures in the literature and applications, some of which will be mentioned or discussed later in this paper.
\begin{example} 
Let $X\in L^q$.
\begin{enumerate}[label = (\roman*)]
    \item (Expectation) 
    $\rho(X)\Let\E[X]$
    is a coherent risk measure.
    
    \item (Essential supremum) 
    $\rho(X)\Let \operatorname*{ess\,sup}_{\omega\in\Omega} X(\omega)$
    is a coherent risk measure.
    
    \item (Entropic risk measure) For $\alpha>0$, 
     $
        \rho_{\alpha}(X)
        \Let \frac{1}{\alpha}\log \E\left[e^{\alpha X}\right]
$
is 
a convex
risk measure.

    \item (Value-at-Risk) For a confidence level $\alpha\in(0,1)$, 
     $   \mathrm{VaR}_\alpha(X)
        \Let \inf\left\{m\in\R:\ \mathbb{P}(X\le m)\ge \alpha\right\}
        $
    is a risk measure (monotone and translation invariant), but it is not convex, hence not coherent.
    
    \item (Conditional Value-at-Risk) For a confidence level $\alpha\in(0,1)$, 
     $   \mathrm{CVaR}_\alpha(X) \Let \frac{1}{1-\alpha}\int_\alpha^1 \mathrm{VaR}_t(X) \diff t
    $ 
    
    is a coherent risk measure.

        \item (Optimized certainty equivalent,~\cite{ref:ben1986expected,ref:ben2007old}) Let $\ell:\R\to\R$ be a proper,
    convex, lower semicontinuous, and nondecreasing loss function satisfying
    $\ell(0)=0$.
     Then $
        \rho_{\ell}(X)     \Let
        \inf_{t\in\R}
        \left\{
            t+\E\left[\ell(X-t)\right]
        \right\}
$
 is 
  a convex risk measure but not coherent in general. 

    \item (Spectral risk measure,~\cite{ref:acerbi2002spectral}) Let $\phi:[0,1]\to\R_+$ be a nonnegative and 
    nondecreasing function satisfying
     $   \int_0^1 \phi(t)\diff t = 1.
    $ 
    Then 
     $   \rho_{\phi}(X)
        \Let
        \int_0^1 \mathrm{VaR}_t(X)\phi(t)\diff t 
    $
a coherent risk measure. In particular,
    $\mathrm{CVaR}_{\alpha}$ is a spectral risk measure with
  $      \phi(t)=\frac{1}{1-\alpha}\mathbf 1_{\{t\in[\alpha,1]\}}.
   $

\end{enumerate}

\end{example}

\section{Stochastic Optimization under Elliptical Distribution Ambiguity}\label{sec:model}

In this section, we propose our composite Bayesian risk--distributionally robust optimization (CBR--DRO) model for stochastic optimization problem with elliptical distribution ambiguity. Before that, we first review the basic idea of stochastic optimization, as well as two popular models in the literature that are used for decision-making under epistemic uncertainty in the underlying true distribution.

\subsection{Stochastic Optimization under Epistemic Uncertainty}\label{sec:SO-EU}

 We consider the following one-stage stochastic minimization problem: \begin{align}\label{eq:one-stage-SM-general} \min_{x\in \mathcal{X}} \E_{P}[f(x,\xi)], \end{align} where 
 $\xi$ is a random vector in $\mathbb{R}^p$ with distribution $P$,
 $\mathcal{X}\subset \mathbb{R}^d$ denotes the feasible set and $f:\mathcal{X}\times \mathbb{R}^p\to \mathbb{R}$ is the cost function. Throughout this paper, we assume that $\mathcal{X}$ is convex and compact, and that $f(\cdot,\xi)$ is convex for every fixed $\xi\in\mathbb{R}^p$. In practice, particularly in data-driven settings, the true distribution $P$ is typically unknown and only limited observational data of $\xi$ are available. To mitigate the decision-making risk arising from such epistemic uncertainty, one may construct an ambiguity set $\mathcal{M}$ based on historical data, such that the true distribution $P$ belongs to $\mathcal{M}$ with high confidence. This leads to the following distributionally robust optimization (DRO) formulation: \[ \min_{x\in\mathcal{X}} \max_{Q \in \mathcal{M}} \E_{Q}[f(x,\xi)], \] where the decision maker seeks a solution that performs best under the worst-case distribution within the ambiguity set. As a nonparametric framework, DRO has attracted considerable attention in recent years due to its flexibility in modeling distributional ambiguity and its tractability across a broad range of applications. Nevertheless, DRO is often criticized for being overly conservative, which can lead to poor out-of-sample performance, especially when the ambiguity set is chosen to be excessively large. We refer the reader to~\cite{rahimian2019distributionally,ref:kuhn2025distributionally} for comprehensive surveys of the DRO literature.

Another principal way to deal with the epistemic uncertainty in~\eqref{eq:one-stage-SM-general} is the Bayesian risk optimization (BRO) approach~\citep{ref:wu2018bayesian,ref:qian2019composite}. In this approach, the true distribution $P$ is supposed to belong to a parametric family $\{P_{\theta}: \theta\in \Theta\}$, where $\Theta$ is a finite dimensional parameter space, and there exists $\theta^\star\in \Theta$ such that $P=P_{\theta^\star}$. Then from a Bayesian perspective, the epistemic uncertainty on $P$ is captured by the randomness of our belief on $\theta$. Let $\theta$ be a random variable whose posterior distribution is inferred from the historical data of $\xi$. Then the objective of~\eqref{eq:one-stage-SM-general} becomes a random variable induced by the distribution of $\theta$, i.e., $\E_{P_{\theta}}[f(x,\xi)]$. To hedge against the uncertainty captured by the randomness of $\theta$, one may consider the following Bayesian risk optimization (BRO) problem:
\[
    \min_{x\in\mathcal{X}} \rho\left(\E_{P_{\theta}}[f(x,\xi)]\right),
\]
where the decision-making risk is modeled by some risk measure $\rho$. The BRO approach is suitable when the true distribution is believed to from a parametric family. Conversely, it is less appropriate when the distribution is nonparametric or when the distribution family is unknown.






\subsection{Composite Bayesian Risk--Distributionally Robust Optimization Model}

In this paper, we consider a one-stage stochastic optimization problem \eqref{eq:one-stage-SM-general} where
the random vector 
$\xi$  
%
is elliptically distributed specified as in \eqref{eq:EC-stochastic-rep}. 
In this setup, the stochastic program can be specifically formulated as 
\begin{align}\label{prob:stochastic-EC}
    \min_{x\in \mathcal{X}} \E_{\xi\sim \operatorname{EC}(\mu_0,\Sigma_0,P_0)}[f(x,\xi)].
\end{align}
We concentrate on the case that
the true values of $\mu_0,\Sigma_0$ 
and the true distribution $P_0$ are 
unknown, and we have to estimate the distribution of $\xi$ with sample data of $\xi$ and based on the estimation, approximate the solution of~\eqref{prob:stochastic-EC}. Because of limited data, the estimation can be inaccurate and the approximate solution may perform poorly. That is, again, we are confronted with decision-making risk caused by the epistemic uncertainty in the underlying elliptical distribution $\operatorname{EC}(\mu_0,\Sigma_0,P_0)$. We note that the epistemic uncertainty in the elliptical distribution can be 
decomposed into two parts: incomplete information about the true values of parameters $\mu_0$ and $\Sigma_0$, and incomplete information about the nonparametric true distribution $P_0$. To hedge
the modeling risk arising from 
the two types of uncertainty in an integrated model, 
we propose the following composite distributionally robust Bayesian risk optimization model:
\begin{align}\label{prob:CBR--DRO}
    \min_{x\in X} \rho_{(\mu,\Sigma)\sim \Pi_n}\left(\max_{Q \in \mathcal{M}(\mu,\Sigma)}  \E_{r\sim Q, u\sim \mathrm{Unif}(\mathbb{U}^{p-1})}[f(x,\mu+r\Sigma^{1/2}u)]\right),
    \tag{CBR--DRO}
\end{align}
where $\mathcal{M}(\mu,\Sigma)$ is an ambiguity set of probability distributions that comprises some plausible candidates for $P_0$.
The inner maximization problem
tackles the ambiguity of the true distribution $P_0$ by considering the worst-case over $\mathcal{M}(\mu,\Sigma)$.
The notation $\mathcal{M}(\mu,\Sigma)$ highlights that, although the true distribution $P_0$ itself does not depend on parameters $\mu_0$ and $\Sigma_0$, the construction of its ambiguity set may rely on our estimates of these quantities, i.e., $(\mu,\Sigma)$.
To see this, we note that
under the elliptical representation of $\xi$, we can write $r$ in terms of $\xi$ by $
    r = \|\Sigma_0^{-1/2}(\xi - \mu_0)\|,~~r\sim P_0.$
Since $(\mu_0,\Sigma_0)$ are unknown, we cannot obtain samples of $r$ based on the observations of $\xi$. 
Let $\Sigma$ and $\mu$
be an estimate of $(\mu_0,\Sigma_0)$ 
and let
$
    r_i \triangleq \|\Sigma^{-1/2}(\xi_i - \mu)\|,~~i=1,\ldots,n.$
We may treat $r_1,\cdots,r_n$
as
empirical proxies for the draws from $P_0$ and construct an empirical distribution 
$\wh P_n$. We can then construct an ambiguity set 
$\mathcal{M}(\mu,\Sigma)$ with $\wh P_n$. This explains why the ambiguity set depends on $(\mu,\Sigma)$. Note that, at this stage, we do not impose a specific structure on $\mathcal{M}(\mu,\Sigma)$. In principle, any data-driven ambiguity set can be incorporated into our framework, including Wasserstein-based~\citep{ref:mohajerin2018data,ref:gao2023distributionally}, moment-based~\citep{ref:delage2010distributionally}, and $\phi$-divergence-based ambiguity sets~\citep{ref:ben2013robust}.


Given a candidate $(\mu,\Sigma)$, the worst-case cost function of the inner DRO problem in~\eqref{prob:CBR--DRO} is denoted by
\begin{align}\label{def:F}
    F(x,\mu,\Sigma)\Let \max_{Q \in \mathcal{M}(\mu,\Sigma)}  \E_{r\sim Q, u\sim \mathrm{Unif}(\mathbb{U}^{p-1})}[f(x,\mu+r\Sigma^{1/2}u)].
\end{align}
In~\eqref{prob:CBR--DRO}, we further adopt a BRO approach to hedge against decision risk induced by the parametric uncertainty in $F(x,\mu,\Sigma)$, that is, the potential estimation error of $(\mu_0,\Sigma_0)$. Specifically, we treat $\mu$ and $\Sigma$ as random variables and construct posterior distributions $\Pi_n$ of them with observed sample data of $\xi$ to 
encode our belief about 
the underlying true parameters $\mu_0$ and $\Sigma_0$. 
This posterior uncertainty propagates to $F(x,\mu,\Sigma)$, which makes it a random variable. Finally, we apply a risk measure $\rho$ to $F(x,\mu,\Sigma)$ and consider 
\[
    \min_{x\in \mathcal{X}} \rho_{(\mu,\Sigma)\sim \Pi_n}\left(F(x,\mu,\Sigma)\right)
\]
to obtain decisions that are robust to the uncertainty in $(\mu,\Sigma)$.

\subsubsection{Comparison with Robust--Distributionally Robust Optimization}

To see the modeling advantage of~\eqref{prob:CBR--DRO}, we compare it with a more straightforward framework that utilizes both robust optimization and distributionally robust optimization to hedge against parametric uncertainty and nonparametric uncertainty in an integrated model, that is,
\begin{align}\label{prob:robust-dro}
    \min_{x\in X} \max_{(\mu,\Sigma) \in \mathcal{U}} F(x,\mu,\Sigma),
\end{align}
where $F(x,\mu,\Sigma)$ is the worst-case cost function that accounts for the distributional uncertainty defined as in~\eqref{def:F}, and an ambiguity set $\mathcal{U}\subset \R^p\times \PD$ is introduced to model the uncertainty in $\mu$ and $\Sigma$, and is believed to contain true parameters $(\mu_0,\Sigma_0)$. 
Although~\eqref{prob:robust-dro} offers a clear modeling framework, it may be numerically intractable since $F(x,\mu,\Sigma)$ is generally non-concave in $(\mu,\Sigma)$. Then, in view of~\eqref{prob:CBR--DRO}, it serves not only as a modeling alternative to~\eqref{prob:robust-dro}, but also as a computationally tractable approximation of it (the computational framework will be discussed in Section~\ref{sec:comp}).  To see how~\eqref{prob:CBR--DRO} approximates~\eqref{prob:robust-dro}, suppose that the posterior distribution of $(\mu,\Sigma)$ supports on $\mathcal{U}$.
Then taking $\rho(X)=\mathrm{CVaR}_\alpha(X)$ for some $\alpha\in (0,1)$, we consider a special case of~\eqref{prob:CBR--DRO}:
\begin{align}\label{prob:cvar-F}
    \min_{x\in X} \mathrm{CVaR}_\alpha\left(F(x,\mu,\Sigma)\right).
\end{align}
Since $\lim_{\alpha\to 1}\mathrm{CVaR}_\alpha(F(x,\mu,\Sigma)) =\max_{(\mu,\Sigma)\in\mathcal{U}} F(x,\mu,\Sigma)$ \citep[see][Theorem~1]{ref:anderson2020varying}, 
\eqref{prob:cvar-F} is an approximation of~\eqref{prob:robust-dro}, especially when $\alpha$ approaches $1$. 
Such approximation property also holds for many other risk measures, e.g., $\mathrm{VaR}_\alpha$, the entropic risk measure $\frac{1}{\eta}\log \mathbb{E}[e^{\eta X}]$, and more generally, law-invariant coherent risk measures. Indeed, for $\mathrm{CVaR}_\alpha$ and $\mathrm{VaR}_\alpha$, the risk measures focus on the upper $\alpha$-tail of $F(x,\mu,\Sigma)$, and thus converge to $\max_{(\mu,\Sigma)\in \mathcal{U}} F(x,\mu,\Sigma)$ as $\alpha\to 1$. For entropic risk measure, $\E[e^{\eta X}]$ is dominated by the largest value of $X$ as $\eta\to \infty$, and thus the risk measure also converges to the maximum. For law-invariant coherent risk measures, one can construct a sequence of the risk measures that converges to the worst case, since any of them admits a representation in terms of $\mathrm{CVaR}_\alpha$~\citep{ref:shapiro2013kusuoka}.

\section{Construction of Posterior Distribution}\label{sec:gibbs}

In this section, we discuss the construction of the posterior distribution of $(\mu,\Sigma)$.
Let $\xi_1,\ldots,\xi_n$ be i.i.d. samples of random vector $\xi$. Recall that, in the classical Bayesian setting where the distribution of $\xi$ is assumed to be parameterized by a parameter $\theta$ with density $p(\xi\mid\theta)$, the posterior distribution of $\theta$ is obtained via Bayes' rule:
\begin{align}\label{eq:bayes-rule}
    \pi(\theta\mid \xi_1,\ldots,\xi_n) \propto \pi(\theta)\prod_{i=1}^n p(\xi_i\mid \theta),
\end{align}
where $\pi(\theta)$ denotes the prior density of $\theta$. 
When $\xi$ is assumed to be elliptically distributed, the density function of $\xi$ admits the form $
    p(x) \propto g((x-\mu)^\transpose \Sigma ^{-1} (x-\mu)),$
where $g:\R_+\to \R_+$ is the density generator of the radial random variable $r$.
However, under the general elliptical distribution assumption, there is no access to an explicit expression of $g$, and thus the classical Bayes update rule~\eqref{eq:bayes-rule} 
is not applicable. 
Indeed, the expression of density generator is determined by the distribution of radius $r$ (see~\citet[Theorem 2.9]{ref:fang2018symmetric} for rigorous claim), and the latter is generally unknown for elliptical distribution. 
To deal with the issue, in the remaining parts of this section, we first recall a general Gibbs inference framework for optimal solutions of optimization models in Section~\ref{sec:gibbs-loss}, and then propose a novel estimating-equation-based Gibbs inference framework. We note that the location vector $\mu$ and shape matrix $\Sigma$ of an elliptical distribution can be associated to a pair of estimating equations, and thus based on the framework, admits a consistent Bayesian posterior update rule.

\subsection{Gibbs Posterior Inference Framework}\label{sec:bayes-infer}

The so-called {\it Gibbs posterior inference} is a posterior updating framework in which the parameter of interest is defined as the minimizer of a loss function. The classical Bayesian framework~\eqref{eq:bayes-rule} can be viewed as a special case of this approach when the loss function is chosen to be the negative log-likelihood. In this section, we propose a novel Gibbs-like posterior updating framework in which the parameter of interest is instead characterized as the root of an estimating equation, rather than as the minimizer of a loss function.
We emphasize that the proposed framework broadens the range of scenarios in which Gibbs posteriors can be applied. In particular, there exist settings in which a well-defined loss function is difficult or impossible to specify, while a valid estimating equation can still be naturally formulated.

The loss-function-based Gibbs posterior was first proposed in~\cite{ref:zhang2006information}, and discussed in~\cite{ref:bissiri2010bayesian, ref:bissiri2012bayesian, ref:bissiri2016general}. Theoretical guarantees for Gibbs posterior inference have been established in~\cite{ref:syring2023gibbs, ref:martin2022direct}.
In the remainder of this section, we first review the basic idea of Gibbs posterior inference. We then explain why the loss-function-based framework is inadequate for parameter estimation in elliptical distributions, and subsequently introduce our estimation-equation-based framework as a practical alternative.

\subsubsection{Gibbs Posterior Inference Based on Loss Function}\label{sec:gibbs-loss}

Let $\xi$ be a random vector in $\R^p$. Suppose that the parameter of interest $\theta_0$ can be written as the unique minimizer of an expected loss, i.e., 
\begin{align}
    \label{eq:theta_0}
    \theta_0 = \arg\min_{\theta\in \Theta} L(\theta) \triangleq \E[l(\theta,\xi)],
\end{align}
where the expectation is taken with respect to the distribution of $\xi$, and $l:\Theta\times \R^p\to \R$ is a loss function. For example, if $\Theta=\R^p$ and $l(\theta,\xi) = \|\theta-\xi\|^2$, then $\theta_0$ coincides with the mean vector of $\xi$; if the density function of $\xi$ is known to be $f(\xi,\theta')$ for some $\theta'\in \Theta$ and $l(\theta,\xi)=-\log f(\xi,\theta)$, i.e., the negative log-likelihood function, then the minimizer of $L(\theta)$ is $\theta_0=\theta'$ 
because
$L(\theta)-L(\theta')
=\dd_{\text{KL}}\!\left(f(\cdot,\theta'),f(\cdot,\theta)\right)\geq 0$, and equality holds when $\theta=\theta'$. 
Here and throughout the paper, $\dd_{\mathrm{KL}}$ denotes the Kullback--Leibler (KL) divergence between two probability distributions or their density functions.
Now let $\pi(\theta)$ be our prior beliefs about $\theta_0$, and $\xi_1,\ldots,\xi_n$ be i.i.d. samples of $\xi$. Our goal is to find a valid framework to update our beliefs to a posterior distribution $\pi(\theta\mid \xi_1,\ldots,\xi_n)$ based on the given samples, so that the posterior distribution can encode the information from both our prior and samples. One persuasive way to motivate the idea is to consider the following optimization problem:
\begin{align}\label{prob:nu-kl}
    \min_{\nu\in \mathscr{P}(\Theta)} \int_{\theta\in \Theta} \left[\frac{1}{n}\sum_{i=1}^n l(\theta,\xi_i)\right] \nu(\theta)\diff \theta + \frac{1}{\eta n}\dd_{\text{KL}}(\nu, \pi),
\end{align}
where $\mathscr{P}(\Theta)\triangleq\left\{\nu:\Theta\to \R_+\mid \int_{\theta\in\Theta}\nu(\theta)\diff\theta=1\right\}$ is the set of all density functions on $\Theta$.

The first term in the objective of~\eqref{prob:nu-kl}, which is essentially an integral of empirical loss $\frac{1}{n}\sum_{i=1}^n l(\theta,\xi_i)$, encodes the information from samples. 
The second term in~\eqref{prob:nu-kl} is used as a penalty term that prevents the optimal $\nu$ from being too far away from the prior $\pi$, and $\eta>0$ is a parameter that adjusts the strength of this penalty.
When the sample size is relatively small, the optimal solution of~\eqref{prob:nu-kl} balances the information from samples and prior belief $\pi$. 
When the sample size becomes large, the second term vanishes and 
$
    \int_{\theta\in \Theta} \left[\frac{1}{n}\sum_{i=1}^n l(\theta,\xi_i)\right] \nu(\theta)\diff \theta \approx \int_{\theta\in \Theta} L(\theta) \nu(\theta) \diff \theta.
$
Consequently, the minimizer of \eqref{prob:nu-kl} approaches the Dirac measure at $\theta=\theta_0$, denoted $D_{\theta_0}$, because $\theta=\theta_0$ uniquely solves $\min_{\theta\in\Theta} L(\theta)$
and
\bgeqn 
    \label{eq:theta_0-a}
    \theta_0 = \arg\min_{\theta\in \Theta} L(\theta) 
    \quad \Longleftrightarrow \quad 
    D_{\theta_0} =\arg\min_{\nu\in \mathscr{P}(\Theta)}
\mathbb{E}_\nu[L(\theta)].
\edeqn 
See \citet[Theorem 1]{ref:martin2022direct} for a rigorous claim of this observation.

Due to the 
specific structure of KL divergence, it can be shown \citep[see, e.g.,][]{ref:zhang2006information, ref:bissiri2016general} that the unique optimal solution of~\eqref{prob:nu-kl} admits a closed form:
\begin{align}\label{eq:gibbs-post}
    \nu^\star(\theta\mid \xi_1,\ldots,\xi_n) = \frac{\exp\left(-\eta\sum_{i=1}^nl(\theta,\xi_i)\right)\pi(\theta)}{\int_{\theta \in\Theta}\exp\left(-\eta\sum_{i=1}^nl(\theta,\xi_i)\right)\pi(\theta)\diff \theta},
\end{align}
which is called {\it Gibbs posterior} in the literature.
Intuitively,~\eqref{eq:gibbs-post} assigns more posterior density to values of $\theta$ 
where the empirical loss
$\sum_{i=1}^n l(\theta,\xi_i)$ is small, that is, to $\theta$ that provides a better fit to the observed samples. 
From the structure \eqref{eq:gibbs-post}, we can also see why the parameter $\eta>0$ is called the learning rate of the Gibbs posterior in the literature. Indeed, it visualizes the effect of the penalty term in \eqref{prob:nu-kl}. When $\eta$ is close to $0$, the penalty term dominates the objective function and the optimal solution remains close to the prior distribution $\pi(\theta)$; in this case, the posterior learns little from the samples. When $\eta$ is set very large (tends to $+\infty$), the penalty term vanishes and again, 
the optimal solution of \eqref{prob:nu-kl} 
approaches the Dirac measure at the optimal solution of SAA problem, for the same reason as in \eqref{eq:theta_0-a}.
In this case, the posterior is driven almost entirely by the samples. Note that when the loss function $l$ is set to be negative log-likelihood function and $\eta=1$,~\eqref{eq:gibbs-post} reduces to the classical Bayesian update formula~\eqref{eq:bayes-rule}. 

The Gibbs posterior generalizes the classical Bayesian posterior and is particularly useful when a precise statistical model for the parameter of interest is unavailable, for example, when the form of the underlying density function is unknown. 
However, for parameter estimation in elliptical distributions (and in many other practical settings), to figure out a valid loss function of the form in~\eqref{eq:theta_0} may be impossible. For example, to estimate the location parameter of an elliptical distribution, one may consider the loss function $L(\mu)=\E[\|\mu-\xi\|]$. As discussed in Remark~\ref{rmk:ellip}, however, the expectation of $\|\mu-\xi\|$ may not exist, even though the location parameter $\mu_0$ remains well-defined. 
A similar difficulty arises in the estimation of the shape parameter when the second moment of $\xi$ does not exist. This motivates us to propose a novel Gibbs posterior inference framework, which does not depend on a loss function, but on a general estimating equation.

\subsubsection{Gibbs Posterior Inference Based on Estimating Equation}
\label{sec:gibbs-est-eq}


Let the parameter space $\Theta$ be a subset of $E$, where $E$ is a finite-dimensional linear space equipped with a norm $\|\cdot\|$. 
Suppose 
that the parameter of interest $\theta_0\in\Theta$ is the unique root of the following estimating equation:
\begin{align}
\label{eq:Phi-ERENJ}
    \Phi(\theta) \triangleq\E[\varphi(\theta,\xi)]=0,
\end{align}
where $\varphi(\theta,\xi):\Theta\times \R^p\to \Theta$ is an estimation function. Under suitable regularity conditions, the above estimation-equation-based approach recovers the loss function based approach in~\eqref{eq:theta_0}. In particular, one can choose $\varphi(\theta,\xi)=\nabla_1 l(\theta,\xi)$, where $\nabla_1 l(\theta,\xi)$ denotes the gradient of $l$ with respect to its first argument. 
When the expectation and differentiation are interchangeable,
the unique root of~\eqref{eq:Phi-ERENJ} coincides with the optimal solution of~\eqref{eq:theta_0}. Moreover, there are settings in which the loss-function-based approach is not applicable, whereas the estimation-equation framework in~\eqref{eq:Phi-ERENJ} remains valid.

Consider, for example, the estimation of the location vector of an elliptically distributed random variable $\xi$. Although the first moment of $\xi$ may not exist, and thus the loss function $L(\mu)=\E[\|\xi-\mu\|]$ is not 
well-defined, 
we can define the estimation function
\begin{align}\label{eq:varphi-mu}
    \varphi_1(\mu,\xi) = \frac{\xi-\mu}{\|\xi-\mu\|}. 
\end{align}
The expectation of $\varphi(\mu,\xi)$ always exists, since $\|\varphi(\mu,\xi)\|=1$ for all $\mu$ and $\xi$.
The unique solution of $\E[\varphi(\mu,\xi)]=0$ is $\mu=\mu_0$, we will come back to this in  Proposition~\ref{prop:est-eq-mu} in Section~\ref{sec:est-eq-mu}.
Now let us focus on the construction of Gibbs posteriors based on estimating equation~\eqref{eq:Phi-ERENJ}.

Let $\xi_1,\ldots,\xi_n$ be $n$ 
i.i.d.~samples of $\xi$.
We begin by considering point estimation of $\theta_0$.
A point estimator of $\theta_0$ can be derived by solving the sample average approximation version of~\eqref{eq:Phi-ERENJ}:
\bgeqn 
\label{eq:Phi_n-ERENJ}
     \Phi_n(\theta)\triangleq \frac{1}{n}\sum_{i=1}^n \varphi(\theta,\xi_i) = 0.
\edeqn 
The root of the above equation  equals to the optimal solution of the following optimization problem:
 \bgeqn 
    \label{prob:min-varphi-saa}
    \min_{\theta\in\Theta}~~\left\| \Phi_n(\theta)\right\|,
\edeqn 
where 
$ \|\Phi_n(\theta)\|$
serves as the objective function that penalizes $\theta$ when it does not fit the samples. As $n$ tends to infinity, $\Phi_n(\theta)$ converges to $\Phi(\theta)$ and when the convergence is uniform, the optimal solution of~\eqref{prob:min-varphi-saa} converges to $\theta_0$, see e.g.~\cite{homem2008rates,xu2010uniform}.

Our focus here is not on point estimation, rather we concentrate on Bayesian approach. 
Specifically, we assign a probability
distribution on $\theta$, called prior distribution, based on available information or subjective judgment before observing samples of $\xi$.
The prior distribution reflects our 
belief about the likelihood of $\theta$ taking a particular value amid ambiguity of $\theta_0$. Let $\Pi$ denote the prior with density $\pi$. 
The next definition describes how the posterior distribution/density is updated under the Gibbs' framework.



\begin{definition}[Gibbs posterior and its 
consistency]\label{def:gibbs}
    Let the parameter of interest $\theta_0 \in \Theta$ be the unique solution of the estimating equation \eqref{eq:Phi-ERENJ}.
    Let $\Pi$ and $\pi$ be the prior distribution of the parameter and its density, respectively.
    The density of Gibbs posterior distribution with the learning rate $\eta>0$ based on $n$ samples $\xi_1,\ldots,\xi_n$, 
    written $\Pi_n^\eta(\cdot\mid \xi_1,\ldots,\xi_n)$, is defined as:
    \begin{align}\label{eq:post-eq-based}
        \pi_n^\eta(\theta\mid \xi_1,\ldots,\xi_n) = \frac{\exp\left(-\eta
        n\|\Phi_n(\theta)\|
        \right)\pi(\theta)}{\int_{\theta \in\Theta}
        \exp\left(-\eta
        n\|\Phi_n(\theta)\|
        \right)\pi(\theta)\diff \theta}.
    \end{align}
    Moreover, $\Pi_n^\eta$ is said to be strongly consistent if for any $\delta>0$, 
    \begin{align}
        \label{eq:thm1-theta-theta_0}
            \Pi_n^\eta(\{\theta: \|\theta-\theta_0\|\leq \delta\}\mid \xi_1,\ldots,\xi_n) \asto 1~\text{as}~n\to\infty.
    \end{align}
\end{definition}
From \eqref{eq:post-eq-based}, we observe that $ \frac{\exp\!\left(-\eta n\|\Phi_n(\theta)\|\right)} {\int_{\Theta} \exp\!\left(-\eta n\|\Phi_n(\theta)\|\right)\pi(\theta)\,\diff\theta}$ acts as a weighting factor. The term $\left\|\sum_{i=1}^n \varphi(\theta,\xi_i)\right\|$ can be viewed as a ``loss function'' that penalizes the unfitness of $\theta$, and the posterior distribution assigns more density when the ``loss'' is small.
Specifically, parameter values with smaller residuals $\|\Phi_n(\theta)\|$ receive larger weights and therefore contribute more heavily to the posterior distribution. Conversely, parameter values with larger residuals are assigned smaller weights. This behavior is consistent with \eqref{eq:Phi_n-ERENJ}, which suggests that $\theta$ tends to be closer to the estimator when the residual $\|\Phi_n(\theta)\|$ is small.
The learning rate $\eta$ is a hyperparameter of the Gibbs posterior framework as introduced in Section~\ref{sec:gibbs-loss}.
It turns out that any positive learning rate leads to a consistent posterior distribution. The following theorem formalizes the above arguments: under regularity conditions, the posterior distribution is consistent for any learning rate $\eta>0$. 

\begin{theorem}\label{thm:consist-gibbs}
    Let
    $\Phi$ and $\Phi_n$ be defined as in 
    \eqref{eq:Phi-ERENJ} and \eqref{eq:Phi_n-ERENJ} respectively.
    Let $\theta_0$ be 
    a 
    solution to~\eqref{eq:Phi-ERENJ}. 
   Assume:  (a) for any $\theta\in\Theta$, 
   $\E[
   \|\varphi(\theta,\xi)\|]< \infty$,
        (b) for any $\delta>0$,
        \begin{align}
            \liminf_{n\to\infty} \inf_{\theta\in\Theta:\|\theta-\theta_0\|\geq \delta}\|\Phi_n(\theta)\| > 0~~\text{a.s.},
        \end{align}
        (c) $\Pi(\{\theta: \|\Phi(\theta)\|\leq \delta\}) > 0$ for any $\delta>0$. 
    Then $\Pi_n^\eta(\cdot\mid \xi_1,\ldots,\xi_n)$ is strongly consistent in the sense of \eqref{eq:thm1-theta-theta_0} for any learning rate $\eta>0$. Moreover, for any $\delta>0$, there exists a constant $c>0$, such that 
    \begin{align}
        \label{ineq:as-event}
        \Pi_n^\eta\Big(\{\theta: \|\theta-\theta_0\|\geq \delta\}\mid \xi_1,\ldots,\xi_n\Big) \leq \exp(-c\eta n)\quad \text{a.s. eventually.}
    \end{align}
\end{theorem}

Condition (b) requires that the residual function $\|\Phi_n(\theta)\|$ remain uniformly bounded away from zero whenever $\theta$ is separated from the true parameter $\theta_0$ by a positive distance. In other words, parameter values outside any neighborhood of $\theta_0$ cannot asymptotically yield arbitrarily small residuals. This property is consistent with \eqref{eq:Phi_n-ERENJ}, where $\theta_0$ is characterized by the vanishing of the residual. 
A sufficient condition for condition (b) is to assume that $\sup_{\theta\in \Theta} \|\Phi_n(\theta)-\Phi(\theta)\|\asto 0$ and $\inf_{\theta\in\Theta:\|\theta-\theta_0\|\geq \delta}\|\Phi(\theta)\|>0$ for any $\delta>0$, see \citet[Theorem~1]{ref:martin2022direct} for the loss function based case. 
However, the uniform law of large numbers over the whole parameter space can be difficult to verify, especially when the parameter space is not compact. It turns out that condition (b) is less restrictive and sufficient for the strong consistency, and will be useful in the proof of Theorem \ref{thm:gibbs-mu-sigma-consistency} in the next section.
Condition (c) requires that the prior assign positive probability to every neighborhood of the solution set characterized by small values of $\|\Phi(\theta)\|$. In particular, the support of the prior must contain the true parameter $\theta_0$. This assumption is analogous to the standard prior-support condition in Bayesian consistency theory and is consistent with \eqref{eq:Phi-ERENJ}.


\subsection{Posterior Distribution of $(\mu,\Sigma)$}\label{sec:gibbs-ellip}

Now we are ready to construct posterior distributions for the location vector $\mu_0$ and the shape matrix $\Sigma_0$ of an elliptically distributed random vector $\xi$, based on $n$ i.i.d. samples $\xi_1,\ldots,\xi_n$. In the following, we first present candidate estimating equations for point estimation of $\mu_0$ and $\Sigma_0$, respectively. We then apply the estimating-equation-based Gibbs-like framework developed in Section~\ref{sec:gibbs-est-eq} to construct the corresponding posterior distributions. Before turning to technical details, we shall highlight the key intuition: although the distribution of $\xi$ depends on the (generally unknown) distribution of $r$, and the first and the second moments of $\xi$ may not exist,  the location and shape parameters $\mu_0,\Sigma_0$ remain identifiable from samples of $\xi$. This is because $\mu_0$ and $\Sigma_0$ contain only the locational and directional information of $\xi$, and the effect of $r$ can be removed by constructing estimating equations that normalize out the radial magnitude.

\subsubsection{Estimating Equation for the Location Parameter}\label{sec:est-eq-mu}

As mentioned in~\eqref{eq:varphi-mu}, even when $\E[\|\xi\|]$ does not exist, the location parameter $\mu_0$ remains well-defined as the unique solution to
\[
    \E\left[\frac{\xi-\mu}{\|\xi-\mu\|}\right]=0.
\]
In fact, for an elliptically distributed random vector, $\mu_0$ can be interpreted as the center of symmetry of the distribution, and hence, the direction of the normalized random vector $(\xi-\mu_0)/\|\xi-\mu_0\|$ is balanced in expectation, i.e., $\E[(\xi-\mu_0)/\|\xi-\mu_0\|]=0$. The following proposition states the argument formally.

\begin{proposition}\label{prop:est-eq-mu}
    Let $\xi$ be an elliptically distributed random vector 
    defined as in \eqref{eq:EC-stochastic-rep}. 
    Define the estimating function by
    \bgeqn 
    \label{eq:phi_1-EDRO-Renjie}
        \varphi_1(\mu,\xi)\Let \frac{\xi-\mu}{\|\xi-\mu\|}.
    \edeqn 
    Then the following assertions hold for $p\geq 2$.
    \begin{enumerate}[label=(\roman*)]
        \item The unique solution of $\E[\varphi_1(\mu,\xi)]=0$ is given by $\mu=\mu_0$.
        \item For any $\delta>0$, $\inf_{\|\mu-\mu_0\|\geq \delta} \|\E[\varphi_1(\mu,\xi)]\| > 0$.
    \end{enumerate}
    
\end{proposition}

\subsubsection{Estimating Equation for the Shape Matrix}

In the following, we propose an estimating equation whose unique root is the shape matrix of the elliptical distribution. To motivate the construction of the estimating equation, we first consider the normalized random vector 
\[
    s \Let \frac{\xi-\mu_0}{\|\xi-\mu_0\|} = \frac{r\Sigma_0^{1/2}u}{\|r\Sigma_0^{1/2}u\|}=\frac{\Sigma_0^{1/2}u}{\|\Sigma_0^{1/2}u\|},
\]
which removes the impact of the radial component $r$ and captures the directional information of $\xi-\mu_0$. It is well-known that $s$ follows an angular central Gaussian (ACG) distribution with shape parameter $\Sigma_0$, whose density function is in the form $
    p(x\mid\Sigma_0) \propto \det(\Sigma_0)^{-1/2}(x^\transpose\Sigma_0^{-1}x)^{-p/2},~~x\in \mathbb{U}^{p-1}.$
Therefore, given $n$ i.i.d. samples of $s$, i.e., $s_1,\ldots,s_n$, the maximum likelihood estimator of $\Sigma_0$ is given by the optimal solution of the following problem:
\begin{align}\label{prob:acg-loglikeli-1}
    \min_{\Sigma\in \PD}\;\; \log\det(\Sigma)+\frac{p}{n}\sum_{i=1}^n \log\left(s_i^\transpose\Sigma^{-1}s_i\right).
\end{align}
It turns out that the optimal solution of~\eqref{prob:acg-loglikeli-1} coincides with Tyler's shape matrix estimator~\citep{ref:tyler1987distribution,ref:tyler1987statistical}, which is defined as the fixed point solution $\Sigma'$ to the equation:
\begin{align}\label{eq:tyler}
    \Sigma' = \frac{p}{n}\sum_{i=1}^n \frac{s_is_i^\transpose}{s_i^\transpose{\Sigma'}^{-1}s_i}. 
\end{align}
Note that, as sample size $n$ 
goes to infinity, the summation $\frac{1}{n}\sum_{i=1}^n\frac{s_is_i^\transpose}{s_i^\transpose\Sigma^{-1}s_i}$ tends to the expectation $\E\left[\frac{ss^\transpose}{s^\transpose \Sigma^{-1}s}\right]$. 
This fact gives us a natural candidate for the estimating function $\Sigma_0$ in our elliptical distribution setting, that is,
\begin{align}
\label{eq:phi_2-EDRO-Renjie}    \varphi_2(\mu_0,\Sigma,\xi)\Let \Sigma/p - \frac{(\xi-\mu_0)(\xi-\mu_0)^\transpose}{(\xi-\mu_0)^\transpose\Sigma^{-1}(\xi-\mu_0)}.
\end{align}
Recall that, as mentioned in Remark \ref{rmk:scale}, the scale of the shape matrix is not identifiable without a scale normalization. In particular, the following proposition states that, if the location vector is known to be $\mu_0$, then any root of $\E[\varphi_2(\mu_0,\Sigma,\xi)]=0$ must coincide with $\Sigma_0$ up to scale.

\begin{proposition}\label{prop:est-eq-sigma}
    Let $\xi$ be an elliptically distributed random vector 
    defined as in \eqref{eq:EC-stochastic-rep}. 
    Let the estimating function $\varphi_2$ be defined as in \eqref{eq:phi_2-EDRO-Renjie}. 
    Then 
    $\E[\varphi_2(\mu_0,\Sigma, \xi)]=0$ if and only if $\Sigma=c\Sigma_0$ for some constant $c\in\R_{++}$.  
    
\end{proposition}

\subsubsection{Gibbs Posterior of $(\mu,\Sigma)$}

To specialize the framework introduced in Section \ref{sec:gibbs-est-eq} to the estimation of $\mu$ and $\Sigma$, we specify the underlying linear space $E$ and parameter space $\Theta$ in the first place.
Let $E \Let \mathbb{R}^p \times \mathbb{S}^p$ be the linear space in which the parameter pair $(\mu,\Sigma)$ is embedded, and equip it with the norm
\[
    \|(\mu,\Sigma)\|
    \triangleq
    \|\mu\|_2 + \|\Sigma\|_{\text{F}}.
\]
Since the shape matrix is only identifiable up to scale, here we fix the scale of it by requiring $\tr(\Sigma)=p$ and consider the parameter space $\Theta \Let \mathbb{R}^p \times \{\Sigma\in\mathbb{S}^p_{++}:\tr(\Sigma)=p\} \subset E$. As a direct consequence of Propositions~\ref{prop:est-eq-mu} and \ref{prop:est-eq-sigma}, the following proposition provides an estimating equation for the parameter pair $(\mu,\Sigma)$.

\begin{proposition}\label{prop:est-eq-mu-Sigma}
 Let $\xi$ be an elliptically distributed random vector 
    defined as in \eqref{eq:EC-stochastic-rep}. 
 Define the estimating function $\varphi:\Theta\times \R^p\to E$ by
    \bgeqn 
    \label{eq:phi-gibbs}
        \varphi(\mu,\Sigma,\xi) \Let \begin{pmatrix}
            \varphi_1(\mu,\xi)\\
            \varphi_2(\mu,\Sigma,\xi)
        \end{pmatrix},
    \edeqn 
    where $\varphi_1$ and $\varphi_2$ are defined as in Propositions~\ref{prop:est-eq-mu} and \ref{prop:est-eq-sigma}. 
    Then for $p \geq 2$,
    \bgeqn 
    \label{eq:expected-phi-ERENJ}
   \Phi(\mu,\Sigma)\Let  \E[\varphi(\mu,\Sigma,\xi)]=0
    \edeqn 
    has a unique solution $(\mu_0,\Sigma_0)$
    in $\Theta$.
\end{proposition}

Now we are ready to propose the estimating-equation-based Gibbs posterior for $(\mu,\Sigma)$. Let $\xi_1,\ldots,\xi_n$ be $n$ i.i.d. samples of $\xi$. Then the Gibbs posterior distribution (see Definition \ref{def:gibbs}) induced by the estimating function $\varphi(\mu,\Sigma,\xi)$ in Proposition \ref{prop:est-eq-mu-Sigma} is of density 
\begin{align}\label{eq:post-eq-based-1}
        \pi_n^\eta(\mu,\Sigma \mid \xi_1,\ldots,\xi_n) = \frac{\exp\left(-\eta\left\|\sum_{i=1}^n \varphi(\mu,\Sigma ,\xi_i)\right\|\right)\pi(\mu,\Sigma )}{\int_{\mu,\Sigma  \in\Theta}\exp\left(-\eta\left\|\sum_{i=1}^n \varphi(\mu,\Sigma ,\xi_i)\right\|\right)\pi(\mu,\Sigma)\diff (\mu,\Sigma)}.
    \end{align}
Moreover, the following Theorem shows that the above posterior distribution is consistent.

\begin{theorem}\label{thm:gibbs-mu-sigma-consistency}
    Let $\xi$ be an elliptically distributed random vector     defined as in \eqref{eq:EC-stochastic-rep}. Let $\varphi(\mu,\Sigma,\xi)$ be the estimating function defined as in \eqref{eq:phi-gibbs} and $\Pi_n(\cdot\mid\xi_1,\ldots,\xi_n)$ be the Gibbs posterior distribution based on $\varphi(\mu,\Sigma,\xi)$. Let $p\geq 2$. Assume that the prior distribution $\Pi$ satisfies that $\Pi(\{(\mu,\Sigma): \|\Phi(\mu,\Sigma)\|\leq \delta\}) > 0$ for any $\delta>0$. Then $\Pi_n^\eta$ is strongly consistent for any $\eta>0$, 
    that is, for any $\delta > 0$, 
    \begin{equation}
        \Pi_n^\eta\big(\{(\mu,\Sigma)\in\Theta:\|(\mu,\Sigma)-(\mu_0,\Sigma_0)\|\leq \delta\}\mid \xi_1,\ldots,\xi_n\big)
        \asto 1~~\text{as $n\to \infty$.}
    \end{equation}
    Moreover, for any $\delta>0$, there exists a constant $c>0$, such that 
    \begin{align}
        \label{ineq:as-event'}
        \Pi_n^\eta\Big(\{(\mu,\Sigma): \|(\mu,\Sigma)-(\mu_0,\Sigma_0)\|\geq \delta\}\mid \xi_1,\ldots,\xi_n\Big) \leq \exp(-c\eta n)\quad \text{a.s. eventually.}
    \end{align}
\end{theorem}

\section{Computational method to solve CBR--DRO}\label{sec:comp}

In this section, we develop efficient computational procedures for solving \eqref{prob:CBR--DRO}:
\begin{align}
    \min_{x\in \mathcal{X}} \rho_{(\mu,\Sigma)\sim \Pi_n}\left(F(x,\mu,\Sigma)\right),
\end{align}
where 
is $F$ the worst-case cost function of the inner DRO problem defined as in \eqref{def:F}.
We 
begin by establishing
standard convexity and sample average approximation results
in Sections \ref{sec:cvx-cbrdro} and \ref{sec:saa} and then propose a cutting plane method to solve \eqref{prob:CBR--DRO} in Section \ref{sec:cutting-alg}.


\subsection{Convexity of \eqref{prob:CBR--DRO}} \label{sec:cvx-cbrdro}

Throughout this section, we make the following assumption.


\begin{assumption}\label{ass:convex}
    $\mathcal{X}\subseteq\R^d$ is nonempty, compact, and convex;
    for every $y\in\R^p$, $f(\cdot,y)$ is finite and convex
    on $\R^d$; and $\rho:L^q\to\R$ is a law-invariant
    convex risk measure.
\end{assumption}

To ease the 
exposition, 
let ${v}(x) \triangleq \rho_{(\mu,\Sigma)\sim \Pi_n}\left(F(x,\mu,\Sigma)\right)$. The
next proposition 
ensures convexity of problem \eqref{prob:CBR--DRO}.

\begin{proposition}\label{prop:convexity}
    Under Assumption~\ref{ass:convex}, \eqref{prob:CBR--DRO} is a convex optimization problem in $x$. Moreover, if $\E_{(\mu,\Sigma)\sim \Pi_n}[|F(x,\mu,\Sigma)|^q]<\infty$ for all $x\in \R^d$, then ${v}(\cdot)$
    is Lipschitz continuous
    over $X$.
\end{proposition}

    To ensure that ${v}(x)$ is well-defined on $\mathcal{X}$, i.e., 
    $\E_{(\mu,\Sigma)\sim \Pi_n}[|F(x,\mu,\Sigma)|^q]<\infty$ for $x\in \R^p$, the following provides a sufficient condition. Suppose that for each $x\in\R^d$, there exist constants $C_x>0$ and $k\in \R_+$ such that $
        |f(x,y)| \le C_x(1+\|y\|^k), \forall\, y\in\R^p,$
    and that the ambiguity set $\mathcal{M}(\mu,\Sigma)$ only contains distributions with finite $k$-th moment, i.e.,
    $
        \sup_{Q\in\mathcal{M}(\mu,\Sigma)} \E_{r\sim Q}[|r|^k] \le m_k(\mu,\Sigma) <\infty.$
    Then $|F(x,\mu,\Sigma)| \le C_x'(1+\|\mu\|^k + \|\Sigma\|_{\F}^{k/2} m_k(\mu,\Sigma))$ for some constant $C_x'>0$, and the integrability condition $\E_{(\mu,\Sigma)\sim\Pi_n}[|F(x,\mu,\Sigma)|^q]<\infty$ reduces to a moment condition on $(\mu,\Sigma)$ under the posterior $\Pi_n$, i.e., 
    $\E_{(\mu,\Sigma)\sim \Pi_n}[\|\mu\|^{kq}+\|\Sigma\|_{\F}^{kq/2}m_k(\mu,\Sigma)^q] < \infty.$

\subsection{Sample Average Approximation}\label{sec:saa}

Although Proposition~\ref{prop:convexity} shows that \eqref{prob:CBR--DRO} is a convex optimization problem, there remain two main difficulties to solve the problem numerically: (a) the objective function ${v}(x)$ is a functional of the posterior distribution $\Pi_n$, and no closed-form of ${v}(x)$ is available, and (b) for each fixed $(\mu,\Sigma)$, $F(x,\mu,\Sigma)$ itself is an optimal value function, and evaluation of it requires one to solve a programming problem. Issue (b) may be tackled by introducing a first-order oracle for $F$. See the next section for details. To deal with (a), we turn to optimize a sample average approximation of ${v}(x)$:
\begin{align}
\label{eq:saa}
    \hat{v}_N(x) \;\Let\; \rho\left(\bigl\{F(x,\mu_i,\Sigma_i)\bigr\}_{i=1}^N\right),
\end{align}
where $\{(\mu_i,\Sigma_i)\}_{i=1}^N$ are i.i.d. samples drawn from $\Pi_n$. In the rest of this paper, $\rho\left(\bigl\{F(x,\mu_i,\Sigma_i)\bigr\}_{i=1}^N\right)$ stands for the risk of the random variable that is uniformly distributed over $\bigl\{F(x,\mu_i,\Sigma_i)\bigr\}_{i=1}^N$. In the following theorem, we provide a uniform convergence theorem for the sample average approximation in \eqref{eq:saa}, and thereby justify its use as a surrogate for ${v}(x)$ in the numerical framework.

\begin{theorem}\label{thm:unif-conv}
    Let $\rho:L^q\to \R$ be a law-invariant convex risk measure and $\xi\in \R^p$ be a random vector. Let $\mathcal{X}\subset \R^d$ be a compact set, and $g:\mathcal{X}\times \R^p\to \R$ be a function with $g(x,\xi)\in L^q$ for all $x\in \mathcal{X}$. Let $\xi_1,\ldots,\xi_N$ be i.i.d. copies of $\xi$.
    Let ${v}(x) \Let \rho(g(x,\xi))$ and $\widehat{{v}}_N(x)\Let \rho(\{g(x,\xi_i)\}_{i=1}^N)$. 
Assume:
(a) $\rho$ is $\kappa_1$-Lipschitz continuous for some constant $\kappa_1>0$, i.e., $
            |\rho(X_1)-\rho(X_2)|
            \le
            \kappa_1\|X_1-X_2\|_{L^q},
            \forall X_1,X_2\in \operatorname{dom}(\rho).$
    (b) $g(x,\xi)$ is continuous in $x$ almost surely, and there exists $\kappa_2(\xi)\in L^q$ such that for any $\xi\in\R^p$, $
            \left|g(x,\xi)-g(y,\xi)\right| \leq \kappa_2(\xi)\|x-y\|, \forall x,y\in\mathcal{X}.$
    Then 
        \bgeqn 
            \sup_{x\in\mathcal{X}} \left|{v}(x)-\widehat{{v}}_N(x)\right| \asto 0.
        \edeqn 
\end{theorem}

Condition (a) in Theorem~\ref{thm:unif-conv} requires that $\rho$ is Lipschitz continuous.
Several risk measures in the literature satisfy the condition. For example, it's known that CVaR, or more generally, spectral risk measure~\citep{ref:pichler2013natural}, satisfy condition (a). Moreover, it is also easy to verify that
an optimized certainty equivalent risk measure $
    \rho_\ell(X)=\inf_{t\in\R}\{t+\E[\ell(X-t)]\} $
satisfies condition (a) if $\ell$ is globally Lipschitz continuous. See \citet[Lemma 12]{ref:prashanth2022wasserstein} for a proof.

\subsection{Cutting Plane Algorithm}\label{sec:cutting-alg}

In this section, we propose the cutting plane method that solves \eqref{prob:CBR--DRO}. The algorithm is based on a first-order oracle for the inner value function $F(x,\mu,\Sigma)$, 
which is specified by the following assumption.

\begin{assumption}[First-order oracle for $F$]\label{ass:oracle}
    A first-order oracle for $F$ is available, that is, for fixed $\mu\in\R^p$ and $\Sigma\in \PD$, the value of $F(x, \mu, \Sigma)$ and the gradient or a subgradient of $F$ with respect to $x$ can be computed for any $x\in \mathcal{X}$.
\end{assumption}

We now explain how Assumption~\ref{ass:oracle} can be fulfilled in practice. Recall from \eqref{def:F} that $F(x,\mu,\Sigma)$ is defined as the optimal value of the inner distributionally robust problem, which is a one-dimensional DRO problem over the distribution of the radial variable $r$. The computation of the optimal value of a DRO problem has been extensively studied for commonly used ambiguity sets, including Wasserstein balls \citep{ref:mohajerin2018data,ref:gao2023distributionally}, moment-based sets \citep{ref:delage2010distributionally}, and $\phi$-divergence balls \citep{ref:ben2013robust}.

For the subgradient, let $\mathcal{Q}^*(x,\mu,\Sigma)$ denote the set of optimal solutions of \eqref{def:F}. Since $f(\cdot,y)$ is convex, $F(\cdot,\mu,\Sigma)$ is convex as a pointwise supremum of convex functions.  
Under regularity assumptions, it follows by Danskin's theorem
\begin{equation}\label{eq:F-subdiff}
    \partial_x F(x,\mu,\Sigma)
    = \mathrm{conv}\Big\{
        \E_{r\sim Q^*,\, u\sim \mathrm{Unif}(\mathbb{U}^{p-1})}
        \big[\nabla_x f(x,\mu+r\Sigma^{1/2}u)\big]
        : Q^* \in \mathcal{Q}^*(x,\mu,\Sigma)
    \Big\}.
\end{equation}
Thus $F(\cdot,\mu,\Sigma)$ is differentiable at $x$ whenever the worst-case distribution is unique, and may be nondifferentiable, but admits subgradients by \eqref{eq:F-subdiff} when $\mathcal{Q}^*(x,\mu,\Sigma)$ is not a singleton.

Now we are ready to propose the cutting plane method for \eqref{prob:CBR--DRO} and its convergence theorem.

\begin{algorithm}[htbp]
\caption{Cutting Plane Method for \eqref{prob:CBR--DRO}}
\label{alg:cutting-alg}
\begin{algorithmic}[1]
\State Set precision $\epsilon$ for stopping.
\State Draw i.i.d. samples  $\{(\mu_i,\Sigma_i)\}_{i=1}^N$ with the posterior distribution $\Pi_n$.
\State Choose initial point $x_0\in\mathcal{X}$, set $\mathrm{UB}_0\gets +\infty$, $\mathrm{LB}_0\gets -\infty$, $k\gets 0$
\Repeat
    \For{$i=1,\ldots,N$}
        \State Obtain the value of $F(x_k,\mu_i,\Sigma_i)$ and the subgradient $\nabla_x F(x_k,\mu_i,\Sigma_i)$ from the DRO solver. Denote them by $(F_{i,k}, g_{i,k})$.
        \State Define a cut by $
            L_{i,k}(x) \Let F_{i,k} + \la g_{i,k}, x-x_k\ra,~\forall x\in\mathcal{X}.$
    \EndFor
    \State Compute $\hat{{v}}_N(x_k) =  \rho\left(\{F(x_k,\mu_i,\Sigma_i)\}_{i=1}^N\right)$.
    \State $\mathrm{UB}_{k+1}\gets \min\big(\mathrm{UB}_k,\;\hat{{v}}_N(x_k)\big)$
    
    \State Solve
    \begin{equation}\label{eq:master}
    \begin{aligned}
        x_{k+1} \;\gets\;
        \argmin_{x\in\mathcal{X},\,\hat{F}_i\in\R}
        \;\; & \; \rho\bigl(\{\hat{F}_i\}_{i=1}^N\bigr)\\
        \text{s.t.}\;\;
        & \; \hat{F}_i \;\ge\; F_{i,j} + \bigl\la g_{i,j},x-x_j\bigr\ra,
          \quad i=1,\ldots,N,\;\; j=0,\ldots,k.
    \end{aligned}
\end{equation}
Let $\mathrm{LB}_{k+1}$ be equal to the optimal value of \eqref{eq:master}.
    \State $k\gets k+1$
\Until{$\mathrm{UB}_{k}-\mathrm{LB}_{k}\le \varepsilon$}
\Return $x^* \gets \argmin_{x_j,j=0,\ldots,k-1} \hat{{v}}_N(x_j)$
\end{algorithmic}
\end{algorithm}

\begin{theorem}[Convergence of Algorithm~\ref{alg:cutting-alg}]\label{thm:cutting-convergence}
Suppose that Assumptions~\ref{ass:convex} and~\ref{ass:oracle} hold, and that for
each $i = 1,\ldots,N$, the function $F(\cdot,\mu_i,\Sigma_i)$ is finite valued on $\R^d$. Let $\hat{v}_N^* \Let \min_{x\in\mathcal{X}} \hat{v}_N(x)$, and let $\{x_k\}$,
$\{\mathrm{UB}_k\}$, and $\{\mathrm{LB}_k\}$ be generated by
Algorithm~\ref{alg:cutting-alg} with tolerance $\varepsilon \ge 0$. Then the
following assertions hold.
\begin{enumerate}[label=(\roman*)]
    \item For every $k \ge 1$,
    $\mathrm{LB}_k \le \hat{v}_N^* \le \mathrm{UB}_k$; moreover,
    $\{\mathrm{UB}_k\}$ is nonincreasing and $\{\mathrm{LB}_k\}$ is nondecreasing.
    \item If 
    $\varepsilon=0$, then $\mathrm{UB}_k - \mathrm{LB}_k \to 0$ as $k\to\infty$. Consequently,
    $\mathrm{UB}_k \downarrow \hat{v}_N^*$, $\mathrm{LB}_k \uparrow \hat{v}_N^*$,
    and every accumulation point of $\{x_k\}$ is an optimal solution of
    $\min_{x\in\mathcal{X}} \hat{v}_N(x)$.
    \item If $\varepsilon > 0$, then Algorithm~\ref{alg:cutting-alg} terminates in
    finitely many iterations, and the returned point $x^*$ satisfies
    $\hat{v}_N(x^*) \le \hat{v}_N^* + \varepsilon$.
\end{enumerate}
\end{theorem}

\begin{remark}[A tighter lower bound]\label{rmk:tighter-lb}
Algorithm~\ref{alg:cutting-alg} differs from the classical Kelley
cutting-plane method~\citep{ref:kelley1960cutting} in the way the lower bound of the objective function is successively constructed. In the classical cutting-plane method, the lower bound of the objective function $\hat{{v}}_N(x)$ is given by a piecewise-linear function
\begin{equation}\label{eq:kelley-lb}
    \hat v_N(x) \ge \max_{j=0,1,\ldots,k}
    \big\{\hat v_N(x_j) + \langle g_j^N, x - x_j\rangle\big\},
    \qquad g_j^N \in \partial \hat v_N(x_j).
\end{equation}
Algorithm~\ref{alg:cutting-alg} instead builds a scenario-wise piecewise-linear
lower bound of $F(\cdot,\mu_i,\Sigma_i)$,
\begin{equation}\label{eq:scenario-lb}
    F(x,\mu_i,\Sigma_i) \ge \hat F_i(x) \triangleq \max_{j=0,1,\ldots,k}
    \big\{F_{i,j} + \langle g_{i,j}, x - x_j\rangle\big\},
    \qquad i = 1,\ldots,N.
\end{equation}
By the monotonicity of $\rho$ (Assumption~\ref{ass:convex}), a global lower bound on $\hat{{v}}_N(x)$ is given by
\begin{equation}\label{eq:our-lb}
    \hat v_N(x) \ge \rho\big(\{\hat F_i(x)\}_{i=1}^N\big).
\end{equation}
We claim that in practice, \eqref{eq:our-lb} is tighter than \eqref{eq:kelley-lb}. To see this, define $
    h_j(x) \triangleq \rho\big(\{F_{i,j} + \langle g_{i,j}, x - x_j\rangle\}_{i=1}^N\big)$,
which is convex as the composition of 
the monotone convex functional 
$\rho$ with affine mappings.
By the monotonicity of $\rho$, 
$
    h_j(x) \leq \hat{v}_N(x),\quad\forall x\in\mathcal{X}
$
and $h_j(x_j) = \hat{v}_N(x_j)$,
which implies that $\partial h_j(x_j) \subseteq \partial \hat v_N(x_j)$. We assume that the Kelley subgradient in~\eqref{eq:kelley-lb} is taken to be $g_j^N \in \partial h_j(x_j)$, which is the natural choice in practice, since it is precisely the subgradient produced by the chain rule from the scenario data $(F_{i,j}, g_{i,j})$ used in Algorithm \ref{alg:cutting-alg}. Then, for every $x \in \mathcal{X}$,
\begin{equation}\label{eq:lb-comparison}
    \rho\big(\{\hat F_i(x)\}_{i=1}^N\big)
    \;\ge\; \max_{j=0,1,\ldots,k} h_j(x)
    \;\ge\; \max_{j=0,1,\ldots,k}
    \big\{\hat v_N(x_j) + \langle g_j^N, x - x_j\rangle\big\},
\end{equation}
where the first inequality is by the monotonicity of $\rho$
and the fact that $\max_{j=0,1,\ldots,k}\,\bigl\{F_{i,j}+\la g_{i,j},\,x-x_j\ra\bigr\} \geq F_{i,j}+\la g_{i,j},\,x-x_j\ra$, and the second inequality is by the first order condition of convexity, i.e., $
    h_j(x)\geq h_j(x_j)+\la g_j^N,x-x_j\ra = \hat v_N(x_j)+\la g_j^N,x-x_j\ra.$
The inequality \eqref{eq:lb-comparison} is typically strict. Indeed, the first inequality in \eqref{eq:lb-comparison} holds with equality at $x$ only if there is a common index $j^*$ such that $\hat F_i(x) = F_{i,j^*} + \langle g_{i,j^*}, x - x_{j^*}\rangle$
simultaneously for all scenarios $i$ contributing to the value of $\rho$, and the second holds with equality only if $h_{j^*}$ is affine. After the first few iterations, however, different scenarios are typically maximized by different cuts, so that no common index $j^*$ exists and
\eqref{eq:our-lb} is strictly tighter than \eqref{eq:kelley-lb}.
\end{remark}

\begin{remark}[Convex representation of the master problem \eqref{eq:master}]
The master problem \eqref{eq:master} in Algorithm \ref{alg:cutting-alg} 
can be reformulated as linear programming problem (if $\mathcal{X}$ is a polyhedron) or conic programming problem (if $\mathcal{X}$ is conic-representable) for many commonly used risk measure. 
In these cases, the master problem is solvable by standard
off-the-shelf solvers. For instance:
\begin{itemize}[leftmargin=1.5em]
    \item Expectation: $\E[\{\hat{F}_i\}_{i=1}^N]=\tfrac{1}{N}\sum_i \hat{F}_i$. \eqref{eq:master} becomes a linear programming problem.
    \item Conditional Value-at-Risk: $\mathrm{CVaR}_\alpha(\{\hat{F}_i\}_{i=1}^N)= \min_{s\in\R}\{s + \tfrac{1}{(1-\alpha)N}
        \sum_i [\hat F_i-s]_+\}$. \eqref{eq:master} becomes a linear programming problem with an extra scalar variable $s$.
    \item Entropic risk: $\rho(\{\hat{F}_i\}_{i=1}^N)=\tfrac{1}{\eta}\log\bigl(\tfrac{1}{N}\sum_i e^{\eta \hat F_i}\bigr)$. \eqref{eq:master} can be formulated as an
        exponential cone programming problem.
\end{itemize}
\end{remark}

\section{An Application to Portfolio Optimization}
\label{sec:num-exp}

In this section, we specialize the proposed CBR--DRO framework to quadratic-shortfall optimization problem, and further consider its downstream application to portfolio optimization. We conduct numerical experiments of CBR--DRO, and compare its performance with sample-average approximation (SAA) and distributionally robust optimization (DRO), with both synthetic and real return data.

\subsection{Quadratic-shortfall Optimization}

We consider the quadratic-shortfall optimization problem~\citep{ref:fishburn1977mean,ref:chen2011tight}
\begin{equation}\label{eq:num-loss-1}
    \min_{x\in\Delta_p}\;\;\E_{\xi\sim \mathrm{EC}(\mu_0,\Sigma_0,P_0)}\left[\big(\tau - x^\transpose \xi\big)_+^2\right],
\end{equation}
where $\Delta_p \Let \{x\in\R^p_+ : \mathbf{1}^\transpose x = 1\}$ is the $p$-dimensional simplex, $\xi\in\R^p$ is an elliptically distributed random vector, and $\tau\in\R$ is a fixed shortfall threshold. One can view $\xi$ as the random returns of $p$ assets, and $x^\top \xi$ as the portfolio return induced by the portfolio weight $x$. The quadratic-shortfall optimization problem aims to find a portfolio allocation that minimizes the expected squared shortfall of the portfolio return below a prescribed target level $\tau$.
In other words, it penalizes downside deviations from the target return, and ignores outcomes in which the portfolio return exceeds $\tau$. In this subsection, we set $p=20$ and $\tau=1$, and conduct numerical experiments with synthetic data.

\begin{table}[h]
\centering
\renewcommand{\arraystretch}{0.9}
\begin{tabular}{@{}lcccc@{}}
\toprule
Family 
& Gaussian 
& Laplace 
& $t_\nu~\text{for}~\nu\in\{4,5,10\}$ 
& slash$_4$ \\
\midrule
Tail 
& light 
& medium 
& heavier as $\nu$ smaller 
& as heavy as $t_4$ \\
\bottomrule
\end{tabular}
\caption{Distributions used as the data-generating distribution.}
\label{tab:num-dists}
\end{table}

The ground-truth elliptical family is $\xi\sim \mathrm{EC}(\mu_0,\Sigma_0,P_0)$
with location $\mu_0\in\R^p$ and shape $\Sigma_0\in\PD$, both generated randomly from
$\mu_0\sim\mathcal N(0,I_p)$ and $\Sigma_0 = p\cdot AA^\transpose/\tr(AA^\transpose)$
with $A_{ij}\stackrel{\mathrm{i.i.d.}}{\sim}\mathcal N(0,1)$. We consider six instances from elliptical family with different tail heaviness, as summarized in Table~\ref{tab:num-dists}.
For the implementation of \eqref{prob:CBR--DRO}, 
we consider two choices of the outer risk measure $\rho$: posterior $\mathrm{CVaR}_{0.9}$ and posterior expectation, and two choices of the inner ambiguity set: type-2 Wasserstein ball and KL-divergence ball. The corresponding models are denoted by CBR-CVaR-W2, CBR-CVaR-KL, CBR-Exp-W2, and CBR-Exp-KL. Algorithm~\ref{alg:cutting-alg} solves the SAA problem $\min_{x\in\Delta_p}\hat{v}_N(x)$ with $N=500$ posterior draws and $M=500$ unit-sphere samples for the inner oracle. We compare all CBR--DRO variants with the plain sample-average approximation and DRO versions of \eqref{eq:num-loss-1}.

For posterior inference, we use the priors $\mu\sim\mathcal N(0,10I_p)$ and $\Sigma=pS/\tr(S)$, where $S$ follows a Wishart distribution with $p+2$ degrees of freedom and scale matrix $I_p/p$. Samples form posterior distributions are drawn via Markov Chain Monte Carlo.
The hyperparameters in CBR--DRO and DRO are chosen via cross-validation (CV): the radius of the ambiguity set in DRO  is selected from $\{0,0.03,0.1,0.3,1,3\}$, and for CBR--DRO, we jointly select the Gibbs-posterior learning rate and the radius of the inner ambiguity set from $\{1,3,10,30,100\}\times \{0,0.01,0.03,0.1\}$.

For each distribution and each replication $r=1,\ldots,R$, we draw $n\in\{20,50\}$ training samples $\xi_1,\ldots,\xi_n$ and solve the CBR--DRO variants, plain SAA, and DROs. We evaluate the decision $\hat x_r$ obtained in each replication by its expected quadratic-shortfall loss under the true data-generating distribution: $\E_{\xi\sim\mathrm{EC}(\mu_0,\Sigma_0,P_0)}
    \left[\bigl(\tau-\hat x_r^\transpose\xi\bigr)_+^2\right].$
We note that the loss can be computed via a one-dimensional integration, since it depends only on the scalar projection $\hat x_r^\transpose\xi$, whose distribution is known for each elliptical family considered.

Over $R=30$ replications, we compute (a) the average of the losses, and (b) their empirical $\mathrm{CVaR}_{0.9}$. We note that $\mathrm{CVaR}_{0.9}$ takes the average of the worst 10\% of observed losses, and reflects the performance of the models under extreme cases.
Table~\ref{tab:num-res-1} reports the average loss and $\mathrm{CVaR}_{0.9}$ of losses for both sample sizes. Each CBR--DRO entry is followed by a ratio to the better of the baselines:
\[
    \mathrm{ratio}
    =\frac{\mathrm{CBR}}{\min(\mathrm{SAA},\mathrm{DROs})}\times 100\%.
\]
Ratios below $100\%$ indicate an improvement over both SAA and DRO, and the corresponding CBR--DRO entries are bolded. The results show that CBR--DRO has the best average and extreme-case performance in most of the settings. Notably, in data-limited regime $n=20$, the improvement of CBR--DRO in extreme-case performance (Table \ref{tab:num-res-1-cvar-n20}) is substantial (up to 58.9\%).
The improvements in average and extreme-case performance highlight the benefits of incorporating information about the data-generating distribution when hedging against ambiguity: by exploiting the structure of the underlying data-generating distribution and hedging against the parameter uncertainty and distributional uncertainty in a composite model, CBR--DRO provides a robust but less conservative solution, compare to plain SAA and DRO.

\begin{table}[htbp]
\centering
\renewcommand{\arraystretch}{1.08}
\setlength{\tabcolsep}{2.5pt}
\footnotesize
\captionsetup{font=footnotesize}
\captionsetup[subtable]{font=footnotesize,skip=3pt}

\begin{subtable}{\textwidth}
\centering
\begin{tabular}{@{}lrrrrrrr@{}}
\toprule
Distribution & SAA & DRO-W2 & DRO-KL & CBR-CVaR-W2 & CBR-Exp-W2 & CBR-CVaR-KL & CBR-Exp-KL \\
\midrule
Gaussian & $0.1161$ & $0.1130$ & $0.1171$ & $0.1193\,(105.6\%)$ & $\mathbf{0.1058\,(93.6\%)}$ & $0.1192\,(105.4\%)$ & $\mathbf{0.1061\,(93.9\%)}$ \\
Laplace & $0.1276$ & $0.1195$ & $0.1304$ & $\mathbf{0.1103\,(92.3\%)}$ & $\mathbf{0.0958\,(80.2\%)}$ & $\mathbf{0.1110\,(92.9\%)}$ & $\mathbf{0.0972\,(81.3\%)}$ \\
slash$_4$ & $0.2410$ & $0.2272$ & $0.2442$ & $\mathbf{0.2077\,(91.4\%)}$ & $\mathbf{0.1967\,(86.6\%)}$ & $\mathbf{0.2090\,(92.0\%)}$ & $\mathbf{0.1989\,(87.5\%)}$ \\
$t_{10}$ & $0.1496$ & $0.1480$ & $0.1524$ & $\mathbf{0.1402\,(94.7\%)}$ & $\mathbf{0.1343\,(90.8\%)}$ & $\mathbf{0.1410\,(95.3\%)}$ & $\mathbf{0.1360\,(91.9\%)}$ \\
$t_5$ & $0.2012$ & $0.1866$ & $0.2142$ & $\mathbf{0.1771\,(94.9\%)}$ & $\mathbf{0.1615\,(86.5\%)}$ & $\mathbf{0.1751\,(93.9\%)}$ & $\mathbf{0.1624\,(87.0\%)}$ \\
$t_4$ & $0.2519$ & $0.2207$ & $0.2570$ & $\mathbf{0.2176\,(98.6\%)}$ & $\mathbf{0.1963\,(88.9\%)}$ & $\mathbf{0.2172\,(98.4\%)}$ & $\mathbf{0.1977\,(89.6\%)}$ \\
\bottomrule
\end{tabular}
\subcaption{Sample size $n=20$: average loss.}
\label{tab:num-res-1-mean-n20}
\end{subtable}
\par\vspace{0.4em}

\begin{subtable}{\textwidth}
\centering
\begin{tabular}{@{}lrrrrrrr@{}}
\toprule
Distribution & SAA & DRO-W2 & DRO-KL & CBR-CVaR-W2 & CBR-Exp-W2 & CBR-CVaR-KL & CBR-Exp-KL \\
\midrule
Gaussian & $0.1680$ & $0.1657$ & $0.1658$ & $\mathbf{0.1656\,(99.9\%)}$ & $\mathbf{0.1419\,(85.6\%)}$ & $0.1660\,(100.2\%)$ & $\mathbf{0.1410\,(85.1\%)}$ \\
Laplace & $0.2171$ & $0.2142$ & $0.2166$ & $\mathbf{0.1531\,(71.5\%)}$ & $\mathbf{0.1261\,(58.9\%)}$ & $\mathbf{0.1529\,(71.4\%)}$ & $\mathbf{0.1356\,(63.3\%)}$ \\
slash$_4$ & $0.3666$ & $0.3397$ & $0.3656$ & $\mathbf{0.2744\,(80.8\%)}$ & $\mathbf{0.2483\,(73.1\%)}$ & $\mathbf{0.2749\,(80.9\%)}$ & $\mathbf{0.2586\,(76.1\%)}$ \\
$t_{10}$ & $0.2308$ & $0.2198$ & $0.2494$ & $\mathbf{0.2033\,(92.5\%)}$ & $\mathbf{0.1808\,(82.3\%)}$ & $\mathbf{0.2039\,(92.8\%)}$ & $\mathbf{0.1929\,(87.8\%)}$ \\
$t_5$ & $0.2976$ & $0.2642$ & $0.3344$ & $\mathbf{0.2450\,(92.7\%)}$ & $\mathbf{0.1971\,(74.6\%)}$ & $\mathbf{0.2449\,(92.7\%)}$ & $\mathbf{0.1978\,(74.9\%)}$ \\
$t_4$ & $0.3828$ & $0.3304$ & $0.3875$ & $\mathbf{0.3220\,(97.5\%)}$ & $\mathbf{0.2630\,(79.6\%)}$ & $\mathbf{0.3213\,(97.2\%)}$ & $\mathbf{0.2614\,(79.1\%)}$ \\
\bottomrule
\end{tabular}
\subcaption{Sample size $n=20$:$\mathrm{CVaR}_{0.9}$ of losses.}
\label{tab:num-res-1-cvar-n20}
\end{subtable}
\par\vspace{0.4em}

\begin{subtable}{\textwidth}
\centering
\begin{tabular}{@{}lrrrrrrr@{}}
\toprule
Distribution & SAA & DRO-W2 & DRO-KL & CBR-CVaR-W2 & CBR-Exp-W2 & CBR-CVaR-KL & CBR-Exp-KL \\
\midrule
Gaussian & $0.0924$ & $0.0932$ & $0.0945$ & $\mathbf{0.0896\,(97.0\%)}$ & $\mathbf{0.0891\,(96.4\%)}$ & $\mathbf{0.0898\,(97.2\%)}$ & $\mathbf{0.0891\,(96.5\%)}$ \\
Laplace & $0.0922$ & $0.0927$ & $0.0940$ & $\mathbf{0.0876\,(95.1\%)}$ & $\mathbf{0.0856\,(92.9\%)}$ & $\mathbf{0.0880\,(95.4\%)}$ & $\mathbf{0.0862\,(93.5\%)}$ \\
slash$_4$ & $0.1848$ & $0.1819$ & $0.1925$ & $\mathbf{0.1737\,(95.5\%)}$ & $\mathbf{0.1714\,(94.2\%)}$ & $\mathbf{0.1736\,(95.4\%)}$ & $\mathbf{0.1710\,(94.0\%)}$ \\
$t_{10}$ & $0.1222$ & $0.1237$ & $0.1273$ & $\mathbf{0.1109\,(90.8\%)}$ & $\mathbf{0.1101\,(90.1\%)}$ & $\mathbf{0.1119\,(91.5\%)}$ & $\mathbf{0.1104\,(90.3\%)}$ \\
$t_5$ & $0.1544$ & $0.1542$ & $0.1551$ & $\mathbf{0.1450\,(94.1\%)}$ & $\mathbf{0.1428\,(92.6\%)}$ & $\mathbf{0.1450\,(94.0\%)}$ & $\mathbf{0.1418\,(92.0\%)}$ \\
$t_4$ & $0.1812$ & $0.1747$ & $0.1846$ & $\mathbf{0.1698\,(97.2\%)}$ & $\mathbf{0.1651\,(94.5\%)}$ & $\mathbf{0.1699\,(97.3\%)}$ & $\mathbf{0.1653\,(94.6\%)}$ \\
\bottomrule
\end{tabular}
\subcaption{Sample size $n=50$: average loss.}
\label{tab:num-res-1-mean-n50}
\end{subtable}
\par\vspace{0.4em}

\begin{subtable}{\textwidth}
\centering
\begin{tabular}{@{}lrrrrrrr@{}}
\toprule
Distribution & SAA & DRO-W2 & DRO-KL & CBR-CVaR-W2 & CBR-Exp-W2 & CBR-CVaR-KL & CBR-Exp-KL \\
\midrule
Gaussian & $0.1105$ & $0.1174$ & $0.1130$ & $\mathbf{0.1042\,(94.2\%)}$ & $\mathbf{0.1033\,(93.5\%)}$ & $\mathbf{0.1050\,(95.0\%)}$ & $\mathbf{0.1035\,(93.7\%)}$ \\
Laplace & $0.1164$ & $0.1155$ & $0.1201$ & $\mathbf{0.1038\,(89.9\%)}$ & $\mathbf{0.1026\,(88.9\%)}$ & $\mathbf{0.1031\,(89.3\%)}$ & $\mathbf{0.1028\,(89.0\%)}$ \\
slash$_4$ & $0.2582$ & $0.2207$ & $0.2855$ & $\mathbf{0.1941\,(87.9\%)}$ & $\mathbf{0.1990\,(90.2\%)}$ & $\mathbf{0.1919\,(86.9\%)}$ & $\mathbf{0.1987\,(90.0\%)}$ \\
$t_{10}$ & $0.1638$ & $0.1655$ & $0.1955$ & $\mathbf{0.1342\,(81.9\%)}$ & $\mathbf{0.1305\,(79.7\%)}$ & $\mathbf{0.1344\,(82.1\%)}$ & $\mathbf{0.1324\,(80.8\%)}$ \\
$t_5$ & $0.1859$ & $0.1939$ & $0.1843$ & $\mathbf{0.1735\,(94.1\%)}$ & $\mathbf{0.1707\,(92.6\%)}$ & $\mathbf{0.1736\,(94.2\%)}$ & $\mathbf{0.1669\,(90.5\%)}$ \\
$t_4$ & $0.2213$ & $0.2155$ & $0.2216$ & $0.2180\,(101.1\%)$ & $\mathbf{0.1962\,(91.0\%)}$ & $0.2176\,(101.0\%)$ & $\mathbf{0.1956\,(90.8\%)}$ \\
\bottomrule
\end{tabular}
\subcaption{Sample size $n=50$: $\mathrm{CVaR}_{0.9}$ of losses.}
\label{tab:num-res-1-cvar-n50}
\end{subtable}

\caption{Average and extreme-case performance of SAA, DRO, and CBR--DRO.}
\label{tab:num-res-1}
\end{table}

\subsection{Portfolio Optimization with Real Data}
\label{sec:num-real}

In this section, we apply the quadratic-shortfall optimization problem \eqref{eq:num-loss-1} to portfolio optimization with real data. 
We use the ``30 Industry Portfolios'' from the Kenneth R. French Data Library.\footnote{\href{https://mba.tuck.dartmouth.edu/pages/faculty/ken.french/Data_Library/det_30_ind_port.html}{Kenneth R. French Data Library: 30 Industry Portfolios}.} The dataset contains the daily returns of 30 portfolio groups by industry, e.g., agriculture, machinery, or communication. We use data from January 2017 to August 2026 in this experiment. In this experiment, we consider the monthly-level returns, which are constructed by compounding the daily returns. 
We implement the quadratic-shortfall portfolio problem in \eqref{eq:num-loss-1} with $p=30$ and $\tau=0$, 
so that the objective penalizes negative portfolio returns. At each month-end, each method (SAA, DRO, CBR--DRO) uses the preceding $36$ monthly return observations as samples, and determines the allocation for the following month. Portfolios are rebalanced monthly, and we assume 
$1\,\permil $
transaction cost for rebalance. We evaluate the performance of resulting portfolios from January 2020 through August 2026. The hyperparameters of the models are selected via cross-validation using the past 36 month's data. 

Figure~\ref{fig:portfolio-real-returns} displays the cumulative returns curve of each model, and
Table \ref{tab:portfolio-real-performance} shows the annual returns and Sharpe ratios of the portfolios from January 2020 to August 2026. CBR--DRO entries that perform better than all SAA and DRO baselines are bolded.
The results show that all CBR--DRO variants achieve higher final cumulative returns than SAA and DRO baselines, and most of them achieve better annual returns and Sharpe ratios than the baselines during 2020--2026.
Combining the results from the preceding synthetic-data experiments, these advantages suggest that the
robustness--conservatism balance underlying the proposed CBR--DRO framework may help improve both realized returns and risk-adjusted performance, and hence contribute to better overall portfolio performance.

\begin{figure}[htbp]
\centering
\includegraphics[width=0.9\linewidth]{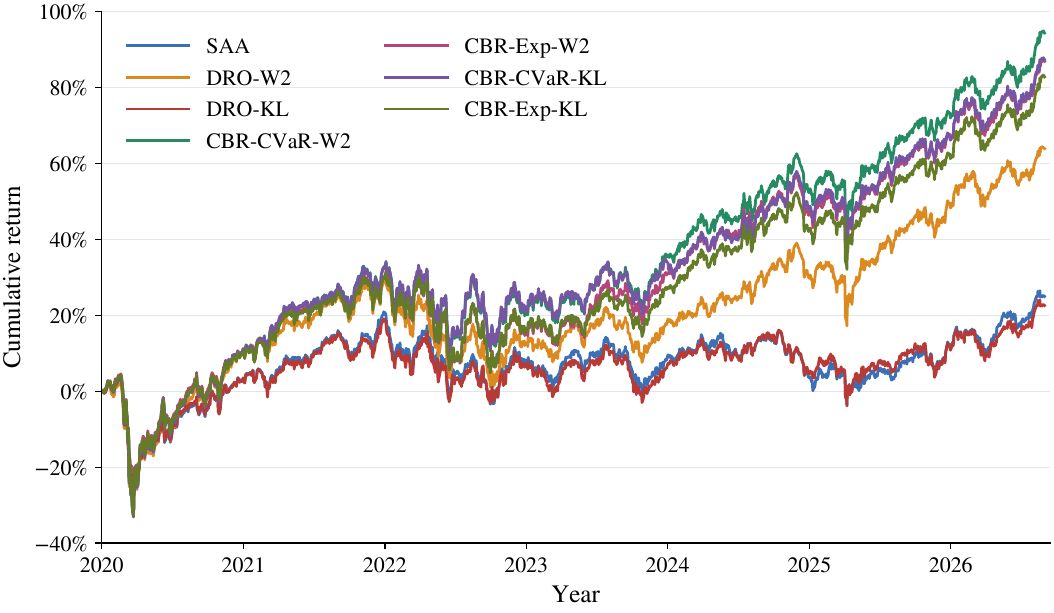}
\caption{Cumulative portfolio returns from January 2020 to August 2026.}
\label{fig:portfolio-real-returns}
\end{figure}

\begin{table}[htbp]
\centering
\scriptsize
\setlength{\tabcolsep}{4pt}
\resizebox{\textwidth}{!}{%
\begin{tabular}{llrrrrrrr}
\toprule
Period & Metric & SAA & DRO-W2 & DRO-KL & \shortstack{CBR-CVaR-\\W2} & \shortstack{CBR-Exp-\\W2} & \shortstack{CBR-CVaR-\\KL} & \shortstack{CBR-Exp-\\KL} \\
\midrule
2020 & Ann. return (\%) & 3.58 & 10.52 & 3.41 & \textbf{11.25} & \textbf{11.20} & \textbf{11.13} & \textbf{10.62} \\
 & Sharpe ratio & 0.24 & 0.49 & 0.23 & \textbf{0.51} & \textbf{0.52} & \textbf{0.51} & \textbf{0.50} \\
\addlinespace[2pt]
2021 & Ann. return (\%) & 16.64 & 16.55 & 14.98 & \textbf{19.22} & \textbf{17.02} & \textbf{19.17} & \textbf{17.38} \\
 & Sharpe ratio & 1.31 & 1.44 & 1.20 & \textbf{1.49} & 1.33 & \textbf{1.47} & 1.34 \\
\addlinespace[2pt]
2022 & Ann. return (\%) & -10.66 & -12.69 & -10.18 & \textbf{-8.71} & -11.14 & \textbf{-7.89} & -10.46 \\
 & Sharpe ratio & -0.68 & -0.75 & -0.72 & \textbf{-0.36} & \textbf{-0.46} & \textbf{-0.32} & \textbf{-0.43} \\
\addlinespace[2pt]
2023 & Ann. return (\%) & 0.26 & 4.50 & -0.69 & \textbf{11.83} & \textbf{12.96} & \textbf{10.01} & \textbf{9.08} \\
 & Sharpe ratio & -0.29 & 0.03 & -0.34 & \textbf{0.52} & \textbf{0.57} & \textbf{0.40} & \textbf{0.35} \\
\addlinespace[2pt]
2024 & Ann. return (\%) & -3.98 & 10.67 & -1.07 & \textbf{12.67} & \textbf{12.24} & 10.38 & \textbf{12.22} \\
 & Sharpe ratio & -0.83 & 0.49 & -0.55 & \textbf{0.63} & \textbf{0.61} & 0.45 & \textbf{0.61} \\
\addlinespace[2pt]
2025 & Ann. return (\%) & 7.02 & 14.01 & 5.61 & 12.72 & 13.31 & 12.29 & 13.31 \\
 & Sharpe ratio & 0.38 & 1.36 & 0.18 & \textbf{1.39} & \textbf{1.56} & 1.30 & \textbf{1.54} \\
\addlinespace[2pt]
2026 (Jan.--Aug.) & Ann. return (\%) & 19.15 & 16.23 & 16.39 & \textbf{20.18} & \textbf{19.32} & \textbf{19.55} & \textbf{20.68} \\
 & Sharpe ratio & 1.35 & 1.13 & 1.14 & \textbf{1.57} & \textbf{1.43} & \textbf{1.45} & \textbf{1.50} \\
\midrule
Overall & Cumulative (\%) & 24.98 & 63.94 & 22.60 & \textbf{94.38} & \textbf{86.88} & \textbf{87.36} & \textbf{82.80} \\
 & Ann. return (\%) & 3.40 & 7.70 & 3.10 & \textbf{10.48} & \textbf{9.83} & \textbf{9.88} & \textbf{9.47} \\
 & Sharpe ratio & 0.11 & 0.39 & 0.09 & \textbf{0.54} & \textbf{0.50} & \textbf{0.50} & \textbf{0.48} \\
\bottomrule
\end{tabular}%
}
\caption{Annual returns and Sharpe ratios of the portfolios from January 2020 to August 2026.}
\label{tab:portfolio-real-performance}
\end{table}

\section{Concluding Remarks}

\label{sec:conclusion}
In this paper, we investigate the integration of Bayesian learning and distributionally robust optimization to address distributional uncertainty associated with elliptical random variables. In contrast to existing work in this area, we adopt a Gibbs posterior framework that does not require specification of the underlying probability density function and instead relies on estimating equations.

Several promising directions remain for future research. First, although our consistency results establish that any fixed learning rate $\eta>0$ yields posterior consistency, this asymptotic guarantee provides limited guidance on the choice of $\eta$ for finite-sample uncertainty quantification. It would therefore be of considerable interest to develop principled methods for selecting an optimal learning rate.
Second, for ease of exposition, we restrict our attention to a single-stage decision-making problem. Many practical applications, however, are inherently multistage. We anticipate that the proposed CBR--DRO framework can be extended to dynamic decision-making settings, similar to those considered in \citet{ref:tacskesen2025optimality,ma2026adaptive}, where the exogenous uncertainty is characterized by an elliptical random vector. In such settings, the location and shape parameters may be allowed to evolve over time to capture distributional shifts \citep{keehan2025don}.
Third, from a statistical perspective, it would be valuable to characterize the contraction rate of the proposed Gibbs posterior as the sample size increases. While our consistency result ensures that the posterior eventually concentrates around the true parameter, a contraction-rate analysis would further quantify the speed of this concentration and determine whether optimal rates can be achieved under suitable regularity conditions~\citep{ref:syring2023gibbs}.
We leave these important questions for future investigation.

\bibliographystyle{plainnat}
\bibliography{ref}

@book{ref:fang2018symmetric,
  title={Symmetric multivariate and related distributions},
  author = {Fang, Kai-Tai and Kotz, Samuel
          and Ng, Kai W.},
  year={2018},
  publisher={Chapman and Hall/CRC}
}

@article{ref:hanasusanto2015distributionally,
  title={Distributionally robust multi-item newsvendor problems with multimodal demand distributions},
  author={Hanasusanto, Grani A and Kuhn, Daniel and Wallace, Stein W and Zymler, Steve},
  journal={Mathematical Programming},
  volume={152},
  number={1},
  pages={1--32},
  year={2015},
  publisher={Springer}
}

@article{ref:jiang2024distributionally,
  title={Distributionally favorable optimization: A framework for data-driven decision-making with endogenous outliers},
  author={Jiang, Nan and Xie, Weijun},
  journal={SIAM Journal on Optimization},
  volume={34},
  number={1},
  pages={419--458},
  year={2024},
  publisher={SIAM}
}

@article{ref:xie2021distributionally,
  title={On distributionally robust chance constrained programs with {Wasserstein} distance},
  author={Xie, Weijun},
  journal={Mathematical Programming},
  volume={186},
  number={1},
  pages={115--155},
  year={2021},
  publisher={Springer}
}

@article{ref:tacskesen2025optimality,
  title={Optimality of linear policies in distributionally robust linear quadratic control},
  author={Ta{\c{s}}kesen, Bahar and Iancu, Dan A and Ko{\c{c}}yi{\u{g}}it, {\c{C}}a{\u{g}}{\i}l and Kuhn, Daniel},
  journal={arXiv preprint arXiv:2508.11858},
  year={2025}
}

@article{ref:wiesemann2014distributionally,
  title={Distributionally robust convex optimization},
  author={Wiesemann, Wolfram and Kuhn, Daniel and Sim, Melvyn},
  journal={Operations Research},
  volume={62},
  number={6},
  pages={1358--1376},
  year={2014},
  publisher={INFORMS}
}

@article{ref:shapiro2023bayesian,
  title={{Bayesian} distributionally robust optimization},
  author={Shapiro, Alexander and Zhou, Enlu and Lin, Yifan},
  journal={SIAM Journal on Optimization},
  volume={33},
  number={2},
  pages={1279--1304},
  year={2023},
  publisher={SIAM}
}

@article{ref:prashanth2022wasserstein,
  title={A {Wasserstein} distance approach for concentration of empirical risk estimates},
  author = {Prashanth, L. A. and Bhat, Sanjay P.},
  journal={Journal of Machine Learning Research},
  volume={23},
  number={238},
  pages={1--61},
  year={2022}
}

@article{ref:pichler2013natural,
  title={The natural Banach space for version independent risk measures},
  author={Pichler, Alois},
  journal={Insurance: Mathematics and Economics},
  volume={53},
  number={2},
  pages={405--415},
  year={2013},
  publisher={Elsevier}
}

@article{ref:hewitt1955symmetric,
  title={Symmetric measures on Cartesian products},
  author={Hewitt, Edwin and Savage, Leonard J},
  journal={Transactions of the American Mathematical Society},
  volume={80},
  number={2},
  pages={470--501},
  year={1955}
}

@article{ref:catoni2007pac,
  title={{PAC}-{Bayesian} supervised classification: the thermodynamics of statistical learning},
  author={Catoni, Olivier},
  journal={arXiv preprint arXiv:0712.0248},
  year={2007}
}

@inproceedings{ref:mcallester1998some,
  title={Some pac-bayesian theorems},
  author={McAllester, David A},
  booktitle={Proceedings of the eleventh annual conference on Computational learning theory},
  pages={230--234},
  year={1998}
}

@article{keehan2025don,
  title={Don't Look Back in Anger: {Wasserstein} Distributionally Robust Optimization with Nonstationary Data},
  author={Keehan, Dominic ST and Anderson, Edward J and Wiesemann, Wolfram},
  journal={arXiv preprint arXiv:2510.18566},
  year={2025}
}

@article{ma2026adaptive,
  title={Adaptive Distributionally Robust Optimal Control with {Bayesian} Ambiguity Sets},
  author={Ma, Wentao and Chen, Zhiping and Xu, Huifu and Zhou, Enlu},
  journal={arXiv preprint arXiv:2604.06936},
  year={2026}
}

@article{xu2018distributionally,
  title={Distributionally robust optimization with matrix moment constraints: Lagrange duality and cutting plane methods},
  author={Xu, Huifu and Liu, Yongchao and Sun, Hailin},
  journal={Mathematical programming},
  volume={169},
  number={2},
  pages={489--529},
  year={2018},
  publisher={Springer}
}

@inproceedings{ref:iyengar2023hedging,
  title={Hedging against Complexity: Distributionally Robust Optimization
         with Parametric Approximation},
  author={Iyengar, Garud and Lam, Henry and Wang, Tianyu},
  booktitle={Proceedings of the 26th International Conference on
             Artificial Intelligence and Statistics},
  volume={206},
  pages={9976--10011},
  year={2023},
  publisher={PMLR}
}

@article{ref:shang2022distributionally,
  title={Distributionally Robust Chance Constraint with
         Unimodality-Skewness Information and Conic Reformulation},
  author={Shang, Chao and Wang, Chao and You, Keyou and Huang, Dexian},
  journal={Operations Research Letters},
  volume={50},
  number={2},
  pages={176--183},
  year={2022},
  doi={10.1016/j.orl.2022.01.016}
}

@article{homem2008rates,
  title={On rates of convergence for stochastic optimization problems under non--independent and identically distributed sampling},
  author={Homem-de-Mello, Tito},
  journal={SIAM Journal on Optimization},
  volume={19},
  number={2},
  pages={524--551},
  year={2008},
  publisher={SIAM}
}

@article{ref:mohajerin2018data,
  title={Data-driven distributionally robust optimization using the {Wasserstein} metric: Performance guarantees and tractable reformulations},
  author={Mohajerin Esfahani, Peyman and Kuhn, Daniel},
  journal={Mathematical Programming},
  volume={171},
  number={1},
  pages={115--166},
  year={2018},
  publisher={Springer}
}

@article{ref:ben2007old,
  title={An old-new concept of convex risk measures: The optimized certainty equivalent},
  author={Ben-Tal, Aharon and Teboulle, Marc},
  journal={Mathematical Finance},
  volume={17},
  number={3},
  pages={449--476},
  year={2007},
  publisher={Wiley Online Library}
}

@article{ref:ben1986expected,
  title={Expected utility, penalty functions, and duality in stochastic nonlinear programming},
  author={Ben-Tal, Aharon and Teboulle, Marc},
  journal={Management Science},
  volume={32},
  number={11},
  pages={1445--1466},
  year={1986},
  publisher={INFORMS}
}

@article{xu2010uniform,
  title={Uniform exponential convergence of sample average random functions under general sampling with applications in stochastic programming},
  author={Xu, Huifu},
  journal={Journal of Mathematical Analysis and Applications},
  volume={368},
  number={2},
  pages={692--710},
  year={2010},
  publisher={Elsevier}
}

@article{ref:delage2010distributionally,
  title={Distributionally robust optimization under moment uncertainty with application to data-driven problems},
  author={Delage, Erick and Ye, Yinyu},
  journal={Operations research},
  volume={58},
  number={3},
  pages={595--612},
  year={2010},
  publisher={INFORMS}
}

@article{ref:kelley1960cutting,
  title={The cutting-plane method for solving convex programs},
  author={Kelley, Jr, James E},
  journal={Journal of the society for Industrial and Applied Mathematics},
  volume={8},
  number={4},
  pages={703--712},
  year={1960},
  publisher={SIAM}
}

@article{ref:ben2013robust,
  title={Robust solutions of optimization problems affected by uncertain probabilities},
  author = {Ben-Tal, Aharon
          and den Hertog, Dick
          and De Waegenaere, Anja M. B.
          and Melenberg, Bertrand
          and Rennen, Gijs},
  journal={Management Science},
  volume={59},
  number={2},
  pages={341--357},
  year={2013},
  publisher={INFORMS}
}

@article{ref:gao2023distributionally,
  title={Distributionally robust stochastic optimization with {Wasserstein} distance},
  author={Gao, Rui and Kleywegt, Anton},
  journal={Mathematics of Operations Research},
  volume={48},
  number={2},
  pages={603--655},
  year={2023},
  publisher={INFORMS}
}

@book{rockafellar1997convex,
  title={Convex analysis},
  author={Rockafellar, R Tyrrell},
  volume={28},
  year={1997},
  publisher={Princeton university press}
}

@article{ref:tyler1987statistical,
  title={Statistical analysis for the angular central Gaussian distribution on the sphere},
  author={Tyler, David E},
  journal={Biometrika},
  volume={74},
  number={3},
  pages={579--589},
  year={1987},
  publisher={Oxford University Press}
}

@article{ref:tyler1987distribution,
  title={A distribution-free M-estimator of multivariate scatter},
  author={Tyler, David E},
  journal={The annals of Statistics},
  pages={234--251},
  year={1987},
  publisher={JSTOR}
}

@incollection{ref:martin2022direct,
  title={Direct {Gibbs} posterior inference on risk minimizers: Construction, concentration, and calibration},
  author={Martin, Ryan and Syring, Nicholas},
  booktitle={Handbook of Statistics},
  volume={47},
  pages={1--41},
  year={2022},
  publisher={Elsevier}
}

@article{ref:bissiri2012bayesian,
  title={On {Bayesian} learning via loss functions},
  author={Bissiri, Pier Giovanni and Walker, Stephen G},
  journal={Journal of statistical planning and inference},
  volume={142},
  number={12},
  pages={3167--3173},
  year={2012},
  publisher={Elsevier}
}

@article{ref:syring2023gibbs,
  title={{Gibbs} posterior concentration rates under sub-exponential type losses},
  author={Syring, Nicholas and Martin, Ryan},
  journal={Bernoulli},
  volume={29},
  number={2},
  pages={1080--1108},
  year={2023},
  publisher={Bernoulli Society for Mathematical Statistics and Probability}
}

@article{ref:bissiri2010bayesian,
  title={On {Bayesian} learning from Bernoulli observations},
  author={Bissiri, Pier Giovanni and Walker, Stephen G},
  journal={Journal of statistical planning and inference},
  volume={140},
  number={11},
  pages={3520--3530},
  year={2010},
  publisher={Elsevier}
}

@article{ref:zhang2006information,
  title={Information-theoretic upper and lower bounds for statistical estimation},
  author={Zhang, Tong},
  journal={IEEE Transactions on Information Theory},
  volume={52},
  number={4},
  pages={1307--1321},
  year={2006},
  publisher={IEEE}
}

@article{rahimian2019distributionally,
  title={Distributionally robust optimization: A review},
  author={Rahimian, Hamed and Mehrotra, Sanjay},
  journal={arXiv preprint arXiv:1908.05659},
  year={2019}
}

@article{ref:capponi2022systemic,
  title={Systemic risk-driven portfolio selection},
  author={Capponi, Agostino and Rubtsov, Alexey},
  journal={Operations Research},
  volume={70},
  number={3},
  pages={1598--1612},
  year={2022},
  publisher={INFORMS}
}

@article{ref:shapiro2013kusuoka,
  title={On {Kusuoka} representation of law invariant risk measures},
  author={Shapiro, Alexander},
  journal={Mathematics of Operations Research},
  volume={38},
  number={1},
  pages={142--152},
  year={2013},
  publisher={INFORMS}
}

@article{ref:anderson2020varying,
  title={Varying confidence levels for {CVaR} risk measures and minimax limits},
  author={Anderson, Edward and Xu, Huifu and Zhang, Dali},
  journal={Mathematical Programming},
  volume={180},
  number={1},
  pages={327--370},
  year={2020},
  publisher={Springer}
}

@article{ref:acerbi2002spectral,
  title={Spectral measures of risk: A coherent representation of subjective risk aversion},
  author={Acerbi, Carlo},
  journal={Journal of Banking \& Finance},
  volume={26},
  number={7},
  pages={1505--1518},
  year={2002},
  publisher={Elsevier}
}

@article{ref:wu2018bayesian,
  title={A {Bayesian} risk approach to data-driven stochastic optimization: Formulations and asymptotics},
  author={Wu, Di and Zhu, Helin and Zhou, Enlu},
  journal={SIAM Journal on Optimization},
  volume={28},
  number={2},
  pages={1588--1612},
  year={2018},
  publisher={SIAM}
}

@article{ref:kuhn2025distributionally,
  title={Distributionally robust optimization},
  author={Kuhn, Daniel and Shafiee, Soroosh and Wiesemann, Wolfram},
  journal={Acta Numerica},
  volume={34},
  pages={579--804},
  year={2025},
  publisher={Cambridge University Press}
}

@article{ref:follmer2002convex,
  title={Convex measures of risk and trading constraints},
  author={F{\"o}llmer, Hans and Schied, Alexander},
  journal={Finance and stochastics},
  volume={6},
  number={4},
  pages={429--447},
  year={2002},
  publisher={Springer}
}

@article{ref:artzner1999coherent,
  title={Coherent measures of risk},
  author={Artzner, Philippe and Delbaen, Freddy and Eber, Jean-Marc and Heath, David},
  journal={Mathematical finance},
  volume={9},
  number={3},
  pages={203--228},
  year={1999},
  publisher={Wiley Online Library}
}

@article{ref:chen2011tight,
  title={Tight bounds for some risk measures, with applications to robust portfolio selection},
  author={Chen, Li and He, Simai and Zhang, Shuzhong},
  journal={Operations Research},
  volume={59},
  number={4},
  pages={847--865},
  year={2011},
  publisher={INFORMS}
}

@article{ref:fishburn1977mean,
  title={Mean-risk analysis with risk associated with below-target returns},
  author={Fishburn, Peter C},
  journal={The American Economic Review},
  volume={67},
  number={2},
  pages={116--126},
  year={1977},
  publisher={JSTOR}
}

@article{ref:qian2019composite,
  title={A composite risk measure framework for decision making under uncertainty},
  author={Qian, Peng-Yu and Wang, Zi-Zhuo and Wen, Zai-Wen},
  journal={Journal of the Operations Research Society of China},
  volume={7},
  number={1},
  pages={43--68},
  year={2019},
  publisher={Springer}
}

@article{ref:bissiri2016general,
  title={A general framework for updating belief distributions},
  author={Bissiri, Pier Giovanni and Holmes, Chris C and Walker, Stephen G},
  journal={Journal of the Royal Statistical Society Series B: Statistical Methodology},
  volume={78},
  number={5},
  pages={1103--1130},
  year={2016},
  publisher={Oxford University Press}
}

@article{ref:KOZUBOWSKI201359,
title = {Multivariate generalized Laplace distribution and related random fields},
journal = {Journal of Multivariate Analysis},
volume = {113},
pages = {59-72},
year = {2013},
note = {Special Issue on Multivariate Distribution Theory in Memory of Samuel Kotz},
issn = {0047-259X},
author = {Tomasz J. Kozubowski and Krzysztof Podgórski and Igor Rychlik},
}

@book{ref:delmas2024elliptically,
  title={Elliptically symmetric distributions in signal processing and machine learning},
  author={Delmas, Jean-Pierre and El Korso, Mohammed Nabil and Fortunati, Stefano and Pascal, Fr{\'e}d{\'e}ric},
  volume={1},
  year={2024},
  publisher={Springer}
}

@article{ref:ollila2012complex,
  title={Complex elliptically symmetric distributions: Survey, new results and applications},
  author={Ollila, Esa and Tyler, David E and Koivunen, Visa and Poor, H Vincent},
  journal={IEEE Transactions on signal processing},
  volume={60},
  number={11},
  pages={5597--5625},
  year={2012},
  publisher={IEEE}
}

@article{ref:bodnar2007distribution,
  title={The distribution of the sample variance of the global minimum variance portfolio in elliptical models},
  author={Bodnar, Taras and Schmid, Wolfgang},
  journal={Statistics},
  volume={41},
  number={1},
  pages={65--75},
  year={2007},
  publisher={Taylor \& Francis}
}

@book{ref:mcneil2015quantitative,
  title={Quantitative risk management: concepts, techniques and tools-revised edition},
  author={McNeil, Alexander J and Frey, R{\"u}diger and Embrechts, Paul},
  year={2015},
  publisher={Princeton university press}
}

@article{ref:kibria2006short,
  title={A short review of multivariate t-distribution},
  author = {Kibria, B. M. Golam
          and Joarder, Anwar H.},
  journal={Journal of Statistical research},
  volume={40},
  number={1},
  pages={59--72},
  year={2006}
}

\clearpage
\appendix
\numberwithin{equation}{section}
\numberwithin{theorem}{section}
\numberwithin{lemma}{section}
\numberwithin{proposition}{section}
\numberwithin{corollary}{section}
\numberwithin{assumption}{section}
\numberwithin{definition}{section}
\numberwithin{example}{section}
\numberwithin{remark}{section}
\numberwithin{figure}{section}
\numberwithin{table}{section}

\section{Lemmas and Proofs}\label{sec:lemmas}

\subsection{Lemmas}

The proofs of the results in this paper rely on the following lemmas.

\begin{lemma}\label{lemma:sup-P-E}
    Let $\xi\in\R^p$ be an elliptically distributed random vector defined as in \eqref{eq:EC-stochastic-rep} and admit a bounded density function $f_\xi$, i.e., $\|f_{\xi}\|_{\infty}\Let\sup_{y\in \R^p} f_\xi(y) < \infty$. Then for any compact set $X\subset \R^p$, it holds that
    \begin{enumerate}[label = (\roman*)]
        \item $\lim\limits_{\delta \to 0} \sup_{x\in X} \prob(\|\xi-x\|\leq \delta)=0$,
        \item for $p \ge 2$, $\sup_{x\in X} \E[1/\|\xi-x\|] < \infty$. 
    \end{enumerate}
\end{lemma}

\begin{proof}
    Part (i). Let $\mathbb{B}(x,\delta)\Let \{y\in\R^p: \|y-x\|\leq \delta\}$. Then for any $x \in X$ and $\delta > 0$, it holds that
    \[
        \prob(\|\xi - x\| \le \delta)
        = \int_{\mathbb{B}(x,\delta)} f_\xi(y)\,\diff y
        \le \|f_\xi\|_\infty \, |\mathbb{B}(x,\delta)|,
    \]
    where $|\mathbb{B}(x,\delta)|$ denotes the volume of $\mathbb{B}(x,\delta)$. Since $|\mathbb{B}(x,\delta)|\to 0$ as $\delta\to 0$, we conclude that $\prob(\|\xi - x\| \le \delta)$ also tends to $0$ as $\delta\to 0$.

    Part (ii). Let $x \in X$ be arbitrary but fixed. We first note that
    \[
        \E\!\left[\frac{1}{\|\xi - x\|}\right]
        = \int_{\R^p} \frac{f_\xi(y)}{\|y-x\|}\,\diff y
        = \int_{\|y-x\|\le 1} \frac{f_\xi(y)}{\|y-x\|}\,\diff y
        + \int_{\|y-x\|>1} \frac{f_\xi(y)}{\|y-x\|}\,\diff y.
    \]
    For the second term,
    \[
        \int_{\|y-x\|>1} \frac{f_\xi(y)}{\|y-x\|}\,\diff y
        \le \int_{\|y-x\|>1} f_\xi(y)\,\diff y \le 1.
    \]
    For the first term, since $f_\xi$ is uniformly bounded over $\R^p$, it follows that
    \[
        \int_{\|y-x\|\le 1} \frac{f_\xi(y)}{\|y-x\|}\,\diff y
        \le \|f_\xi\|_\infty \int_{\|y-x\|\le 1} \frac{\diff y}{\|y-x\|} = \|f_\xi\|_{\infty} \int_{\|z\|\leq 1} \frac{\diff z}{\|z\|}.
    \]
    It remains to show that $\int_{\|z\|\leq 1} \frac{1}{\|z\|} \diff z$ is bounded. By polar coordinates, let $z=r\alpha$, where $r\in[0,1]$ and $\alpha\in \mathbb{U}^{p-1}$. Then
    \[
        \int_{\|z\|\le 1} \frac{\diff z}{\|z\|} = |\mathbb{U}^{p-1}| \int_0^1 r^{p-2}\,\diff r,
    \]
    where $|\mathbb{U}^{p-1}| $ denotes the volume of the unit sphere. Note that $\int_0^1 r^{p-2}\,\diff r$ is finite if and only if $p\geq 2$, which shows that $\E[1/\|\xi-x\|]$ is finite when $p\geq 2$. Since the upper bound is independent of $x$, we conclude that $\sup_{x\in X} \E[1/\|\xi-x\|]$ is also finite. 
\end{proof}

\begin{lemma}\label{lemma:E-phi-1-inf}
    Let $\{x_\ell\}_{\ell=1}^{\infty}$ be any sequence in $\R^p$. If $\|x_\ell\| \to \infty$ as $\ell\to\infty$, then $\|\E[\varphi_1(x_\ell,\xi)]\|\to 1$. 
\end{lemma}

\begin{proof}
    Let $v_\ell\Let x_\ell/\|x_\ell\|, \rho_\ell\Let\|x_\ell\|$ and write $x_\ell = \rho_{\ell} v_\ell$. Since the unit sphere is a compact set, one can choose a convergent subsequence of $\{v_\ell\}_{\ell=1}^{\infty}$, say $\{v_{\ell_k}\}_{k=1}^\infty$, and assume that $v_{\ell_k}\to v$. Then it holds that for almost every $\xi$, 
    \[
        \varphi_1(x_{\ell_k},\xi) = \frac{\xi/\rho_{\ell_k}-v_{\ell_k}}{\|\xi/\rho_{\ell_k}-v_{\ell_k}\|} \to -v,
    \]
    which means that 
    \[
        \|\E[\varphi_1(x_{\ell_k},\xi)]\| \to \|-v\| = 1.
    \]
    The above shows that every subsequence of $\E[\varphi_1(x_{\ell},\xi)]$ has a further subsequence converging to $1$, which implies that the whole sequence converges to $1$ and completes the proof.
\end{proof}

\begin{lemma}\label{lemma:g-z}
    Let $g(x)\Let x/\|x\|$ for $0\neq x\in \R^p$. Then for any $x,y\in \R^p$ satisfying $tx+(1-t)y\neq 0,\forall t\in [0,1]$, it holds that 
    \[
        \|g(x)-g(y)\|\leq \sup_{t\in[0,1]} \frac{1}{\|tx+(1-t)y\|} \|x-y\|.
    \]
\end{lemma}
\begin{proof}
    The result is a direct application of the mean value inequality for vector-valued functions. First note that the Jacobian matrix of $g$ is given by
    \[
        Dg(z) = \frac{1}{\|z\|} \left(I-\frac{zz^\top}{\|z\|^2}\right),
    \]
    and it holds that 
    \[
        \|Dg(z)\|_2 = \frac{1}{\|z\|}
    \]
    since $I-\frac{zz^\top}{\|z\|^2}$ is an orthogonal projection matrix and $\|I-\frac{zz^\top}{\|z\|^2}\|=1$. Then by the mean value inequality for vector-valued functions, we have
    \begin{align*}
        \|g(x)-g(y)\|\;\leq\; &\sup_{t\in[0,1]} \|Dg(tx+(1-t)y)\|_2 \|x-y\|\\
        \;=\; &\sup_{t\in[0,1]} \frac{1}{\|tx+(1-t)y\|} \|x-y\|,
    \end{align*}
    which completes the proof.
\end{proof}

\begin{lemma}\label{lemma:phi-1-ub}
    Let $\xi\in \R^p$ be arbitrary but fixed. Then for any $\delta >0$ and $x,y\in\R^p$ such that $\|x-y\|\leq \delta$, it holds that 
    \[
        \|\varphi_1(x,\xi)- \varphi_1(y,\xi)\| \leq 2\mathds{1}_{\{\|\xi-y\|\leq 2\delta\}} + \frac{2\delta}{\|\xi-y\|}\mathds{1}_{\{\|\xi-y\|>2\delta\}}.
    \]
\end{lemma}
\begin{proof}
    Recall that $\varphi_1(x,\xi) = (\xi-x)/{\|\xi-x\|}$.
    By definition of $\varphi_1$, it readily holds that $\|\varphi_1(x,\xi) - \varphi_1(y,\xi)\| \leq 2$. Then it suffices to show that if $\|\xi-y\|\geq 2\delta$, then 
    \[
        \|\varphi_1(x,\xi)- \varphi_1(y,\xi)\| \leq \frac{2\delta}{\|\xi-y\|}. 
    \]
    Now let $g(z) = z/\|z\|$ be defined as in Lemma \ref{lemma:g-z}. To use Lemma \ref{lemma:g-z}, we note that $\varphi_1(x,\xi)=g(\xi-x)$. Moreover, if $\|\xi-y\|\geq 2\delta$, then neither $\xi-x$ or $\xi-y$ equals to $0$, and $t(\xi-x)+(1-t)(\xi-y) = \xi -(tx+(1-t)y) \neq 0$ for any $t\in[0,1]$. Let $z_t\Let (tx+(1-t)y)$. Then by Lemma~\ref{lemma:g-z}, we have 
    \begin{align}
        \|\varphi_1(x,\xi)-\varphi_1(y,\xi)\| \;=\; &\|g(\xi-x)-g(\xi-y)\| \;\leq\; \sup_{t\in[0,1]} \frac{1}{\|\xi -z_t\|} \|x-y\|\nonumber\\ 
        \;\leq\; &\frac{2\|x-y\|}{\|\xi-y\|} \leq \frac{2\delta}{\|\xi-y\|}, \label{ineq:varphi-zt}
    \end{align}
    where the second inequality is by the inequality 
    \[
        \|\xi-z_t\| \geq \|\xi-y\| -\|z_t-y\| \geq \|\xi-y\| - \delta \geq \frac{1}{2}\|\xi-y\|,
    \]
    and \eqref{ineq:varphi-zt} completes the proof.
\end{proof}

\begin{lemma}\label{lemma:phi1-ULLN}
    Let $X\subset \R^p$ be a compact set. Then it holds that
    \begin{align}
        \sup_{x\in X}\left\|\frac{1}{n}\sum_{i=1}^n \varphi_1(x,\xi_i)-\E[\varphi_1(x,\xi)]\right\|\xrightarrow{\text{a.s.}}  0. \label{eq:ULLN}
    \end{align}
\end{lemma}

\begin{proof}
    Let $\mathbb{B}(x,r)\Let \{y\in \R^p: \|y-x\|\leq r\}$. Let $x_1,\ldots,x_N$ be the centers of a $\delta$-net of $X$, i.e., $X\subset \bigcup_{j=1}^N \mathbb{B}(x_j,\delta)$. Then for any $x\in X$ and $x_j$ satisfying $\|x-x_j\|\leq \delta$, it holds that
    \begin{align*}
        &\left\|\frac{1}{n}\sum_{i=1}^n \varphi_1(x, \xi_i) - \E[\varphi_1(x,\xi)]\right\|\\ \;\leq\; &\left\|\E[\varphi_1(x,\xi)] - \E[\varphi_1(x_j,\xi)]\right\| + \left\|\frac{1}{n}\sum_{i=1}^n \varphi_1(x, \xi_i) - \frac{1}{n}\sum_{i=1}^n \varphi_1(x_j, \xi_i)\right\| + \left\|\frac{1}{n}\sum_{i=1}^n \varphi_1(x_j, \xi_i) - \E[\varphi_1(x_j,\xi)]\right\|\\
        \;\leq\; &\E[\|\varphi_1(x,\xi)-\varphi_1(x_j,\xi)\|] + \frac{1}{n}\sum_{i=1}^n \left\|\varphi_1(x,\xi_i)-\varphi_1(x_j,\xi_i)\right\| + \left\|\frac{1}{n}\sum_{i=1}^n \varphi_1(x_j, \xi_i) - \E[\varphi_1(x_j,\xi)]\right\|.
    \end{align*}
    Hence, we have 
    \begin{align}
        &\sup_{x\in X}\left\|\frac{1}{n}\sum_{i=1}^n \varphi_1(x,\xi_i)-\E[\varphi_1(x,\xi)]\right\|\nonumber\\
        \;\leq\; & \max_{1\leq j\leq N} \sup_{x:\|x-x_j\|\leq \delta}\E[\|\varphi_1(x,\xi)-\varphi_1(x_j,\xi)\|] + \max_{1\leq j\leq N} \sup_{x:\|x-x_j\|\leq \delta}\frac{1}{n}\sum_{i=1}^n \left\|\varphi_1(x,\xi_i)-\varphi_1(x_j,\xi_i)\right\|\nonumber\\ &+ \max_{1\leq j\leq N} \left\|\frac{1}{n}\sum_{i=1}^n \varphi_1(x_j, \xi_i) - \E[\varphi_1(x_j,\xi)]\right\|. \label{ineq:ULLN-cover}
    \end{align}
    Now let $n\to\infty$ and we analyze the right hand side of \eqref{ineq:ULLN-cover} term by term. For the second term, we first note that for fixed $n$, 
    \[
        \max_{1\leq j\leq N} \sup_{x:\|x-x_j\|\leq \delta}\frac{1}{n}\sum_{i=1}^n \left\|\varphi_1(x,\xi_i)-\varphi_1(x_j,\xi_i)\right\| \leq \max_{1\leq j\leq N} \frac{1}{n}\sum_{i=1}^n \sup_{x:\|x-x_j\|\leq \delta}\left\|\varphi_1(x,\xi_i)-\varphi_1(x_j,\xi_i)\right\|.
    \]
    Since $\|\varphi_1(x,\xi_i)-\varphi_1(x_j,\xi_i)\|$ is uniformly bounded by $2$, by the strong law of large numbers, it follows that
    \[
        \limsup_{n\to\infty}\max_{1\leq j\leq N} \sup_{x:\|x-x_j\|\leq \delta}\frac{1}{n}\sum_{i=1}^n \left\|\varphi_1(x,\xi_i)-\varphi_1(x_j,\xi_i)\right\|\leq \max_{1\leq j\leq N} E\left[\sup_{x:\|x-x_j\|\leq \delta}\|\varphi_1(x,\xi)-\varphi_1(x_j,\xi)\|\right]
    \]
    almost surely.
    Similarly, for the third term, by the strong law of large numbers, it follows that 
    \[
        \max_{1\leq j\leq N} \left\|\frac{1}{n}\sum_{i=1}^n \varphi_1(x_j, \xi_i) - \E[\varphi_1(x_j,\xi)]\right\| \xrightarrow{\text{a.s.}} 0.
    \]
    For the first term, it is independent of $n$ and we note that
    \[
        \sup_{x:\|x-x_j\|\leq \delta}\E[\|\varphi_1(x,\xi)-\varphi_1(x_j,\xi)\|] \leq E\left[\sup_{x:\|x-x_j\|\leq \delta}\|\varphi_1(x,\xi)-\varphi_1(x_j,\xi)\|\right].
    \]
    Combining the above upper bound for the three terms, we have
    \begin{align}
        \limsup_{n\to\infty}\sup_{x\in X}\left\|\frac{1}{n}\sum_{i=1}^n \varphi_1(x,\xi_i)-\E[\varphi_1(x,\xi)]\right\| \leq 2\max_{1\leq j\leq N} \E\left[\sup_{x:\|x-x_j\|\leq \delta}\|\varphi_1(x,\xi)-\varphi_1(x_j,\xi)\|\right], \label{ineq:ULLN-limsup}
    \end{align}
    almost surely. Finally, it suffices to show that the right hand side of \eqref{ineq:ULLN-limsup} converges to $0$ as $\delta\to 0$. To see this, we note that by Lemma \ref{lemma:phi-1-ub}, it holds that
    \begin{align*}
        \E\left[\sup_{x:\|x-x_j\|\leq \delta}\|\varphi_1(x,\xi)-\varphi_1(x_j,\xi)\|\right] \;\leq\; &2 \prob(\|\xi-x_j\|\leq 2\delta) + \E\left[\frac{2\delta}{\|\xi-x_j\|}\mathds{1}_{\{\|\xi-x_j\|>2\delta\}}\right]\\
        \;\leq\; & 2 \prob(\|\xi-x_j\|\leq 2\delta) + \E\left[\frac{2\delta}{\|\xi-x_j\|}\right].
    \end{align*}
    Hence, it follows that
    \begin{align*}
        \text{(RHS of \eqref{ineq:ULLN-limsup})} \;\leq\; &4\max_{1\leq j\leq N} \prob(\|\xi-x_j\|\leq 2\delta) + 4\delta\max_{1\leq j\leq N} \E\left[\frac{1}{\|\xi-x_j\|}\right]\\
        \;\leq\; &4\sup_{x\in X} \prob(\|\xi-x\|\leq 2\delta) + 4\delta \sup_{x\in X} \E\left[\frac{1}{\|\xi-x\|}\right]. 
    \end{align*}
    By Lemma \ref{lemma:sup-P-E}, we conclude that (RHS of \eqref{ineq:ULLN-limsup}) tends to $0$ as $\delta\to 0$ and complete the proof.
\end{proof}

\begin{lemma}\label{lemma:h-ub}
    Let $X\subset \R^p$ and $Y\subset \{\Sigma\in \PD: \tr(\Sigma) = p\}$ be compact sets. Let $\xi$ be arbitrary but fixed. Define 
    \[
        h(\mu,\Sigma,\xi) \Let \frac{(\xi-\mu)(\xi-\mu)^\top}{(\xi-\mu)^\top\Sigma^{-1}(\xi-\mu)}.
    \]
    Then for any $\delta>0$ and $(\mu,\Sigma), (\mu',\Sigma')\in X\times Y$ such that $\|\mu-\mu'\| + \|\Sigma-\Sigma'\| \leq \delta$, it holds that 
    \begin{align}
        \|h(\mu,\Sigma,\xi) - h(\mu',\Sigma',\xi)\|_{\F} \leq 2p\mathds{1}_{\{\|\xi-\mu'\|\leq 2\delta\}} + C\delta + C\frac{\delta}{\|\xi-\mu'\|}\mathds{1}_{\{\|\xi-\mu'\|>2\delta\}}, \label{ineq:oscillation-bd-mu-Sigma}
    \end{align}
    where $C$ is a constant only depending on choice of set $X$ and $Y$. 
\end{lemma}
\begin{proof}
    We first consider some constants that only depend on the choice of $X$ and $Y$ or the whole space. 
    Since $Y$ is a compact subset of set of positive definite matrices, we have
    \[
        \underline{\lambda} \Let \inf_{\Sigma\in Y} \lambda_{\min}(\Sigma) > 0. 
    \]
    Moreover, since $\tr(\Sigma)=p$ for all $\Sigma\in Y$, we have 
    \[
        \sup_{\Sigma\in Y}\lambda_{\max}(\Sigma) \leq p. 
    \]
    Now we note that 
    \[
        \|h(\mu,\Sigma,\xi) - h(\mu',\Sigma',\xi)\|_{\F} \leq \|h(\mu,\Sigma,\xi)-h(\mu',\Sigma,\xi)\|_{\F} + \|h(\mu',\Sigma,\xi)-h(\mu',\Sigma',\xi)\|_{\F}. 
    \]
    In the following, we show that 
    \begin{align}
        \label{ineq:oscillation-bd-mu-Sigma-1}
        \|h(\mu,\Sigma,\xi)-h(\mu',\Sigma,\xi)\|_{\F} \leq 2p\mathds{1}_{\{\|\xi-\mu'\|\leq 2\delta\}} + C_1 \frac{\delta}{\|\xi-\mu'\|} \mathds{1}_{\{\|\xi-\mu'\|\geq 2\delta\}}
    \end{align}
    and 
    \begin{align}
        \label{ineq:oscillation-bd-mu-Sigma-2}
        \|h(\mu',\Sigma,\xi)-h(\mu',\Sigma',\xi)\|_{\F} \leq C_2\delta,
    \end{align}
    where $C_1$ and $C_2$ are constants only depending on $\underline{\lambda}$ and $p$. Then combining the above three inequalities, we show that \eqref{ineq:oscillation-bd-mu-Sigma} holds.

    Now we show that \eqref{ineq:oscillation-bd-mu-Sigma-1} holds. 
    First note that a crude bound gives that 
    \[
        \|h(\mu,\Sigma,\xi)\|_{\F} = \frac{\|\xi-\mu\|_2^2}{(\xi-\mu)^\top\Sigma^{-1}(\xi-\mu)} \leq \lambda_{\max}(\Sigma) \leq p,
    \]
    where the equality is by the fact that $\|xx^\top\|_{\text{F}}=\|x\|_2^2$, and the first inequality is by $\|x\|_2^2 \leq \lambda_{\max}(\Sigma) x^\top\Sigma^{-1}x$. Then it holds that
    \[
        \|h(\mu,\Sigma,\xi)-h(\mu',\Sigma,\xi)\|_{\F} \leq 2p.
    \]
    Then to show \eqref{ineq:oscillation-bd-mu-Sigma-1}, it suffices to show that if $\|\xi-\mu'\|\geq 2\delta$, then
    \[
        \|h(\mu,\Sigma,\xi)-h(\mu',\Sigma,\xi)\|_{\F} \leq  C_1 \frac{\delta}{\|\xi-\mu'\|}.
    \]
    Now assume that $\|\xi-\mu'\|\geq 2\delta$ and let $v\Let \Sigma^{-1/2}(\xi-\mu)$ and $u\Let \Sigma^{-1/2}(\xi-\mu')$. Then we have
    \begin{align}
        \|h(\mu,\Sigma,\xi) - h(\mu',\Sigma,\xi)\|_{\F}  \;=\; &\left\|\Sigma^{1/2} \frac{vv^\top}{\|v\|_2^2}\Sigma^{1/2}-\Sigma^{1/2} \frac{uu^\top}{\|u\|_2^2}\Sigma^{1/2}\right\|_{\F}\nonumber\\
        \;\leq\; &\lambda_{\max}(\Sigma)\left\| \frac{vv^\top}{\|v\|_2^2}- \frac{uu^\top}{\|u\|_2^2}\right\|_{\F}\nonumber\\
        \;=\; &\lambda_{\max}(\Sigma) \left\|\overline{v}(\overline{v}-\overline{u})^\top+(\overline{v}-\overline{u})\overline{u}^\top\right\|_{\F}
        \quad (\overline{v}\Let v/\|v\|,\overline{u}\Let u/\|u\|)
        \nonumber\\
        \;\leq\; &p \left(\|\overline{v}\|\|\overline{v}-\overline{u}\| + \|\overline{u}\|\|\overline{v}-\overline{u}\|\right)\nonumber\\
        \;=\; &2p\left\|\frac{v}{\|v\|}-\frac{u}{\|u\|}\right\|.\label{ineq:2p-v-u}
    \end{align}
    Since $\|\mu-\mu'\|\leq \delta$ and $\|\xi-\mu'\|\geq 2\delta$, we note that $\xi$ does not fall on the line segment that connects $\mu$ and $\mu'$, and hence, $tv+(1-t)u\neq 0$ for all $t\in [0,1]$. Then by Lemma \ref{lemma:g-z}, it holds that 
    \begin{align*}
        (\text{RHS of}~\eqref{ineq:2p-v-u}) \;\leq\; &2p\sup_{t\in [0,1]} \frac{1}{\|tu+(1-t)v\|}\|u-v\|\\
        \;\leq\; &2p\sup_{t\in [0,1]} \frac{1}{\lambda_{\min} (\Sigma^{-1/2})\|t(\xi-\mu)+(1-t)(\xi-\mu')\|} \lambda_{\max} (\Sigma^{-1/2})\|\mu-\mu'\|\\
       \;=\; &2p \frac{\lambda_{\max}(\Sigma^{-1/2})}{\lambda_{\min}(\Sigma^{-1/2})} \sup_{t\in [0,1]} \frac{1}{\|\xi-z_t\|} \|\mu-\mu'\| \quad  (\text{Let}~z_t\Let t\mu+(1-t)\mu')\\
        \;\leq\; &4p\sqrt{\frac{p}{\underline{\lambda}}}\frac{\|\mu-\mu'\|}{\|\xi-\mu'\|}\\
        \;\leq\; &4p\sqrt{\frac{p}{\underline{\lambda}}}\frac{\delta}{\|\xi-\mu'\|},
    \end{align*}
    where the third inequality is by the inequality
        \[
        \|\xi-z_t\| \geq \|\xi-\mu'\| -\|z_t-\mu'\| \geq \|\xi-\mu'\| - \delta \geq \frac{1}{2}\|\xi-\,\mu'\|.
    \]
    The above completes the proof of \eqref{ineq:oscillation-bd-mu-Sigma-1} with $C_1 = 4p\sqrt{\frac{p}{\underline{\lambda}}}$.

    Now we show that \eqref{ineq:oscillation-bd-mu-Sigma-2} holds. We have
    \begin{align*}
        \|h(\mu',\Sigma,\xi)-h(\mu',\Sigma',\xi)\|_{\F} \;=\; &\left\|\frac{(\xi-\mu')(\xi-\mu')^\top}{(\xi-\mu')^\top\Sigma^{-1}(\xi-\mu')}-\frac{(\xi-\mu')(\xi-\mu')^\top}{(\xi-\mu')^\top\Sigma'^{-1}(\xi-\mu')}\right\|_{\F}\\
        \;\leq\; &\|\xi-\mu'\|_2^2 \left|\frac{1}{(\xi-\mu')^\top\Sigma^{-1}(\xi-\mu')}-\frac{1}{(\xi-\mu')^\top\Sigma'^{-1}(\xi-\mu')}\right|\\
        \;=\; &\|\xi-\mu'\|_2^2\frac{\left|(\xi-\mu')^\top\Sigma'^{-1}(\xi-\mu')-(\xi-\mu')^\top\Sigma^{-1}(\xi-\mu')\right|}{(\xi-\mu')^\top\Sigma^{-1}(\xi-\mu')(\xi-\mu')^\top\Sigma'^{-1}(\xi-\mu')}\\
        \;\leq\; &\frac{\left|(\xi-\mu')^\top(\Sigma'^{-1}-\Sigma^{-1})(\xi-\mu')\right|}{\|\xi-\mu'\|_2^2/p^2}\\
        \;\leq\; &p^2\|\Sigma'^{-1}-\Sigma^{-1}\|_{\F},
    \end{align*}
    where the second inequality is by $(\xi-\mu')^\top\Sigma^{-1}(\xi-\mu') \geq \|\xi-\mu'\|_2^2\lambda_{\min}(\Sigma^{-1})=\|\xi-\mu'\|_2^2/\lambda_{\max}(\Sigma)\geq \|\xi-\mu'\|_2^2/p$, and the third inequality is by $\left|(\xi-\mu')^\top(\Sigma'^{-1}-\Sigma^{-1})(\xi-\mu')\right|\leq \|\xi-\mu'\|_2^2 \|\Sigma'^{-1}-\Sigma^{-1}\|_{\text{F}}$. Finally, we note that 
    \[
        \|\Sigma'^{-1}-\Sigma^{-1}\|_{\text{F}} = \|\Sigma'^{-1}(\Sigma-\Sigma')\Sigma^{-1}\|_{\F} \leq \lambda_{\max}(\Sigma^{-1}) \lambda_{\max}(\Sigma'^{-1}) \|\Sigma-\Sigma'\|_{\F} \leq \delta/\underline{\lambda}^2.
    \]
    Hence, we conclude that
    \[
        \|h(\mu',\Sigma,\xi)-h(\mu',\Sigma',\xi)\|_{\F} \leq p^2 \delta/\underline{\lambda}^2
    \]
    and completes the proof of \eqref{ineq:oscillation-bd-mu-Sigma-2} with $C_2=p^2/\underline{\lambda}^2$.
\end{proof}

\begin{lemma}\label{lemma:h-ULLN}
    Let $X\subset \R^p$ and $Y\subset \{\Sigma\in \PD: \tr(\Sigma) = p\}$ be compact sets. Then 
    \begin{align}
        \sup_{(\mu,\Sigma)\in X\times Y} \left\|\frac{1}{n}\sum_{i=1}^n\varphi_2(\mu,\Sigma,\xi_i)-\E[\varphi_2(\mu,\Sigma,\xi)]\right\|_{\F} \asto 0, \; \text{as} \; n\to\infty
    \end{align}
\end{lemma}

\begin{proof}
    Let $h(\mu,\Sigma,\xi) \Let \frac{(\xi-\mu)(\xi-\mu)^\top}{(\xi-\mu)^\top\Sigma^{-1}(\xi-\mu)}$. 
    It suffices to show that 
    \begin{align}
        \sup_{(\mu,\Sigma)\in X\times Y} \left\|\frac{1}{n}\sum_{i=1}^n h(\mu,\Sigma,\xi_i)-\E[h(\mu,\Sigma,\xi)]\right\|_{\F} \asto 0 \; \text{as} \; n\to\infty.
    \end{align}
    The proof follows a similar idea 
    to that in the proof of Lemma \ref{lemma:phi1-ULLN}.
    Let $\mathbb{B}((\mu,\Sigma), r)\Let \{(\mu',\Sigma')\in X\times Y: \|\mu-\mu'\| + \|\Sigma-\Sigma'\|_{\F} \leq r\}$. Let $(\mu_1,\Sigma_1), \ldots, (\mu_N,\Sigma_N)$ be centers of a $\delta$-net of $X\times Y$, i.e., $X\times Y\subset \bigcup_{j=1}^N \mathbb{B}((\mu_j,\Sigma_j),\delta)$. Then for any $(\mu,\Sigma)\in X\times Y$ and $(\mu_j,\Sigma_j)$ satisfying $\|\mu-\mu_j\|+\|\Sigma-\Sigma_j\|_{\F}\leq \delta$, it holds that
    \begin{align*}
        &\left\|\frac{1}{n}\sum_{i=1}^n h(\mu,\Sigma,\xi_i)-\E[h(\mu,\Sigma,\xi)]\right\|_{\F}\\
        \leq\; &\|\E[h(\mu,\Sigma,\xi)]-\E[h(\mu_j,\Sigma_j,\xi)]\|_{\F} + \left\|\frac{1}{n}\sum_{i=1}^n h(\mu,\Sigma,\xi_i)-\frac{1}{n}\sum_{i=1}^n h(\mu_j,\Sigma_j,\xi_i)\right\|_{\F} \\
        &+ \left\|\frac{1}{n}\sum_{i=1}^n h(\mu_j,\Sigma_j,\xi_i)-\E[h(\mu_j,\Sigma_j,\xi)]\right\|_{\F}\\
        \leq\; &\E[\|h(\mu,\Sigma,\xi)-h(\mu_j,\Sigma_j,\xi)\|_{\F}] + \left\|\frac{1}{n}\sum_{i=1}^n h(\mu,\Sigma,\xi_i)-\frac{1}{n}\sum_{i=1}^n h(\mu_j,\Sigma_j,\xi_i)\right\|_{\F} \\
        &+ \left\|\frac{1}{n}\sum_{i=1}^n h(\mu_j,\Sigma_j,\xi_i)-\E[h(\mu_j,\Sigma_j,\xi)]\right\|_{\F}.
    \end{align*}
    Hence, we have
    \begin{align}
        &\sup_{(\mu,\Sigma)\in X\times Y} \left\|\frac{1}{n}\sum_{i=1}^n h(\mu,\Sigma,\xi_i)-\E[h(\mu,\Sigma,\xi)]\right\|_{\F}\nonumber\\
        \leq\; &\max_{1\leq j\leq N} \sup_{\|\mu-\mu_j\|+\|\Sigma-\Sigma_j\|_{\F}\leq \delta} \E[\|h(\mu,\Sigma,\xi)-h(\mu_j,\Sigma_j,\xi)\|_{\F}] \nonumber\\
        &+ \max_{1\leq j\leq N} \sup_{\|\mu-\mu_j\|+\|\Sigma-\Sigma_j\|_{\F}\leq \delta} \frac{1}{n}\sum_{i=1}^n \|h(\mu,\Sigma,\xi_i)-h(\mu_j,\Sigma_j,\xi_i)\|_{\F}\nonumber\\
        &+\max_{1\leq j\leq N} \left\|\frac{1}{n}\sum_{i=1}^n h(\mu_j,\Sigma_j,\xi_i)-\E[h(\mu_j,\Sigma_j,\xi)]\right\|_{\F}. \label{ineq:h-ULLN-cover}
    \end{align}
    Now let $n\to \infty$ and we analyze the right hand side of \eqref{ineq:h-ULLN-cover} term by term. For the second term, we first note that for fixed $n$, 
    \begin{align*}
        &\max_{1\leq j\leq N} \sup_{\|\mu-\mu_j\|+\|\Sigma-\Sigma_j\|_{\F}\leq \delta} \frac{1}{n}\sum_{i=1}^n \|h(\mu,\Sigma,\xi_i)-h(\mu_j,\Sigma_j,\xi_i)\|_{\F} \\
        \leq \; &\max_{1\leq j\leq N}  \frac{1}{n}\sum_{i=1}^n \sup_{\|\mu-\mu_j\|+\|\Sigma-\Sigma_j\|_{\F}\leq \delta} \|h(\mu,\Sigma,\xi_i)-h(\mu_j,\Sigma_j,\xi_i)\|_{\F}.
    \end{align*}
    Since $\|h(\mu,\Sigma,\xi)\|_{\F} \leq \lambda_{\max}(\Sigma) \leq \tr(\Sigma) = p$, by the strong law of large numbers, it follows that
    \begin{align*}
        &\limsup_{n\to\infty}\max_{1\leq j\leq N} \sup_{\|\mu-\mu_j\|+\|\Sigma-\Sigma_j\|_{\F}\leq \delta} \frac{1}{n}\sum_{i=1}^n \|h(\mu,\Sigma,\xi_i)-h(\mu_j,\Sigma_j,\xi_i)\|_{\F}\\  \leq &\max_{1\leq j\leq N} \E\left[\sup_{\|\mu-\mu_j\|+\|\Sigma-\Sigma_j\|_{\F}\leq \delta} \|h(\mu,\Sigma,\xi_i)-h(\mu_j,\Sigma_j,\xi_i)\|_{\F}\right]~\text{almost surely.}
    \end{align*}
    Similarly, for the third term, by the strong law of large numbers, 
    \[
        \max_{1\leq j\leq N} \left\|\frac{1}{n}\sum_{i=1}^n h(\mu_j,\Sigma_j,\xi_i)-\E[h(\mu_j,\Sigma_j,\xi)]\right\|_{\F} \asto 0.
    \]
    For the first term, it is independent of $n$ and we note that
    \[
        \sup_{\|\mu-\mu_j\|+\|\Sigma-\Sigma_j\|_{\F}\leq \delta} \E[\|h(\mu,\Sigma,\xi)-h(\mu_j,\Sigma_j,\xi)\|_{\F}] \leq  \E\left[\sup_{\|\mu-\mu_j\|+\|\Sigma-\Sigma_j\|_{\F}\leq \delta} \|h(\mu,\Sigma,\xi)-h(\mu_j,\Sigma_j,\xi)\|_{\F}\right].
    \]
    Combining the above upper bound for the three terms, we have 
    \begin{align}
        &\limsup_{n\to\infty} \sup_{(\mu,\Sigma)\in X\times Y}\left\|\frac{1}{n}\sum_{i=1}^n h(\mu,\Sigma,\xi_i)-\E[h(\mu,\Sigma,\xi)]\right\|_{\F} \nonumber\\
        \leq &2 \max_{1\leq j\leq N} \E\left[\sup_{\|\mu-\mu_j\|+\|\Sigma-\Sigma_j\|_{\F}\leq \delta} \|h(\mu,\Sigma,\xi_i)-h(\mu_j,\Sigma_j,\xi_i)\|_{\F}\right],\label{ineq:h-ULLN-limsup}
    \end{align}
    almost surely. Finally, it suffices to show that the right hand side of \eqref{ineq:h-ULLN-limsup} converges to $0$ as $\delta\to 0$. To see this, we note that by Lemma \ref{lemma:h-ub}, it holds that 
    \begin{align*}
        \E\left[\sup_{\|\mu-\mu_j\|+\|\Sigma-\Sigma_j\|_{\F}\leq \delta} \|h(\mu,\Sigma,\xi_i)-h(\mu_j,\Sigma_j,\xi_i)\|_{\F}\right] \leq 2p\prob(\|\xi-\mu_j\| \leq 2\delta)  + C\delta + C\delta\E\left[\frac{1}{\|\xi-\mu_j\|}\right].
    \end{align*}
    Hence, it follows that
    \begin{align*}
        \text{(RHS of \eqref{ineq:h-ULLN-limsup})} \leq &4p\max_{1\leq j\leq N}\prob(\|\xi-\mu_j\| \leq 2\delta)  + 2C\delta + 2C\delta\max_{1\leq j\leq N}\E\left[\frac{1}{\|\xi-\mu_j\|}\right]\\
        \leq &4p\sup_{\mu\in X}\prob(\|\xi-\mu\| \leq 2\delta)  + 2C\delta + 2C\delta\sup_{\mu\in X}\E\left[\frac{1}{\|\xi-\mu\|}\right].
    \end{align*}
    By Lemma \ref{lemma:sup-P-E}, we conclude that the above tends to $0$ as $\delta\to 0$ and completes the proof.
\end{proof}

\begin{lemma}\label{lemma:ULLN}
    Let $X\subset \R^p$ and $Y\subset \{\Sigma\in \PD: \tr(\Sigma) = p\}$ be compact sets. Then it holds that 
    \[
        \sup_{(\mu,\Sigma)\in X\times Y} \left\|\Phi_n(\mu,\Sigma)-\Phi(\mu,\Sigma)\right\| \asto 0.
    \]
\end{lemma}

\begin{proof}
    The result is a direct consequence of Lemma \ref{lemma:phi1-ULLN} and \ref{lemma:h-ULLN}. Indeed, 
    \begin{align*}
        \left\|\Phi_n(\mu,\Sigma)-\Phi(\mu,\Sigma)\right\| = \left\|\frac{1}{n}\sum_{i=1}^n \varphi_1(\mu,\xi_i)-\E[\varphi_1(\mu,\xi)]\right\| + \left\|\frac{1}{n}\sum_{i=1}^n\varphi_2(\mu,\Sigma,\xi_i)-\E[\varphi_2(\mu,\Sigma,\xi)]\right\|_{\F}.
    \end{align*}
    Then 
    \begin{align*}
        \sup_{(\mu,\Sigma)\in X\times Y} \left\|\Phi_n(\mu,\Sigma)-\Phi(\mu,\Sigma)\right\| \leq &\sup_{\mu\in X} \left\|\frac{1}{n}\sum_{i=1}^n \varphi_1(x,\xi_i)-\E[\varphi_1(x,\xi)]\right\| \\&+\sup_{(\mu,\Sigma)\in X\times Y} \left\|\frac{1}{n}\sum_{i=1}^n\varphi_2(\mu,\Sigma,\xi_i)-\E[\varphi_2(\mu,\Sigma,\xi)]\right\|_{\F} \asto 0.
    \end{align*}
\end{proof}

\begin{lemma}\label{lemma:inf-Phi-ge-0}
    Let $X\subset \R^p$ and $Y\subset \{\Sigma\in \PD: \tr(\Sigma) = p\}$ be compact sets. 
    Let $\Phi(\mu,\Sigma)$ be defined as in \eqref{eq:expected-phi-ERENJ}.
    Then for any $\delta>0$, it holds that 
    \[
        \inf_{(\mu,\Sigma)\in X\times Y: \|\mu-\mu_0\|+\|\Sigma-\Sigma_0\|_{\F}\geq \delta} \|\Phi(\mu,\Sigma)\| > 0.
    \]
\end{lemma}

\begin{proof}
By Proposition~\ref{prop:est-eq-mu-Sigma},  $\Phi(\mu,\Sigma)=0$ has a unique solution
$(\mu_0,\Sigma_0)$ over the set $\R^p\times \{\Sigma\in \PD: \tr(\Sigma) = p\}$. Thus it suffices to show that
$\Phi(\mu,\Sigma)$ is continuous. 

In the following, we show that $\E[\varphi_1(\mu,\xi)]$ and $\E[\varphi_2(\mu,\Sigma,\xi)]$  are continuous in $\mu$ and $(\mu,\Sigma)$ respectively.
   Note that $\varphi_1(\mu,\xi)$ is a function parameterized 
   by $\xi$.   
By the definition 
of $\varphi_1(\cdot,\xi)$, it  is continuous 
at any point $\mu'\neq \xi$.
  Moreover, for each fixed $\mu$, $\|\varphi_1(\mu,\xi)\|\leq 1$ uniformly for all $\xi\neq \mu$.
  Thus
    \begin{align*}
    \lim_{\mu\to\mu'}\E[\varphi_1(\mu,\xi)] \;=\; &\lim_{\mu\to\mu'}\int_{\Xi\setminus\{\mu'\} }\varphi_1(\mu,\xi) P(\diff \xi) 
        =\; \int_{\Xi\setminus \{\mu'\} }\lim_{\mu\to\mu'} \varphi_1(\mu,\xi) P(\diff \xi)\\
        =\; &\int_{\Xi\setminus \{\mu'\}} \varphi_1(\mu',\xi) P(\diff \xi)
        =\; \E[\varphi_1(\mu',\xi)],
    \end{align*}
    where the second equality is ensured 
    by the dominated convergence theorem.
    This shows continuity of $\E[\varphi_1(\mu,\xi)]$ 
    as desired.

    Next, we  show continuity of 
    $\E[\varphi_2(\mu,\Sigma,\xi)]$ 
    in $(\mu,\Sigma)$.
    Since
\begin{align}
\label{eq:phi_2-EDRO-Renjie-lema9}    
\varphi_2(\mu,\Sigma,\xi)\Let 
\Sigma/p -h(\mu,\Sigma,\xi)
=
\Sigma/p - \frac{(\xi-\mu)(\xi-\mu)^\transpose}{(\xi-\mu)^\transpose\Sigma^{-1}(\xi-\mu)},
\end{align}
and     $\|h(\mu,\Sigma,\xi)\|_{\F} \leq \lambda_{\max}(\Sigma)\leq \tr(\Sigma) = p$, then we can  
also use the dominated convergence theorem to show the continuity of
$\E[\varphi_2(\mu,\Sigma,\xi)]$.
\end{proof}

\begin{lemma}\label{lemma:at-most-p}
    Let $\xi\in \R^p$ be a random vector with distribution absolutely continuous with respect to Lebesgue measure on $\R^p$. Let $\xi_1,\ldots,\xi_{p+1}$ be $p+1$ i.i.d. samples of $\xi$. Then 
    \[
        \prob(\xi_1,\ldots,\xi_{p+1}~\text{lie on a common hyperplane}) = 0.
    \]
\end{lemma}

\begin{proof}
    Define
    \[
        H_k \Let \operatorname{aff}(\xi_1,\ldots,\xi_k),
        \qquad k=1,\ldots,p+1,
    \]
    where $\operatorname{aff}(\xi_1,\ldots,\xi_k)$ is the affine hull of $\xi_1,\ldots,\xi_k$.
    For each $k=1,\ldots,p$, the affine subspace $H_k$
    has dimension at most $k-1\leq p-1$ and hence has
    zero Lebesgue measure in $\R^p$.
    Since $\xi_{k+1}$ is independent of
    $(\xi_1,\ldots,\xi_k)$, we have
    \[
        \prob\bigl(
            \xi_{k+1}\in H_k
            \mid \xi_1,\ldots,\xi_k
        \bigr)
        =0
        \quad\text{almost surely}.
    \]
    Taking expectations of the above yields
    \[
        \prob(\xi_{k+1}\in H_k)=0,
        \qquad k=1,\ldots,p.
    \]
    Therefore, with probability one,
    $\xi_{k+1}\notin H_k$ simultaneously for all
    $k=1,\ldots,p$.

    On this probability-one event, adding $\xi_{k+1}$
    increases the dimension of the affine hull by one:
    \[
        \dim H_{k+1}=\dim H_k+1,
        \qquad k=1,\ldots,p.
    \]
    Since $\dim H_1=0$, it follows that
    $\dim H_{p+1}=p$ almost surely.
    Consequently,
    \[
        \prob\bigl(
            \xi_1,\ldots,\xi_{p+1}
            \text{ lie on a common hyperplane}
        \bigr)
        =
        \prob(\dim H_{p+1}<p)
        =0.
    \]
\end{proof}

\subsection{Proof of Proposition \ref{prop:E-uu}}

\begin{proof}[Proof of Proposition \ref{prop:E-uu}]
    Observe that $u$ is law-invariant under orthogonal transformation, i.e.,
    for any orthogonal matrix $Q$, 
    $u\diseq Qu$. Thus 
    $    Q\E[uu^\top]Q^\top = \E[uu^\top]$.
    Denote $\E[uu^\top]$ by $U$. 
    For any $x,y\in \mathbb{U}^{p-1}$, there exists an orthogonal matrix
    $\tilde{Q}$
    such that $x=\tilde{Q}y$. 
    Let $Q=\tilde{Q}$.
    Then 
    \[
        x^\top U x = (Qy)^\top U Qy = y^\top Q^\top U Q y = y^\top U y.
    \]
The relationship above 
implies that, there exists constant $c\geq 0$ such that    
    for any $x \in \mathbb{U}^{p-1}$, $x^\top U x = c =cx^\top x$. The latter implies 
$ x^\top(U-cI)x = 0$
for all $x \in \mathbb{U}^{p-1}$, 
    and hence $U=cI$.
        Finally, 
        we show $c=1/p$. 
   This follows from the fact that
    $
        cp = \tr(U) = \E[\|u\|^2] = 1.
    $
The proof is complete.
\end{proof}

\subsection{Proof of Theorem~\ref{thm:consist-gibbs}}

\begin{proof}[Proof of Theorem~\ref{thm:consist-gibbs}]
    Let $\delta> 0$ and $\eta>0$ be arbitrary but fixed. To ease the exposition, let
    \[
    A_\delta\triangleq\{\theta\in\Theta: \|\theta-\theta_0\|\geq \delta\}, \quad \Pi_n(A_\delta)\triangleq \Pi_n^\eta(A_\delta\mid \xi_1,\ldots,\xi_n).
    \]
   We show that $\Pi_n^\eta(A_\delta)\to 0$ almost surely as $n\to\infty$.
    To 
    this end, we first reformulate $\Pi_n^\eta(A_\delta)$ by
    \begin{align}
        \Pi_n^\eta(A_\delta) 
        \;=\; \frac{\int_{\theta\in A_\delta}\exp\left(-\eta n\left\|\Phi_n(\theta)\right\|\right)\pi(\theta)\diff \theta}{\int_{\theta \in\Theta}\exp\left(-\eta n\left\|\Phi_n(\theta)\right\|\right)\pi(\theta)\diff \theta}.
        \label{eq:Pi-n}
    \end{align}
    Denote the numerator and denominator of~\eqref{eq:Pi-n} by $N_n^\eta(A_\delta)$ and $D_n^\eta$, respectively. We claim that
    \begin{itemize}
        \item[(i)] $N_n^\eta(A_\delta)\leq \exp(-c_1\eta n)$ almost surely as $n\to\infty$ for some constant $c_1>0$ depending only on $\delta$;
        
        \item[(ii)] $D_n^\eta\geq \exp(-c_2\eta n)$ almost surely as $n\to\infty$ for any constant  $c_2>0$.
    \end{itemize}
    Suppose the above two claims hold and set $c_2<c_1$. Then, for all sufficiently large $n$, both $N_n^\eta(A_\delta)\leq \exp(-c_1 \eta n)$ and $D_n^\eta\geq \exp(-c_2\eta n)$ hold almost surely. This gives
    \[
        \Pi_n^\eta(A_\delta) = \frac{N_n^\eta(A_\delta)}{D_n^\eta} \leq \frac{\exp(-c_1\eta n)}{\exp(-c_2\eta n)} = \exp(-(c_1-c_2)\eta n) \to 0 \quad \text{as } n\to\infty,
    \]
    where the inequality holds almost surely for sufficiently large $n$. This shows \eqref{ineq:as-event} and $\Pi_n^\eta(A_\delta)\to 0$ almost surely. In the rest of the proof, we prove statements~(i) and (ii).

    \underline{Statement (i).}
    Since $\int_{\theta\in A_\delta}\pi(\theta)\diff\theta \leq 1$, we have
    \[
        N_n^\eta(A_\delta) = \int_{\theta\in A_\delta}\exp(-\eta n\|\Phi_n(\theta)\|)\pi(\theta)\diff\theta
        \leq \sup_{\theta\in A_\delta}\exp(-\eta n\|\Phi_n(\theta)\|) = \exp\!\Big(-\eta n\inf_{\theta\in A_\delta}\|\Phi_n(\theta)\|\Big).
    \]
    Let
    \[
        Z_\delta
        \Let
        \liminf_{n\to\infty}
        \inf_{\theta\in A_\delta}
        \|\Phi_n(\theta)\|.
    \]
    By condition (b), $Z_\delta>0$ almost surely. We claim that there exists a constant $z_\delta\in \R_+$ such that $Z_\delta = z_\delta$ almost surely. To show this, note that $Z_\delta$ is unchanged by any finite permutation of the observations $\xi_1,\ldots,\xi_n,\ldots$ Then the Hewitt--Savage zero--one law~\citep[Theorem 11.3]{ref:hewitt1955symmetric} implies that for any $z_\delta'\in \R_+$, either $\P(Z_\delta\leq z_\delta') = 0$ or $\P(Z_\delta\leq z_\delta') = 1$. Since $\P(Z_\delta > 0) = 1$, the claim holds.

   To show (i), by the above claim, there exists a constant $c_1>0$ (depending only on $\delta$), say $c_1=z_\delta/2$, such that
    \[
        \inf_{\theta\in A_\delta}\|\Phi_n(\theta)\| \geq c_1 \quad \text{for all sufficiently large } n, \quad \text{a.s.}
    \]
    Therefore, almost surely for all sufficiently large $n$,
    $
        N_n^\eta(A_\delta) \leq \exp(-c_1 \eta n),
    $
    which completes the proof.

    \underline{Statement (ii).} Let $c_2>0$ be arbitrary but fixed, let 
    \bgeq 
    B_{c_2}\triangleq\{\theta\in\Theta: \|\Phi(\theta)\|\leq c_2/2\}.
    \edeq 
    Then
    \begin{align}
        \exp(c_2\eta n)D_n^\eta\;=\;&\int_{\theta\in\Theta}\exp\big(\eta n(c_2-\|\Phi_n(\theta)\|)\big)\pi(\theta)\diff \theta\nonumber\\
        \;\geq\;&\int_{\theta\in B_{c_2}}\exp\big(\eta n(c_2-\|\Phi_n(\theta)\|)\big)\pi(\theta)\diff \theta.\label{ineq:clm-b-1}
    \end{align}
    Denote the right hand side of~\eqref{ineq:clm-b-1} by $\widetilde{D}_n^\eta$. 
    To show claim (b), it suffices to prove that $\lim_{n\to\infty}\widetilde{D}_n^\eta=\infty$ almost surely. Under 
    condition (a), 
    for every fixed $\theta\in \Theta$, the law of large numbers implies that $\Phi_n(\theta)\to\Phi(\theta)$ almost surely. Then 
    $
        \lim_{n\to\infty} \|\Phi_n(\theta)\| \leq c_2/2,~~\text{a.s.},
    $
    for every fixed $\theta\in B_{c_2}$.
    Consequently, we have
    \[
        \lim_{n\to\infty}\exp\big(\eta n(c_2-\|\Phi_n(\theta)\|)\big) = \infty,~~\text{a.s.},
    \]
    for every fixed $\theta\in B_{c_2}$. By Fatou's lemma and 
    condition (c), 
    \[
        \liminf_{n\to\infty} \widetilde{D}_n^\eta \geq \int_{\theta\in B_{c_2}} \left(\lim_{n\to\infty}\exp\big(\eta n(c_2-\|\Phi_n(\theta)\|)\big)\right) \pi(\theta) \diff \theta = \infty,~~\text{a.s.},
    \]
    which completes the proof.
\end{proof}

\subsection{Proof of Proposition \ref{prop:est-eq-mu}}

\begin{proof}
[Proof of Proposition \ref{prop:est-eq-mu}]
    Define an auxiliary function by $\Psi(\mu)\Let \E[\phi(\mu,\xi)]$,
    where $\phi(\mu,\xi)\Let \|\xi-\mu\|-\|\xi-\mu_0\|$. Then $\Psi(\mu)$ is finite for any $\mu\in\R^p$, since 
    $
        \left|\|\xi-\mu\|-\|\xi-\mu_0\|\right| \leq \|\mu-\mu_0\|
    $
    by the reverse triangle inequality. Moreover, note that for $\mu\neq \xi$, $\nabla_\mu \phi(\mu,\xi)=(\mu-\xi)/\|\xi-\mu\|$ and $\|\nabla_\mu \phi(\mu,\xi)\|=1$. By the dominated convergence theorem, 
    \[
        \nabla \Psi(\mu) = \E\left[\frac{\mu-\xi}{\|\xi-\mu\|}\right] = -\E[\varphi_1(\mu,\xi)].
    \]
    Since
    the distribution of $\xi$
    is not supported on a single affine line, then 
    $\Psi(\mu)$ is a strictly convex function (we will show this shortly),
    thus it has 
    at most one global minimizer.

\underline{Proof of (i)}. It suffices to show that $\mu=\mu_0$ is a global minimizer of $\Psi(\mu)$. To see this, we first note that for any $\delta\in\R^p$, 
    \begin{align*}
    \widetilde{\Psi}(\delta)
    \;\Let\;\Psi(\mu_0+\delta)
    &= \E[\|\xi-\mu_0-\delta\|-\|\xi-\mu_0\|]\\
    &= \E[\|r\Sigma_0^{1/2}u-\delta\|
           -\|r\Sigma_0^{1/2}u\|]\\
    &= \E[\|-r\Sigma_0^{1/2}u-\delta\|
           -\|-r\Sigma_0^{1/2}u\|]\\
    &= \E[\|r\Sigma_0^{1/2}u+\delta\|
           -\|r\Sigma_0^{1/2}u\|]\\
    &= \Psi(\mu_0-\delta)\\
    &= \widetilde{\Psi}(-\delta),
\end{align*}
    where 
   the expectation is taken w.r.t.~the joint distribution of $r$ and $u$, and 
    the third line is by the fact that $r\Sigma_0^{1/2}u\diseq -r\Sigma_0^{1/2}u$ since $u$ is uniformly distributed over the sphere. Moreover, since $\widetilde{\Psi}(\delta)$ is a convex function by the convexity of $\Psi$, it holds that for any $\delta\in\R^p$, 
    \[
        \widetilde{\Psi}(0) \;=\; \widetilde{\Psi}\left(\frac{1}{2}\delta + \frac{1}{2}(-\delta)\right) \;\leq\; \frac{1}{2}\widetilde{\Psi}(\delta) + \frac{1}{2}\widetilde{\Psi}(-\delta) = \widetilde{\Psi}(\delta), 
    \]
    which means that $\delta=0$ is the global minimizer of $\widetilde{\Psi}(\delta)$ and hence, $\mu=\mu_0$ is the global minimizer of $\Psi(\mu)$. 

    \underline{Proof of (ii)}. We first note that $\E[\varphi_1(\mu,\xi)]$ is continuous in $\mu$. To see this, let $\mu'$ be fixed and $\{x_\ell\}_{\ell=1}^{\infty}$ be a sequence that converges to $\mu'$. Then $\varphi_1(x_\ell,\xi)\to \varphi_1(\mu',\xi)$ for almost every $\xi$, since $\varphi_1(x,\xi)$ is continuous if $x\neq \xi$. Moreover, note that $\|\varphi_1(x,\xi)\|$ is bounded by $1$. Then by the dominated convergence theorem, it holds that $\E[\varphi_1(x_\ell,\xi)]\to \E[\varphi_1(\mu',\xi)]$ for any fixed $\mu'$ and sequence $x_\ell\to \mu'$, i.e., $\E[\varphi_1(\mu,\xi)]$ is continuous in $\mu$. 
    Then for any $M>\delta>0$, by the continuity of $\E[\varphi_1(\mu,\xi)]$ and the fact that $\mu=\mu_0$ is the unique root, it holds that $\varepsilon_1\Let \inf_{\delta\leq \|\mu-\mu_0\|\leq M} \|\E[\varphi_1(\mu,\xi)]\|>0$. Then to show (ii), it remains to show that there exists $M>0$ such that if $\|\mu-\mu_0\|\geq M$, then $\|\E[\varphi_1(\mu,\xi)]\|\geq 1/2$, since one can choose $\varepsilon = \min\{1/2, \varepsilon_1\} > 0$ and it holds that $\inf_{\|\mu-\mu_0\|\geq \delta}\|\E[\varphi_1(\mu,\xi)]\| \geq \varepsilon$, which is (ii). Indeed, Lemma~\ref{lemma:E-phi-1-inf} readily implies this argument, since one can choose $M$ sufficiently large so that $\|\mu\|\geq \|\mu-\mu_0\|-\|\mu_0\|\geq M-\|\mu_0\|$ is large enough and $\|\E[\varphi_1(\mu,\xi)]\|$ is close to $1$. This completes the proof of (ii).

    Now it remains to show that $\Psi(\mu)=\E[\phi(\mu,\xi)]$ is strictly convex for $p\geq 2$. We note that $\phi(\mu,\xi)$ is convex in $\mu$ for any fixed $\xi$, i.e., for any $\mu_1,\mu_2\in\R^p$, $\mu_1\neq \mu_2$ and $\lambda\in (0,1)$,
    \begin{align*}
        \phi(\lambda \mu_1 + (1-\lambda) \mu_2,\xi)
        \;=\; &\|\xi - (\lambda \mu_1 + (1-\lambda) \mu_2)\| - \|\xi-\mu_0\|\\
        \;\leq\; &\lambda\|\xi-\mu_1\| + (1-\lambda)\|\xi-\mu_2\|- \|\xi-\mu_0\|\\
        \;=\; &\lambda \phi(\mu_1,\xi) + (1-\lambda) \phi(\mu_2,\xi).
    \end{align*}
    Equality in the above triangle inequality can hold only if
$\xi$ belongs to the affine line
\[
    L\Let\{\mu_1+t(\mu_2-\mu_1):t\in\R\}.
\]
Since $p\geq 2$ and $\xi$ admits a density with respect to
Lebesgue measure on $\R^p$, we have $\mathbb{P}(\xi\in L)=0$.
Consequently, for every fixed pair
$\mu_1,\mu_2\in\R^p$ with $\mu_1\neq\mu_2$
and every $\lambda\in(0,1)$,
\begin{align}\label{ineq:str-cvx-as}
    \phi(\lambda\mu_1+(1-\lambda)\mu_2,\xi)
    <
    \lambda\phi(\mu_1,\xi)
    +(1-\lambda)\phi(\mu_2,\xi)
    \quad\text{almost surely}.
\end{align}
    Taking expectation on both sides of~\eqref{ineq:str-cvx-as}, we derive that for any $\mu_1,\mu_2\in\R^p$,$\mu_1\neq \mu_2$ and $\lambda\in(0,1)$,
    \[
        \Psi(\lambda\mu_1+(1-\lambda)\mu_2) < \lambda \Psi(\mu_1)+(1-\lambda)\Psi(\mu_2),
    \]
    i.e., $\Psi(\mu)$ is strictly convex.
\end{proof}

\subsection{Proof of Proposition \ref{prop:est-eq-sigma}}

\begin{proof}
[Proof of Proposition \ref{prop:est-eq-sigma}]
We first check that $\E[\varphi_2(\mu_0,c\Sigma_0,\xi)]=0$ for any $c\in\R_{++}$, and then show that $\E[\varphi_2(\mu_0,\Sigma,\xi)]=0$ implies that $\Sigma=c\Sigma_0$ for some constant $c\in\R_{++}$.
By the 
representation of $\xi$ in \eqref{eq:EC-stochastic-rep}, 
\begin{align*}
    \varphi_2(\mu_0,c\Sigma_0,\xi) \;=\; &c\Sigma_0/p - c\cdot\frac{(\xi-\mu_0)(\xi-\mu_0)^\transpose}{(\xi-\mu_0)^\transpose \Sigma_0^{-1}(\xi-\mu_0)}\\
    \;=\; &c\Sigma_0/p  - c\cdot\frac{r^2\Sigma_0^{1/2}u u^\transpose \Sigma_0^{1/2}}{r^2 u^\transpose \Sigma_0^{1/2} \Sigma_0^{-1}\Sigma_0^{1/2}u}\\
    \;=\;& c\Sigma_0/p - c\cdot\frac{\Sigma_0^{1/2}u u^\transpose \Sigma_0^{1/2}}{u^\transpose u}\\
    \;=\; &c\Sigma_0/p - c\cdot\Sigma_0^{1/2}u u^\transpose \Sigma_0^{1/2}.
\end{align*}
Since $\E[uu^\transpose] = I/p$ (see Proposition \ref{prop:E-uu}),
by taking expectation on both sides of the equation above, we obtain 
$\E[\varphi_2(\mu_0,c\Sigma_0,\xi)]=0$.


Next, we show the ``only if'' part. 
Observe that
\begin{align}
    \E[\varphi_2(\mu_0,\Sigma,\xi)]=0 \;\Longleftrightarrow\; &\Sigma/p = \E\left[\frac{\Sigma_0^{1/2}uu^\transpose\Sigma_0^{1/2}}{u^\transpose\Sigma_0^{1/2}\Sigma^{-1}\Sigma_0^{1/2}u}\right]\nonumber\\
    \;\Longrightarrow\; &(\Sigma_0^{1/2}\Sigma^{-1}\Sigma_0^{1/2})^{-1}/p = \E\left[\frac{uu^\transpose}{u^\transpose\Sigma_0^{1/2}\Sigma^{-1}\Sigma_0^{1/2}u}\right]\nonumber\\
    \text{(Let $X\Let\Sigma_0^{1/2}\Sigma^{-1}\Sigma_0^{1/2}$)}\;\Longrightarrow\;&X^{-1}/p=\E\left[\frac{uu^\transpose}{u^\transpose X u}\right].\label{eq:X-inv/p}
\end{align}
It remains to show that~\eqref{eq:X-inv/p} implies that $X=cI$ for some constant $c\in\R_{++}$. 
Let $X=Q\Lambda Q^\transpose, \Lambda\Let \diag(\lambda_1,\ldots,\lambda_p), \lambda_i>0$ be the spectral decomposition of $X$. 
Then it suffices to show that $\Lambda=cI$ for some constant $c\in\R_{++}$. 
Since $u\diseq Qu$ for the orthogonal matrix $Q$,
\begin{align}
    \eqref{eq:X-inv/p}\;\Longrightarrow\; &Q\Lambda^{-1}Q^\transpose/p = Q\E\left[\frac{uu^\transpose}{u^\transpose\Lambda u}\right]Q^\transpose\nonumber\\
    \;\Longrightarrow\; &\Lambda^{-1}/p = \E\left[\frac{uu^\transpose}{u^\transpose\Lambda u}\right]\nonumber\\
    \;\Longrightarrow\; & 1/p = \E\left[\frac{\lambda_i u_i^2}{\sum_{k=1}^p \lambda_ku_k^2}\right]~~\text{for}\; i=1,\ldots,p.\label{eq:1/p-lambda}
\end{align}
In what follows, we show that
$\lambda_1=\ldots=\lambda_p=c$.
Assume for the sake of 
a contradiction that there exists an index $i^*$ such that $\lambda_{i^*}=\max_{1\leq i\leq p}\lambda_i$ and $\lambda_{i^*}>\lambda_j$ for some $j\neq i^*$. 
Since $u\in\mathbb{U}^{p-1}$, then 
\begin{align*}
    \sum_{k=1}^p \lambda_k u_k^2 < \lambda_{i^*} \sum_{k=1}^p u_k^2 = \lambda_{i^*}~~\text{almost surely},
\end{align*}
which implies
\begin{align}\label{ineq:g-1/p}
    \frac{\lambda_{i^*}u_{i^*}^2}{\sum_{k=1}^p \lambda_k u_k^2} > u_{i^*}^2~~\text{almost surely} \;\Longrightarrow\; \E\left[\frac{\lambda_{i^*}u_{i^*}^2}{\sum_{k=1}^p \lambda_k u_k^2}\right] > \E[u_{i^*}^2] = \frac{1}{p},
\end{align}
where the last equality is by the fact that $\E[uu^\transpose]=I/p$. However, \eqref{ineq:g-1/p} contradicts \eqref{eq:1/p-lambda}, which means that $\Lambda=cI$ for some $c\in\R_{++}$ and completes the proof of this part.
\end{proof}

\subsection{Proof of Proposition \ref{prop:est-eq-mu-Sigma}}

\begin{proof}
[Proof of Proposition \ref{prop:est-eq-mu-Sigma}]
    The conclusion 
    follows directly from Propositions~\ref{prop:est-eq-mu} and~\ref{prop:est-eq-sigma}.  We 
    include a brief proof for completeness. First, $(\mu_0,\Sigma_0)$ is 
    solution to $\E[\varphi(\mu,\Sigma,\xi)]=0$ in that Proposition~\ref{prop:est-eq-mu} ensures $\E[\varphi_1(\mu_0,\xi)]=0$, and Proposition~\ref{prop:est-eq-sigma}  ensures $\E[\varphi_2(\mu_0,\Sigma_0,\xi)]=0$.
    
    To prove the uniqueness, let 
    $(\mu',\Sigma')\in\Theta$ satisfy $\E[\varphi(\mu',\Sigma',\xi)]=0$. By Propositions~\ref{prop:est-eq-mu}~and~\ref{prop:est-eq-sigma}, we have $\mu'=\mu_0$ and $\Sigma'=c\Sigma_0$ for some $c\in\R_{++}$. Since $\tr(\Sigma')=\tr(\Sigma_0)=p$, we must have $c=1$, and hence $\Sigma'=\Sigma_0$. Therefore, the solution is unique.
\end{proof}

\subsection{Proof of Theorem \ref{thm:gibbs-mu-sigma-consistency}}

\begin{proof}
   We use Theorem~\ref{thm:consist-gibbs} to prove the results in this theorem.  
   Thus, it suffices to 
verify   
the conditions (a) and (b) in Theorem~\ref{thm:consist-gibbs}, i.e.,
(a) $\E[\|\varphi(\mu,\Sigma,\xi)\|]<\infty$ for all 
$(\mu,\Sigma)\in \Theta$, and (b) for any $\delta >0$,
    \[
        \liminf_{n\to\infty} \inf_{(\mu,\Sigma)\in\Theta:\|(\mu,\Sigma)-(\mu_0,\Sigma_0)\|\geq \delta} \|\Phi_n(\mu,\Sigma)\| > 0~~\text{almost surely,}
    \]
    where 
    \bgeqn 
    \label{eq:Phi_n-EDRO}
    \Phi_n(\mu,\Sigma)\Let \frac{1}{n}\sum_{i=1}^n \varphi(\mu,\Sigma,\xi_i). \edeqn 

\noindent
 \underline{Condition (a)}. 
 By definition, $\|\varphi(\mu,\Sigma,\xi)\|=\|\varphi_1(\mu,\xi)\|_2+\|\varphi_2(\mu,\Sigma,\xi)\|_{\text{F}}$. Since $\|\varphi_1(\mu,\xi)\|_2=1$ for all $\xi\in \R^p$, it suffices to show that $\E[\|\varphi_2(\mu,\Sigma,\xi)\|_{\text{F}}]< \infty$. Observe that
    \begin{align*}
        \|\varphi_2(\mu,\Sigma,\xi)\|_{\text{F}} \;=\; &\left\|\Sigma/p - \frac{(\xi-\mu)(\xi-\mu)^\transpose}{(\xi-\mu)^\transpose\Sigma^{-1}(\xi-\mu)}\right\|_{\text{F}}
        \;\leq\; \|\Sigma\|_{\text{F}}/p + \left\|\frac{(\xi-\mu)(\xi-\mu)^\transpose}{(\xi-\mu)^\transpose\Sigma^{-1}(\xi-\mu)}\right\|_{\text{F}}\\
        \;=\; &\|\Sigma\|_{\text{F}}/p + \frac{\left\|(\xi-\mu)(\xi-\mu)^\transpose\right\|_{\text{F}}}{(\xi-\mu)^\transpose\Sigma^{-1}(\xi-\mu)}
        \;=\; \|\Sigma\|_{\text{F}}/p + \frac{\left\|(\xi-\mu)\right\|_2^2}{(\xi-\mu)^\transpose\Sigma^{-1}(\xi-\mu)}.
    \end{align*}
    Since $\Sigma\in \PD$, then
    $
        (\xi-\mu)^\transpose\Sigma^{-1}(\xi-\mu) \geq \|\xi-\mu\|_2^2/\lambda_{\max}(\Sigma).
    $
    Consequently
    $
        \|\varphi_2(\mu,\Sigma,\xi)\|_{\text{F}} \leq \|\Sigma\|_{\text{F}}/p + \lambda_{\max}(\Sigma)<\infty
    $
    uniformly  for all 
$(\mu,\Sigma)\in \Theta$ and hence $\E[\|\varphi_2(\mu,\Sigma,\xi)\|_{\text{F}}] < \infty$.
    

\vspace{0.2cm}
   \noindent
 \underline{Condition (b)}. Let $M > 0$ be a sufficiently large constant, 
 and $c>0$ be a sufficiently small constant, both to be specified shortly.
It suffices to show that for any fixed $0<\delta<1$, the following three inequalities hold almost surely:
  \begin{subequations}
    \bgeqn 
        && \liminf_{n\to\infty} \inf_{\substack{\|(\mu,\Sigma)-(\mu_0,\Sigma_0)\|\geq \delta,\\ \|\mu\|\leq M, \lambda_{\min}(\Sigma)\geq c,\\(\mu,\Sigma)\in\Theta}} \|\Phi_n(\mu,\Sigma)\| > 0, 
        \label{ineq:b1}\\
&& \liminf_{n\to\infty} \inf_{\substack{\|(\mu,\Sigma)-(\mu_0,\Sigma_0)\|\geq \delta,\\ \|\mu\|\geq M,(\mu,\Sigma)\in\Theta}} \|\Phi_n(\mu,\Sigma)\| > 0,
         \label{ineq:b2}\\
&&        
\liminf_{n\to\infty} \inf_{\substack{\|(\mu,\Sigma)-(\mu_0,\Sigma_0)\|\geq \delta,\\ \lambda_{\min}(\Sigma)\leq c,(\mu,\Sigma)\in\Theta}} \|\Phi_n(\mu,\Sigma)\| > 0.
\label{ineq:b3}
\edeqn 
  \end{subequations}
  This is because 
the union of the feasible sets in \eqref{ineq:b1}-\eqref{ineq:b3} is $\{(\mu,\Sigma)\in \Theta:\|(\mu,\Sigma)-(\mu_0,\Sigma_0)\|\geq \delta\}$ and 
\eqref{ineq:b1}--\eqref{ineq:b3} implies condition (b). 
    
\underline{Proof of \eqref{ineq:b1}}. 
To ease the notation, let
    \bgeq \Theta_{\delta,M,c}\Let \{(\mu,\Sigma)\in\Theta:\|(\mu,\Sigma)-(\mu_0,\Sigma_0)\|\geq \delta, \|\mu\|\leq M, \lambda_{\min}(\Sigma)\geq c\}.
    \edeq 
    Note that $\Theta_{\delta,M,c}$ is a compact set. 
    Since $\tr(\Sigma)=p$ and all eigenvalues of $\Sigma$ are nonnegative,
    then 
    $\lambda_{\max}(\Sigma)\leq p$.
    By Lemma \ref{lemma:inf-Phi-ge-0}, 
    \[
        \inf_{(\mu,\Sigma)\in \Theta_{\delta,M,c}} \|\Phi(\mu,\Sigma)\| > 0.
    \]
    Moreover, by Lemma \ref{lemma:ULLN}, we have 
    \[
        \sup_{(\mu,\Sigma)\in \Theta_{\delta,M,c}} \left\|\Phi_n(\mu,\Sigma)-\Phi(\mu,\Sigma)\right\| \asto 0.
    \]
    By the reverse triangle inequality, 
    \begin{align*}
        \inf_{(\mu,\Sigma)\in \Theta_{\delta,M,c}} \|\Phi_n(\mu,\Sigma)\| \;\geq\; &\inf_{(\mu,\Sigma)\in \Theta_{\delta,M,c}} \big\{\|\Phi(\mu,\Sigma)\| - \|\Phi_n(\mu,\Sigma)-\Phi(\mu,\Sigma)\|\big\}\\
        \geq\; &\inf_{(\mu,\Sigma)\in \Theta_{\delta,M,c}} \|\Phi(\mu,\Sigma)\| - \sup_{(\mu,\Sigma)\in \Theta_{\delta,M,c}} \left\|\Phi_n(\mu,\Sigma)-\Phi(\mu,\Sigma)\right\|\\
        >\;&0,
    \end{align*}
    almost surely as $n\to \infty$, which completes the proof of \eqref{ineq:b1}.

   \underline{  Proof of \eqref{ineq:b2}}. Let $M$ be chosen sufficiently large such that $M> \|\mu_0\| + 1$. Then 
    \[
        \|\mu\| \geq M \;\Rightarrow\; \|\mu\| \geq \|\mu_0\| + 1 \;\Rightarrow\; \|\mu-\mu_0\| 
        \geq 
        \|\mu\| - \|\mu_0\| > 1 > \delta,
    \]
which means constraint    $\|(\mu,\Sigma)-(\mu_0,\Sigma_0)\|\geq \delta$
is subsumed by 
constraint $\|\mu\|\geq M$. 
    Consequently, 
    \[
        \inf_{\substack{\|(\mu,\Sigma)-(\mu_0,\Sigma_0)\|\geq \delta,\\ \|\mu\|\geq M,(\mu,\Sigma)\in\Theta}} \|\Phi_n(\mu,\Sigma)\|
        = \inf_{\substack{\|\mu\|\geq M, (\mu,\Sigma)\in\Theta}} \|\Phi_n(\mu,\Sigma)\|.
    \]
    Moreover, since
    \begin{align*}
        \inf_{\substack{\|\mu\|\geq M,(\mu,\Sigma)\in\Theta}} \|\Phi_n(\mu,\Sigma)\| = &\inf_{\substack{\|\mu\|\geq M, (\mu,\Sigma)\in\Theta}} \left\|\frac{1}{n}\sum_{i=1}^n \varphi_1(\mu,\xi_i)\right\| + \left\|\frac{1}{n}\sum_{i=1}^n \varphi_2(\mu,\Sigma,\xi_i)\right\|_{\F} \\
        \geq\; &\inf_{\mu\in\R^p, \|\mu\|\geq M} \left\|\frac{1}{n}\sum_{i=1}^n \varphi_1(\mu,\xi_i)\right\|,
    \end{align*}
    then it is enough to show that
    \[
        \liminf_{n\to\infty} \inf_{\mu\in\R^p, \|\mu\|\geq M} \left\|\frac{1}{n}\sum_{i=1}^n \varphi_1(\mu,\xi_i)\right\| > 0,~\text{almost surely.}
    \]
Since $r$ is a single-valued random variable, its distribution is tight and hence the distribution of $\xi$ is also tight.
     Let $R>0$ be a sufficiently large constant such that 
    $
        P(\|\xi\| > R) < 1/8.
$
    We note that
    \begin{align}
        \left\|\frac{1}{n}\sum_{i=1}^n \varphi_1(\mu,\xi_i)\right\| \;=\; & \left\|\frac{1}{n}\sum_{i=1}^n \left(\varphi_1(\mu,\xi_i)\mathds{1}_{\{\|\xi_i\| \le R\}} + \varphi_1(\mu,\xi_i)\mathds{1}_{\{\|\xi_i\| > R\}}\right)\right\| \label{ineq:b2-1}
    \end{align}
    It remains to find a lower bound for the right hand side of \eqref{ineq:b2-1} as $n\to\infty$. Let $u\Let \mu/\|\mu\|$. By Cauchy–Schwarz inequality, for any $x\in \R^p$, 
 $   
        -u^\top x \leq |u^\top x| \leq \|u\|\|x\| = \|x\|. 
    $
    Then
    \begin{align}\label{ineq:b2-2}
        \text{RHS of}~\eqref{ineq:b2-1}\geq \frac{1}{n}\sum_{i=1}^n \left(-u^\top\frac{\xi_i-\mu}{\|\xi_i-\mu\|}\right)\mathds{1}_{\{\|\xi_i\| \le R\}} + \frac{1}{n}\sum_{i=1}^n \left(-u^\top\frac{\xi_i-\mu}{\|\xi_i-\mu\|}\right)\mathds{1}_{\{\|\xi_i\| > R\}} .
    \end{align}
  If $\|\xi_i\| \le R$, then 
    \[
        -u^\top\frac{\xi_i-\mu}{\|\xi_i-\mu\|} = \frac{-u^\top\xi_i+\|\mu\|}{\|\xi_i-\mu\|} \geq \frac{\|\mu\|-R}{\|\mu\|+R},
    \]
    where the inequality is by $u^\top\xi_i \leq \|\xi_i\| \leq R$ and $\|\xi_i-\mu\| \leq \|\mu\|+\|\xi_i\|\leq \|\mu\|+R$. Now let $M$ be sufficiently large such that $M>2R$, then for $\|\mu\| \geq M > 2R$, 
    \begin{align*}
         \frac{1}{n}\sum_{i=1}^n \left(-u^\top\frac{\xi_i-\mu}{\|\xi_i-\mu\|}\right)\mathds{1}_{\{\|\xi_i\| \le R\}} \geq \frac{1}{3}\cdot\frac{1}{n}\sum_{i=1}^n \mathds{1}_{\{\|\xi_i\| \le R\}}.
    \end{align*}
    For the second term in the right hand side of \eqref{ineq:b2-2}, we use the trivial bound
    \[
     \frac{1}{n}\sum_{i=1}^n \left(-u^\top\frac{\xi_i-\mu}{\|\xi_i-\mu\|}\right)\mathds{1}_{\{\|\xi_i\| > R\}} \geq -\frac{1}{n}\sum_{i=1}^n \mathds{1}_{\{\|\xi_i\| > R\}},
    \]
    where the inequality is by the fact that, for any unit vectors $x$ and $y$, $x^\top y \leq 1$.
    Combining the above two lower bounds, we note that for any $\mu$ satisfying $\|\mu\| \geq M > 2R$,
    \begin{align*}
        \text{RHS of \eqref{ineq:b2-1}} \;\geq\; \frac{1}{3}\cdot\frac{1}{n}\sum_{i=1}^n \mathds{1}_{\{\|\xi_i\| \le R\}} -\frac{1}{n}\sum_{i=1}^n \mathds{1}_{\{\|\xi_i\| > R\}}
        \;=\;  \frac{1}{3}-\frac{4}{3}\cdot \frac{1}{n}\sum_{i=1}^n \mathds{1}_{\{\|\xi_i\| > R\}}.
    \end{align*}
    In particular, by the choice of $R$ and strong law of large numbers, it holds that $\lim_{n\to\infty}\frac{1}{n}\sum_{i=1}^n \mathds{1}_{\{\|\xi_i\| > R\}} < 1/8 $, and thus, by setting $M > \max\{\|\mu_0\|+1, 2R\}$, we derive that
    \begin{align}
        \liminf_{n\to\infty} \inf_{\mu\in\R^p, \|\mu\|\geq M} \left\|\frac{1}{n}\sum_{i=1}^n \varphi_1(\mu,\xi_i)\right\| \geq \liminf_{n\to\infty} \left\{\frac{1}{3}-\frac{4}{3}\cdot \frac{1}{n}\sum_{i=1}^n \mathds{1}_{\{\|\xi_i\| > R\}}\right\} > 0, \label{ineq:mu>M}
    \end{align}
    almost surely.

    \underline{ Proof of \eqref{ineq:b3}}. 
    Recall that by the definition, 
    \begin{align}
        \label{eq:b3-Phi}
        \|\Phi_n(\mu,\Sigma)\| = \left\|\frac{1}{n}\sum_{i=1}^n \varphi_1(\mu,\xi_i)\right\| + \left\|\frac{1}{n}\sum_{i=1}^n \varphi_2(\mu,\Sigma,\xi_i)\right\|_{\F}.
    \end{align}
    In \eqref{ineq:mu>M}, we have shown that the first term of \eqref{eq:b3-Phi} is strictly positive as $\|\mu\|\geq M$. 
    Then to show \eqref{ineq:b3}, it suffices to show that, for the second term of \eqref{eq:b3-Phi}, there exists $c>0$ such that
    \begin{align}\label{ineq:b3-1}
        \liminf_{n\to\infty} \inf_{\substack{\|\mu\|\leq M, \lambda_{\min}(\Sigma)\leq c,\\\tr(\Sigma)=p,\Sigma\in \PD}}\left\|\frac{1}{n}\sum_{i=1}^n \varphi_2(\mu,\Sigma,\xi_i)\right\|_{\F} > 0
    \end{align}
    almost surely. We note that 
    \begin{align*}
        \left\|\frac{1}{n}\sum_{i=1}^n \varphi_2(\mu,\Sigma,\xi_i)\right\|_{\F} \;=\; &\left\|\frac{\Sigma}{p}-\frac{1}{n}\sum_{i=1}^n\frac{(\xi_i-\mu)(\xi_i-\mu)^\top}{(\xi_i-\mu)^\top\Sigma^{-1}(\xi_i-\mu)}\right\|_{\F}\\
        \geq\; &\left\|\Sigma/p\right\|_{\F}-\left\|\frac{1}{n}\sum_{i=1}^n \frac{(\xi_i-\mu)(\xi_i-\mu)^\top}{(\xi_i-\mu)^\top\Sigma_{}^{-1}(\xi_i-\mu)}\right\|_{\F}\\
        \geq\;&\frac{1}{\sqrt{p}}-\left\|\frac{1}{n}\sum_{i=1}^n \frac{(\xi_i-\mu)(\xi_i-\mu)^\top}{(\xi_i-\mu)^\top\Sigma_{}^{-1}(\xi_i-\mu)}\right\|_{\F},
    \end{align*}
    where the second inequality holds for $\tr(\Sigma) = p$. In the following, to show \eqref{ineq:b3-1}, we are to show that there exists a constant $c>0$ such that 
    \begin{align}\label{ineq:b3-2}
        \limsup_{n\to\infty} \sup_{\substack{\|\mu\|\leq M, \lambda_{\min}(\Sigma)\leq c,\\\tr(\Sigma)=p,\Sigma\in \PD}}\left\|\frac{1}{n}\sum_{i=1}^n \frac{(\xi_i-\mu)(\xi_i-\mu)^\top}{(\xi_i-\mu)^\top\Sigma_{}^{-1}(\xi_i-\mu)}\right\|_{\F} \leq \frac{1}{2p}~\text{almost surely},
    \end{align}
    since $1/\sqrt{p}-1/(2p)>0$.

To show \eqref{ineq:b3-2}, let $c>0$ be a constant to be specified shortly. We first introduce an auxiliary matrix by 
\[
    \Sigma^{(c)}\Let\frac{\Sigma+cI_p}{1+c}
\]
for each $\Sigma\in\PD$ with $\tr(\Sigma)=p$,
and show that for every $\mu\in\R^p$ and every
$\Sigma\in\PD$ with $\tr(\Sigma)=p$,
\begin{align}\label{ineq:b3-regularized-comparison}
    \left\|
        \frac1n\sum_{i=1}^n
        \frac{(\xi_i-\mu)(\xi_i-\mu)^\top}
             {(\xi_i-\mu)^\top\Sigma^{-1}(\xi_i-\mu)}
    \right\|_{\F}
    \leq(1+c)
    \left\|
        \frac1n\sum_{i=1}^n
        \frac{(\xi_i-\mu)(\xi_i-\mu)^\top}
             {(\xi_i-\mu)^\top(\Sigma^{(c)})^{-1}(\xi_i-\mu)}
    \right\|_{\F}.
\end{align}
To see this, note that
$(\Sigma+cI_p)^{-1}\preceq\Sigma^{-1}$ and
$(\Sigma^{(c)})^{-1}=(1+c)(\Sigma+cI_p)^{-1}$.
Therefore, for every $x\in\R^p$,
\begin{align*}
    0\preceq
    \frac{(x-\mu)(x-\mu)^\top}
         {(x-\mu)^\top\Sigma^{-1}(x-\mu)}
    \preceq
    \frac{(x-\mu)(x-\mu)^\top}
         {(x-\mu)^\top(\Sigma+cI_p)^{-1}(x-\mu)}
    =(1+c)
    \frac{(x-\mu)(x-\mu)^\top}
         {(x-\mu)^\top(\Sigma^{(c)})^{-1}(x-\mu)}.
\end{align*}
Then we obtain
\begin{align*}
    0\preceq
    \frac1n\sum_{i=1}^n
    \frac{(\xi_i-\mu)(\xi_i-\mu)^\top}
         {(\xi_i-\mu)^\top\Sigma^{-1}(\xi_i-\mu)}
    \preceq
    (1+c)\frac1n\sum_{i=1}^n
    \frac{(\xi_i-\mu)(\xi_i-\mu)^\top}
         {(\xi_i-\mu)^\top(\Sigma^{(c)})^{-1}(\xi_i-\mu)},
\end{align*}
and hence \eqref{ineq:b3-regularized-comparison} holds. 

Then by \eqref{ineq:b3-regularized-comparison} and the triangle inequality, it holds that for every $\|\mu\|\leq M$ and every $\Sigma\in\PD$
with $\tr(\Sigma)=p$ and $\lambda_{\min}(\Sigma)\leq c$,
\begin{align}
&\left\|
        \frac1n\sum_{i=1}^n
        \frac{(\xi_i-\mu)(\xi_i-\mu)^\top}
             {(\xi_i-\mu)^\top\Sigma^{-1}(\xi_i-\mu)}
    \right\|_{\F} \leq (1+c)
    \left\|\frac1n\sum_{i=1}^n
        \frac{(\xi_i-\mu)(\xi_i-\mu)^\top}
             {(\xi_i-\mu)^\top(\Sigma^{(c)})^{-1}(\xi_i-\mu)}
    \right\|_{\F}\nonumber\\
    \leq \;&(1+c)\left\{\left\|
        \frac1n\sum_{i=1}^n
        \frac{(\xi_i-\mu)(\xi_i-\mu)^\top}
             {(\xi_i-\mu)^\top(\Sigma^{(c)})^{-1}(\xi_i-\mu)}
        -\E\left[
            \frac{(\xi-\mu)(\xi-\mu)^\top}
                 {(\xi-\mu)^\top(\Sigma^{(c)})^{-1}(\xi-\mu)}
        \right]
    \right\|_{\F}\right.\nonumber\\&\quad+\left.
    \left\|
        \E\left[
            \frac{(\xi-\mu)(\xi-\mu)^\top}
                 {(\xi-\mu)^\top(\Sigma^{(c)})^{-1}(\xi-\mu)}
        \right]
    \right\|_{\F}\right\}. \label{ineq:b3-3}
\end{align}
Hence, to show \eqref{ineq:b3-2}, it suffices to bound the right hand side of \eqref{ineq:b3-3} term by term. 

Now we deal with the second term in \eqref{ineq:b3-3}. To do this, we first claim that
\begin{align}\label{eq:b3-population-zero}
    \lim_{t\downarrow0}
    \sup_{\substack{\|\mu\|\leq M,\ \lambda_{\min}(\Sigma)\leq t,\\
                    \tr(\Sigma)=p,\ \Sigma\in\PD}}
    \left\|
        \E\left[
            \frac{(\xi-\mu)(\xi-\mu)^\top}
                 {(\xi-\mu)^\top\Sigma^{-1}(\xi-\mu)}
        \right]
    \right\|_{\F}=0.
\end{align}
To ease notation, let
\[
    S(t) \Let \sup_{\substack{\|\mu\|\leq M,\ \lambda_{\min}(\Sigma)\leq t,\\
                    \tr(\Sigma)=p,\ \Sigma\in\PD}}
    \left\|
        \E\left[
            \frac{(\xi-\mu)(\xi-\mu)^\top}
                 {(\xi-\mu)^\top\Sigma^{-1}(\xi-\mu)}
        \right]
    \right\|_{\F}.
\]
Then to show \eqref{eq:b3-population-zero}, for each integer $\ell\geq1$, choose a
pair $(\mu_\ell,\Sigma_\ell)$ satisfying
$\|\mu_\ell\|\leq M$, $\tr(\Sigma_\ell)=p$,
$\Sigma_\ell\in\PD$, and $\lambda_{\min}(\Sigma_\ell)\leq1/\ell$,
such that
\begin{align}\label{ineq:b3-4}
     S(1/\ell) \leq \left\|
        \E\left[
            \frac{(\xi-\mu_\ell)(\xi-\mu_\ell)^\top}
                 {(\xi-\mu_\ell)^\top\Sigma_\ell^{-1}(\xi-\mu_\ell)}
        \right]
    \right\|_{\F} + \frac{1}{\ell}.
\end{align}
Since $\{(\mu,\Sigma):\|\mu\|\leq M,\
       \tr(\Sigma)=p,\ \Sigma\in\PSD\}$
is compact, there exists a subsequence of $(\mu_\ell,\Sigma_\ell)$, indexed by $\ell_k$, such that
\[
    \mu_{\ell_k}\to\mu',\qquad
    \Sigma_{\ell_k}\to\Sigma'\in\PSD,\qquad
    \tr(\Sigma')=p,\qquad
    \lambda_{\min}(\Sigma')=0.
\]
Let $v$ be a unit eigenvector of $\Sigma'$ corresponding to
the eigenvalue zero. For any fixed $x\in\R^p$ satisfying
$v^\top(x-\mu')\neq0$, the Cauchy--Schwarz inequality implies that
\[
    (x-\mu_{\ell_k})^\top\Sigma_{\ell_k}^{-1}(x-\mu_{\ell_k})
    \geq
    \frac{[v^\top(x-\mu_{\ell_k})]^2}
         {v^\top\Sigma_{\ell_k}v}
    \longrightarrow+\infty,
\]
where the convergence is by $[v^\top(x-\mu_{\ell_k})]^2\to [v^\top(x-\mu')]^2>0$ and $v^\top\Sigma_{\ell_k}v\downarrow v^\top\Sigma'v=0$.
Consequently,
\[
    \frac{(x-\mu_{\ell_k})(x-\mu_{\ell_k})^\top}
         {(x-\mu_{\ell_k})^\top\Sigma_{\ell_k}^{-1}
          (x-\mu_{\ell_k})}
    \longrightarrow0.
\]
Since by assumption, $\xi$ admits a density,
$\prob(v^\top(\xi-\mu')=0)=0$. Moreover, the upper bound established
in the proof of condition (a) gives that
\[
    \left\|
        \frac{(x-\mu_{\ell_k})(x-\mu_{\ell_k})^\top}
             {(x-\mu_{\ell_k})^\top\Sigma_{\ell_k}^{-1}
              (x-\mu_{\ell_k})}
    \right\|_{\F}
    \leq\lambda_{\max}(\Sigma_{\ell_k})\leq p.
\]
Thus, by the dominated convergence theorem,
\begin{align*}
    \left\|
        \E\left[
            \frac{(\xi-\mu_{\ell_k})(\xi-\mu_{\ell_k})^\top}
                 {(\xi-\mu_{\ell_k})^\top\Sigma_{\ell_k}^{-1}
                  (\xi-\mu_{\ell_k})}
        \right]
    \right\|_{\F}
   \leq
    \E\left[
        \left\|
            \frac{(\xi-\mu_{\ell_k})(\xi-\mu_{\ell_k})^\top}
                 {(\xi-\mu_{\ell_k})^\top\Sigma_{\ell_k}^{-1}
                  (\xi-\mu_{\ell_k})}
        \right\|_{\F}
    \right]
    \longrightarrow0.
\end{align*}
Then by \eqref{ineq:b3-4}, $S(1/\ell_k)\to 0$ as $k\to\infty$. Moreover, since $S(t)$ is nonincreasing as $t$ decreases, this proves \eqref{eq:b3-population-zero}. Then by \eqref{eq:b3-population-zero}, there exists a deterministic
constant $c\in(0,1)$ such that
\begin{align}\label{ineq:b3-population-small}
    \sup_{\substack{\|\mu\|\leq M,\ \lambda_{\min}(\Sigma)\leq2c,\\
                    \tr(\Sigma)=p,\ \Sigma\in\PD}}
    \left\|
        \E\left[
            \frac{(\xi-\mu)(\xi-\mu)^\top}
                 {(\xi-\mu)^\top\Sigma^{-1}(\xi-\mu)}
        \right]
    \right\|_{\F}
    <\frac1{4p}.
\end{align}
Fix such a constant $c$ for the rest of the proof.
For each $\Sigma\in\PD$ with $\tr(\Sigma)=p$, recall the definition of $\Sigma^{(c)}$:
\[
    \Sigma^{(c)}\Let\frac{\Sigma+cI_p}{1+c}.
\]
Then if $\lambda_{\min}(\Sigma)\leq c$, it holds that
\[
    \lambda_{\min}(\Sigma^{(c)})
    \leq\frac{2c}{1+c}\leq2c.
\]
Consequently, \eqref{ineq:b3-population-small} implies that
\begin{align}\label{ineq:b3-sigma-c-expectation}
    &\sup_{\substack{\|\mu\|\leq M,\ \lambda_{\min}(\Sigma)\leq c,\\
                     \tr(\Sigma)=p,\ \Sigma\in\PD}}
    \left\|
        \E\left[
            \frac{(\xi-\mu)(\xi-\mu)^\top}
                 {(\xi-\mu)^\top(\Sigma^{(c)})^{-1}(\xi-\mu)}
        \right]
    \right\|_{\F} \leq
    \sup_{\substack{\|\mu\|\leq M,\ \lambda_{\min}(\Sigma)\leq2c,\\
                    \tr(\Sigma)=p,\ \Sigma\in\PD}}
    \left\|
        \E\left[
            \frac{(\xi-\mu)(\xi-\mu)^\top}
                 {(\xi-\mu)^\top \Sigma^{-1}(\xi-\mu)}
        \right]
    \right\|_{\F}
    <\frac1{4p}.
\end{align}

Now we consider the first term in \eqref{ineq:b3-3}. We note that for $\Sigma\in \PD$, it holds that 
\[
    \lambda_{\min}(\Sigma^{(c)})
    =\frac{\lambda_{\min}(\Sigma)+c}{1+c}
    \geq\frac{c}{1+c}>0,
\]
or in other words, $\Sigma^{(c)} \succeq \frac{c}{1+c}I_p$.
The discussion of the first term in \eqref{ineq:b3-3} is an application of Lemma~\ref{lemma:h-ULLN} to the compact sets
\[
    X\Let \{\mu\in\R^p:\|\mu\|\leq M\}
    \quad\text{and}\quad Y\Let
    \left\{\Sigma\in\PD:
        \tr(\Sigma)=p,\ \Sigma\succeq\frac{c}{1+c}I_p\right\}.
\]
Note that for $(\mu,\Sigma)\in X\times Y$,
\begin{align*}
    \frac1n\sum_{i=1}^n\varphi_2(\mu,\Sigma,\xi_i)
     -\E[\varphi_2(\mu,\Sigma,\xi)]
    =
    -\left\{
        \frac1n\sum_{i=1}^n
        \frac{(\xi_i-\mu)(\xi_i-\mu)^\top}
             {(\xi_i-\mu)^\top \Sigma^{-1}(\xi_i-\mu)}
        -\E\left[
            \frac{(\xi-\mu)(\xi-\mu)^\top}
                 {(\xi-\mu)^\top \Sigma^{-1}(\xi-\mu)}
        \right]
    \right\}.
\end{align*}
Thus, Lemma~\ref{lemma:h-ULLN} implies that
\begin{align}\label{ineq:b3-regularized-uniform}
&\sup_{\substack{\|\mu\|\leq M,\\
                     \tr(\Sigma)=p,\ \Sigma\in\PD}} \left\|
        \frac1n\sum_{i=1}^n
        \frac{(\xi_i-\mu)(\xi_i-\mu)^\top}
             {(\xi_i-\mu)^\top(\Sigma^{(c)})^{-1}(\xi_i-\mu)}
        -\E\left[
            \frac{(\xi-\mu)(\xi-\mu)^\top}
                 {(\xi-\mu)^\top(\Sigma^{(c)})^{-1}(\xi-\mu)}
        \right]
    \right\|_{\F}\nonumber\\
    \leq &\sup_{\substack{\|\mu\|\leq M,\ \Sigma\succeq cI_p/(1+c),\\
                     \tr(\Sigma)=p,\ \Sigma\in\PD}}
    \left\|
        \frac1n\sum_{i=1}^n
        \frac{(\xi_i-\mu)(\xi_i-\mu)^\top}
             {(\xi_i-\mu)^\top \Sigma^{-1}(\xi_i-\mu)}
        -\E\left[
            \frac{(\xi-\mu)(\xi-\mu)^\top}
                 {(\xi-\mu)^\top \Sigma^{-1}(\xi-\mu)}
        \right]
    \right\|_{\F}
    \asto0.
\end{align}

Then combining \eqref{ineq:b3-3} and the upper bounds \eqref{ineq:b3-sigma-c-expectation}, 
\eqref{ineq:b3-regularized-uniform}, we conclude that
\begin{align*}
    &\limsup_{n\to\infty}
    \sup_{\substack{\|\mu\|\leq M,\ \lambda_{\min}(\Sigma)\leq c,\\
                    \tr(\Sigma)=p,\ \Sigma\in\PD}}
    \left\|
        \frac1n\sum_{i=1}^n
        \frac{(\xi_i-\mu)(\xi_i-\mu)^\top}
             {(\xi_i-\mu)^\top\Sigma^{-1}(\xi_i-\mu)}
    \right\|_{\F}\leq\frac{1+c}{4p}<\frac1{2p}
    \qquad\text{almost surely},
\end{align*}
where the last inequality follows from $c<1$.
This completes the proof of \eqref{ineq:b3}.
\end{proof}

\subsection{Proof of Proposition \ref{prop:convexity}}

\begin{proof}[Proof of Proposition \ref{prop:convexity}]
    Since $\rho$ is defined over $L^q$ and    $\E_{(\mu,\Sigma)\sim \Pi_n}[|F(x,\mu,\Sigma)|^q]<\infty$, then ${v}(x)$ is finite valued. The convexity follows from convexity of $F$ in $x$ for each $(\mu,\Sigma)$, and convexity and monotonicity of $\rho$. To see Lipschitz continuity, 
    we note that ${v}$ is convex over $\R^n$, therefore it is locally Lipschitz continuous 
    everywhere and hence globally Lipschitz continuous over any compact set of its domain. 
\end{proof}

\subsection{Proof of Theorem \ref{thm:unif-conv}}

\begin{proof}[Proof of Theorem \ref{thm:unif-conv}]
Let
\[
    \mathscr{P}_q(\R)\Let \left\{\mu\in\mathscr{P}(\R):\int_{\R} |z|^q\mu(\diff z)<\infty\right\}
\]
be the set of all probability measures on $\R$ with finite $q$-th order moment.
Since $\rho$ is law-invariant, it induces a functional
$\mathcal R:\mathscr{P}_q(\R)\to\R$ by
\[
    \mathcal R(\mu)\Let \rho(X),
\]
where $X\in L^q$ is any random variable whose law is $\mu$, and the domain of $\rho$ is given by
\[
    \operatorname{dom}(\mathcal{R}) = \left\{\mu\in\mathscr{P}_{q}(\R):X\in L^q, X\sim \mu, \rho(X) < \infty\right\}.
\]
Let the $q$-Wasserstein distance be defined as
\[
    W_q(\mu,\nu) \Let \inf_{\gamma\in \Gamma(\mu,\nu)} \left(\int_{\R\times \R}|x-y|^q \gamma(\diff x,\diff y)\right)^{1/q} = \inf_{X\sim\mu,Y\sim\nu}
    \left(\mathbb E |X-Y|^q\right)^{1/q} ,\quad \mu,\nu\in \mathscr{P}_q(\R).
\]
In the following, we first show that functional $\mathcal{R}:\mathscr{P}_q(\R)\to\R$ is Lipschitz continuous with respect to the $q$-Wasserstein distance on its domain. Let $\mu,\nu\in \operatorname{dom}(\mathcal{R})$, and let $X\sim \mu,Y\sim\nu$. Then, by the $L^q$-Lipschitz continuity of $\rho$,
\[
    |\mathcal R(\mu)-\mathcal R(\nu)|
    =
    |\rho(X)-\rho(Y)|
    \le
    \kappa_1 \|X-Y\|_{L^q}
    =
    \kappa_1
    \left(\mathbb E |X-Y|^q\right)^{1/q}.
\]
Since the above inequality holds for every possible $X$ and $Y$, by taking infimum over $X\sim \mu,Y\sim \nu$, it holds that
\[
    |\mathcal R(\mu)-\mathcal R(\nu)|
    \le
    \kappa_1
    \inf_{X\sim\mu,Y\sim\nu}
    \left(\mathbb E |X-Y|^q\right)^{1/q} = \kappa_1 W_q(\mu,\nu).
\]
Given $x\in \mathcal{X}$, let $P_x$ be the law of $g(x,\xi)$ induced by the distribution of $\xi$, and $\widehat{P}_{N,x}$ be the uniform distribution over $\{g(x,\xi_i)\}_{i=1}^N$. Then for every $x\in\mathcal X$,
\[
    |\widehat{v}_N(x)-{v}(x)|
    =
    |\mathcal R(\widehat P_{N,x})-\mathcal R(P_x)|
    \le
    \kappa_1 W_q(\widehat P_{N,x},P_x).
\]
Taking the supremum over $x\in\mathcal X$ on both sides, we obtain
\[
    \sup_{x\in\mathcal X}
    |\widehat{v}_N(x)-{v}(x)|
    \le
    \kappa_1
    \sup_{x\in\mathcal X}
    W_q(\widehat P_{N,x},P_x).
\]
Then to show the result, it remains to show that
\[
    \sup_{x\in\mathcal X}
    W_q(\widehat P_{N,x},P_x)
    \to 0
    \qquad \text{a.s.}
\]
 as $ N\to\infty$.
To show this, let $\varepsilon >0$ be a constant. Since $\mathcal X$ is compact, there exists a finite $\varepsilon$-net $\{x_1,\dots,x_M\}\subset \mathcal X$ such that for every
$x\in\mathcal X$, there is some $x_j$ satisfying
$\|x-x_j\|\le \varepsilon$. For such $x_j$, the triangle inequality implies that
\begin{align}
    \label{ineq:Wq-3}
    W_q(\widehat P_{N,x},P_x)
    \le
    W_q(\widehat P_{N,x},\widehat P_{N,x_j})
    +
    W_q(\widehat P_{N,x_j},P_{x_j})
    +
    W_q(P_{x_j},P_x).
\end{align}
Now we analyze the right hand side of \eqref{ineq:Wq-3} term by term. For the third term, since $g(x,\xi) \sim P_x$ and $g(x_j,\xi) \sim P_{x_j}$, by definition of $q$-Wasserstein distance, we have
\begin{align}\label{ineq:Wq-term3}
     W_q(P_x,P_{x_j}) \le\; &\left(
        \E [|g(x,\xi)-g(x_j,\xi)|^q] \right)^{1/q}
    \le \left(\E [\kappa_2(\xi)^q\|x-x_j\|^q] \right)^{1/q} = 
    \|x-x_j\|\|\kappa_2(\xi)\|_{L^q}\nonumber\\
    \leq \;&\varepsilon\|\kappa_2(\xi)\|_{L^q},
\end{align}
where the second inequality is by the Lipschitz continuity of $g$.
Likewise, for the first term, construct the coupling
\[
    \widehat\gamma_N \Let \frac{1}{N}\sum_{i=1}^N \delta_{(g(x,\xi_i),g(x_j,\xi_i))} \in \Gamma(\widehat{P}_{N,x}, \widehat{P}_{N,x_j}).
\]
Then we have
\begin{align}\label{ineq:Wq-term1}
    W_q(\widehat P_{N,x},\widehat P_{N,x_j})
    \le\; &\left(\int_{\R\times \R} |s-t|^q \widehat\gamma_N(\diff s,\diff t)\right)^{1/q} = 
    \left(
        \frac1N
        \sum_{i=1}^N
        |g(x,\xi_i)-g(x_j,\xi_i)|^q
    \right)^{1/q}\nonumber\\
    \le\; &\|x-x_j\|
    \left(\frac1N \sum_{i=1}^N \kappa_2(\xi_i)^q
    \right)^{1/q} \le\; \varepsilon \left(\frac1N \sum_{i=1}^N \kappa_2(\xi_i)^q \right)^{1/q}.
\end{align}
Combining \eqref{ineq:Wq-3}, \eqref{ineq:Wq-term3}, and \eqref{ineq:Wq-term1} and taking supremum on both sides of \eqref{ineq:Wq-3}, we obtain
\begin{align}\label{ineq:sup-Wq}
    \sup_{x\in\mathcal X}
    W_q(\widehat P_{N,x},P_x)
    \le\;&
    \max_{1\le j\le M}
    W_q(\widehat P_{N,x_j},P_{x_j}) +\varepsilon
    \left[
        \left(
            \frac1N
            \sum_{i=1}^N
            \kappa_2(\xi_i)^q
        \right)^{1/q}
        +
        \|\kappa_2(\xi)\|_{L^q}
    \right].
\end{align}
For each fixed $x_j$, since $g(x_j,\xi)\in L^q$, the empirical measure
$\widehat P_{N,x_j}$ converges to $P_{x_j}$ in $W_q$ almost surely. Hence, since the net is finite,
\[
    \max_{1\le j\le M}
    W_q(\widehat P_{N,x_j},P_{x_j})
    \to 0
    \qquad \text{a.s. as}\; N\to\infty.
\]
Moreover, by the strong law of large numbers,
\[
    \left(
        \frac1N
        \sum_{i=1}^N
        \kappa_2(\xi_i)^q
    \right)^{1/q}
    \to
    \|\kappa_2\|_{L^q}
    \qquad \text{a.s.}
\]
Consequently, \eqref{ineq:sup-Wq} implies that
\[
    \limsup_{N\to\infty}
    \sup_{x\in\mathcal X}
    W_q(\widehat P_{N,x},P_x)
    \le
    2\varepsilon \|\kappa_2\|_{L^q}
    \qquad \text{a.s.}
\]
Now let $\varepsilon\to 0$, we obtain
\[
    \limsup_{N\to\infty}
    \sup_{x\in\mathcal X}
    W_q(\widehat P_{N,x},P_x)= 0
    \qquad \text{a.s.}
\]
This completes the proof.
\end{proof}

\subsection{Proof of Theorem \ref{thm:cutting-convergence}}

\begin{proof}[Proof of Theorem \ref{thm:cutting-convergence}]
The proof adapts the convergence analysis of Kelley's method
(see, e.g., \citealp{ref:kelley1960cutting}) to the master problem \eqref{eq:master}. We first prove two preliminary facts: a Lipschitz estimate for $\rho$, and a uniform bound on the cut slopes $g_{i,j}$.

For $a = (a_1,\ldots,a_N) \in \R^N$, write $\rho(\{a_i\}_{i=1}^N)$ for the
risk of the uniform distribution over $\{a_1,\ldots,a_N\}$ as in
\eqref{eq:saa}. We claim that
\begin{equation}\label{eq:rho-lip}
    \left|\rho\big(\{a_i\}_{i=1}^N\big) - \rho\big(\{b_i\}_{i=1}^N\big)\right|
    \;\le\; \max_{1\le i\le N} |a_i - b_i|,
    \qquad \forall\, a, b \in \R^N.
\end{equation}
To see this, let $c \Let \max_{1\le i\le N}|a_i - b_i|$. Then by the monotonicity and translation invariance of $\rho$,
\[
    \rho(\{a_i\}_{i=1}^N) \;\le\; \rho(\{b_i\}_{i=1}^N + c) \;=\; \rho(\{b_i\}_{i=1}^N) + c.
\]
Exchanging the roles of $a$ and $b$, we have \eqref{eq:rho-lip}.

Next we bound the cut slopes. Fix $i \in \{1,\ldots,N\}$. The function $F(\cdot,\mu_i,\Sigma_i)$ is convex and finite on $\R^d$. Hence $\partial_x F(x,\mu_i,\Sigma_i)$ is
nonempty for every $x \in \mathcal{X}$, and is bounded uniformly over the
compact set $\mathcal{X}$
\citep[Theorem~23.4]{rockafellar1997convex}. Define
\begin{align}
    \label{def:G}
    G \;\Let\; \max_{1\le i\le N}\; \sup_{x\in\mathcal{X}}\;
    \sup\big\{\|g\| : g \in \partial_x F(x,\mu_i,\Sigma_i)\big\} \;<\; \infty.
\end{align}
Then $\|g_{i,j}\| \le G$ for all $i$ and $j$. Moreover, for any
$x, y \in \mathcal{X}$ and any $g \in \partial_x F(x,\mu_i,\Sigma_i)$, the
subgradient inequality gives
$F(y,\mu_i,\Sigma_i) \ge F(x,\mu_i,\Sigma_i) + \la g, y - x\ra
\ge F(x,\mu_i,\Sigma_i) - G\|y - x\|$. Exchanging $x$ and $y$ shows that
$F(\cdot,\mu_i,\Sigma_i)$ is $G$-Lipschitz continuous on $\mathcal{X}$.
Combining this with \eqref{eq:rho-lip}, $\hat{v}_N$ is $G$-Lipschitz continuous on $\mathcal{X}$.

\underline{Step 1: proof of part (i).} For $k \ge 0$, 
let 
\[
    \hat{F}_i^k(x) \Let \max_{j=0,1,\ldots,k}
    \big\{F_{i,j} + \la g_{i,j}, x - x_j\ra\big\}\quad\text{for}~i=1,\ldots,N,
\]
and define
\begin{equation}\label{eq:model-fn}
    \underline{v}_k(x) \;\Let\;
    \rho\Big(\left\{\hat{F}_i^k(x)\right\}_{i=1}^N\Big).
\end{equation}
Recall that $\mathrm{LB}_{k+1}$ is the optimal value of \eqref{eq:master}.
We claim that
\begin{equation}\label{eq:master-equiv}
    \mathrm{LB}_{k+1} \;=\; \underline{v}_k(x_{k+1})
    \;=\; \min_{x\in\mathcal{X}} \underline{v}_k(x).
\end{equation}
To see this, let $x \in \mathcal{X}$ be fixed. By construction, a vector $\hat{F} \in \R^N$ is feasible for \eqref{eq:master} at this $x$ if and only if $\hat{F}_i \ge \hat{F}_i^k(x)$ for $i=1,\ldots,N$. Moreover, the objective of \eqref{eq:master} is nondecreasing in $\hat{F}$ by the monotonicity of $\rho$. Then by fixing $x$, the partial minimization of \eqref{eq:master} over $\hat{F}$ has value $\underline{v}_k(x)$, and hence \eqref{eq:master-equiv} follows. The minimum on the right-hand side of \eqref{eq:master-equiv} is attained since $\underline{v}_k$ is continuous in $x$, and $\mathcal{X}$ is compact.

Since $g_{i,j} \in \partial_x F(x_j,\mu_i,\Sigma_i)$, the subgradient inequality gives that 
\[
    F(x,\mu_i,\Sigma_i) \ge F_{i,j} + \la g_{i,j}, x - x_j\ra
\]
for all $x \in \mathcal{X}$ and all $i, j$. Hence
$F(x,\mu_i,\Sigma_i) \ge \hat{F}_i^k(x)$ for $i=1,\ldots,N$, and the monotonicity of $\rho$ gives
\begin{equation}\label{eq:model-under}
    \hat{v}_N(x) \;\ge\; \underline{v}_k(x), \qquad \forall\, x\in\mathcal{X}.
\end{equation}
Minimizing both sides of \eqref{eq:model-under} over $\mathcal{X}$ and using \eqref{eq:master-equiv}, we have $\mathrm{LB}_{k+1} \le \hat{v}_N^*$. Also, $\mathrm{UB}_{k+1} = \min_{j=0,1,\ldots,k}\hat{v}_N(x_j) \ge \hat{v}_N^*$. This
shows the sandwich inequality in part~(i). The sequence $\{\mathrm{UB}_k\}$ is nonincreasing because it is a running minimum. The sequence $\{\mathrm{LB}_k\}$ is nondecreasing because
$\underline{v}_{k+1} \ge \underline{v}_k$ pointwise, and hence $\min_{x\in\mathcal{X}} \underline{v}_{k+1}(x) \geq \min_{x\in\mathcal{X}} \underline{v}_{k}(x) $.

\underline{Step 2: the gap estimate.} We claim that
\begin{equation}\label{eq:gap-estimate}
    \mathrm{UB}_{k+1} - \mathrm{LB}_{k+1} \;\le\; G\,\|x_{k+1} - x_j\|,
    \qquad \forall\, j = 0,1,\ldots,k.
\end{equation}
Let $j \in \{0,\ldots,k\}$ be fixed. For every $i$,
\[
    \hat{F}_i^k(x_{k+1})
    \;\ge\; F_{i,j} + \la g_{i,j}, x_{k+1} - x_j\ra
    \;\ge\; F(x_j,\mu_i,\Sigma_i) - G\|x_{k+1} - x_j\|,
\]
where the first inequality is by the definition of $\hat{F}_i^k$, and the second inequality is by
$F_{i,j} = F(x_j,\mu_i,\Sigma_i)$ and the inequality $\left|\la g_{i,j}, x_{k+1} - x_j\ra\right| \leq G\|x_{k+1}-x_j\|$. The monotonicity and translation invariance of $\rho$ then give
\begin{align}\label{ineq:LB-v}
    \mathrm{LB}_{k+1} \;=\; \underline{v}_k(x_{k+1})
    \;\ge\; \rho\Big(\big\{F(x_j,\mu_i,\Sigma_i)\big\}_{i=1}^N\Big)
        - G\|x_{k+1} - x_j\|
    \;=\; \hat{v}_N(x_j) - G\|x_{k+1} - x_j\|.
\end{align}
On the other hand, by the definition of $\mathrm{UB}_{k+1}$ we have $\mathrm{UB}_{k+1} \leq \hat{v}_N(x_j)$, and \eqref{eq:gap-estimate} follows.

\underline{Step 3: proof of parts (ii) and (iii).} Let
$\Delta_k \Let \mathrm{UB}_k - \mathrm{LB}_k$. By part~(i), $\Delta_k \ge 0$
and $\{\Delta_k\}$ is nonincreasing. Since $\mathcal{X}$ is compact, the
sequence $\{x_k\}$ admits a convergent subsequence $\{x_{k_l}\}$. A
convergent sequence is a Cauchy sequence, so
$\|x_{k_{l+1}} - x_{k_l}\| \to 0$ as $l \to \infty$. Applying
\eqref{eq:gap-estimate}, we have $\Delta_{k_{l+1}} \le G\|x_{k_{l+1}} - x_{k_l}\| \to 0$. Hence
$\Delta_k \to 0$. Combining $\Delta_k \to 0$ with (i), we have $\mathrm{UB}_k \downarrow \hat{v}_N^*$ and $\mathrm{LB}_k \uparrow \hat{v}_N^*$.

Now let $\bar{x}$ be an accumulation point of $\{x_k\}$, and let
$\{x_{m_l}\}$ be a subsequence with $x_{m_l} \to \bar{x}$. By \eqref{ineq:LB-v} and the fact that $\hat{v}_N(x_{m_l}) \ge \mathrm{UB}_{m_{l+1}}$, we obtain
\[
    \hat{v}_N^* \;\le\; \hat{v}_N(x_{m_l})
    \;\le\; \mathrm{LB}_{m_{l+1}} + G\|x_{m_{l+1}} - x_{m_l}\|
    \;\to\; \hat{v}_N^*,
\]
where the convergence uses $\mathrm{LB}_k \uparrow \hat{v}_N^*$ and
$\|x_{m_{l+1}} - x_{m_l}\| \to 0$. Hence
$\hat{v}_N(x_{m_l}) \to \hat{v}_N^*$. By the continuity of $\hat{v}_N$ on
$\mathcal{X}$, $\hat{v}_N(\bar{x}) = \hat{v}_N^*$. This proves part~(ii).

Finally, let $\varepsilon > 0$. Since $\Delta_k \to 0$, there exists a finite
$K$ with $\Delta_K \le \varepsilon$, and the algorithm terminates at
iteration $K$. The returned point satisfies
$\hat{v}_N(x^*) = \mathrm{UB}_K \le \mathrm{LB}_K + \varepsilon
\le \hat{v}_N^* + \varepsilon$, where the last inequality is by part~(i).
This proves part~(iii).
\end{proof}

\end{document}